\documentclass[sts,
reqno
]{imsart}

\usepackage{amsmath, amssymb, color}
\usepackage{amsthm} 
\usepackage[colorlinks,citecolor=blue,urlcolor=blue]{hyperref}
\usepackage{xcolor,colortbl}
\usepackage{color}
\usepackage{amsfonts, fancybox}
\usepackage{amssymb}
\usepackage[american]{babel}
\usepackage{enumitem}
\usepackage{psfrag,color}
\usepackage{dcolumn}
\usepackage{caption}
\usepackage[labelformat=simple]{subcaption}

\usepackage{flafter, bm, dsfont}
\usepackage[section]{placeins}

\usepackage{multirow}
\usepackage{color}
\usepackage{amsmath}
\usepackage{mathrsfs}
\usepackage{amsfonts}
\usepackage{dsfont}
\usepackage{amssymb}
\usepackage{algorithmicx, algorithm, algpseudocode}
\usepackage{graphicx}

\usepackage[noadjust]{cite}
\usepackage{tikz}
\usetikzlibrary{positioning}

\newcommand{\bc}{\begin{center}}
\newcommand{\ec}{\end{center}}
\newcommand{\ba}{\begin{array}}
\newcommand{\ea}{\end{array}}
\newcommand{\be}{\begin{eqnarray}}
\newcommand{\ee}{\end{eqnarray}}
\newcommand{\bel}{\begin{eqnarray}\label}
\newcommand{\eel}{\end{eqnarray}}
\newcommand{\bes}{\begin{eqnarray*}}
\newcommand{\ees}{\end{eqnarray*}}
\newcommand{\bn}{\begin{enumerate}}
\newcommand{\en}{\end{enumerate}}

\definecolor{MIT}{cmyk}{.24, 1.00, .78, .17} 
\definecolor{pink}{cmyk}{0, 1, 0, 0} 
\definecolor{darkgreen}{cmyk}{1,0, 1, 0}

\newcommand{\iid}{{\it i.i.d.\ }}

\newtheorem{theorem}{Theorem}
\newtheorem*{theorem*}{Theorem}
\newtheorem{lemma}[theorem]{Lemma}

\newtheorem{proposition}[theorem]{Proposition}
\newtheorem*{proposition*}{Proposition}

\newcommand{\tr}{\mathop{\mathsf{Tr}}}

\newcommand{\simiid}{\overset{\text\small\rm{iid}}{\sim}}

\usepackage[utf8]{inputenc}
\newcommand{\leadeq}[2][4]{\MoveEqLeft[#1] #2 \nonumber}

\usepackage{stmaryrd}
\usepackage{bbm}
\usepackage[abbreviations]{foreign}

\newcommand{\mockalph}[1]{}

\usepackage{mathtools}
\mathtoolsset{showonlyrefs}
\makeatletter
\newcommand{\subalign}[1]{%
\vcenter{%
\Let@ \restore@math@cr \default@tag
\baselineskip\fontdimen10 \scriptfont\tw@
\advance\baselineskip\fontdimen12 \scriptfont\tw@
\lineskip\thr@@\fontdimen8 \scriptfont\thr@@
\lineskiplimit\lineskip
\ialign{\hfil$\m@th\scriptstyle##$&$\m@th\scriptstyle{}##$\crcr
  #1\crcr
}%
}
}
\makeatother
\algnewcommand\algorithmicinput{\textbf{Input:}}
\algnewcommand\Input{\item[\algorithmicinput]}
\algnewcommand\algorithmicoutput{\textbf{Output:}}
\algnewcommand\Output{\item[\algorithmicoutput]}
\newcommand{\cC}{\mathcal{C}}
\newcommand{\cD}{\mathcal{D}}

\newcommand{\cF}{\mathcal{F}}

\newcommand{\cM}{\mathcal{M}}
\newcommand{\cN}{\mathcal{N}}

\newcommand{\cS}{\mathcal{S}}

\newcommand{\bone}{\mathbf{1}}

\newcommand{\R}{\mathbbm{R}}
\newcommand{\p}{\mathbbm{P}}
\newcommand{\E}{\mathbbm{E}}

\newcommand{\1}{\mathbbm{1}}

\newcommand{\pn}{\p_{\kern-0.25em n}}
\newcommand{\pnm}{\p_{\kern-0.25em n,m}}
\newcommand{\psubm}{\p_{\kern-0.25em m}}
\newcommand{\psubp}{\p_{\kern-0.25em p}}
\newcommand{\cfi}{\cF_{\kern-0.25em \infty}}

\newcommand{\Var}{\mathrm{Var}}

\newcommand{\argmin}{\mathop{\mathrm{argmin}}}
\newcommand{\argmax}{\mathop{\mathrm{argmax}}}

\newcommand{\eps}{\varepsilon}

\newlength{\minipagewidth}
\newcommand{\dimone}{n_1}
\newcommand{\dimtwo}{n_2}
\newcommand{\Done}{D}
\newcommand{\Dtwo}{\tilde{D}}
\newcommand{\AOone}{\bone}
\newcommand{\AOtwo}{\tilde{\bone}}

\newcommand{\subG}{\mathsf{subG}}
\newcommand{\defn}{:=}
\newcommand{\id}{\mathsf{id}}
\newcommand{\pir}{\pi_1^{\mathsf{r}}}
\newcommand{\pirtwo}{\pi_2^{\mathsf{r}}}
\newcommand{\diam}{\mathsf{diam}}
\newcommand{\thetals}{\hat \theta^{\sf ls}}
\newcommand{\thetagls}{\hat \theta^{\sf gls}}
\newcommand{\thetasvt}{\hat \theta^{\sf svt}}
\newcommand{\thetasoft}{\hat \theta^{\sf soft}}

\newcommand{\gaussiannoise}{Z}

\begin{document}

% \maketitle

\begin{frontmatter}

\title{Estimation of Monge Matrices}
\runtitle{Estimation of Monge Matrices}

\author{Jan-Christian H\"utter, Cheng Mao, Philippe Rigollet and Elina Robeva}
\runauthor{H\"utter, Mao, Rigollet and Robeva}
\affiliation{Massachusetts Institute of Technology and Yale University}
%f\re
%	\thankstext{t2}{}

% \address{{Philippe Rigollet}\\
% 	{Department of Mathematics} \\
% 	{Massachusetts Institute of Technology}\\
% 	{77 Massachusetts Avenue,}\\
% 	{Cambridge, MA 02139-4307, USA}\\
% 	\printead{rigollet}
% }

%
%
%
%
%
%

\begin{abstract}
Monge matrices and their permuted versions known as pre-Monge matrices naturally appear in many domains across science and engineering. While the rich structural properties of such matrices have long been leveraged for algorithmic purposes, little is known about their impact on statistical estimation. In this work, we propose to view this structure as a shape constraint and study the problem of estimating a Monge matrix subject to additive random noise. More specifically, we establish the minimax rates of estimation of Monge and pre-Monge matrices. In the case of pre-Monge matrices, the minimax-optimal least-squares estimator is not efficiently computable, and we propose two efficient estimators and establish their rates of convergence. Our theoretical findings are supported by numerical experiments. 
%Monge matrices and their permuted versions appear in various domains. Although their properties and roles in combinatorial optimization have long been studied, there is little understanding of Monge matrices in the presence of random noise. In this work, we study estimation of a Monge matrix, whose rows and columns are potentially shuffled by latent permutations, under sub-Gaussian noise. The least square estimator is shown to achieve the minimax rates of estimation. In the permuted setup, we study two efficient estimators and their rates of convergence. 
\end{abstract}

%	\begin{keyword}[class=AMS]
%		\kwd[Primary ]{}
%		%
%		\kwd[; secondary ]{}
%	\end{keyword}
%	\begin{keyword}[class=KWD]
%	\end{keyword}

\end{frontmatter}

\section{Introduction}

A matrix $\theta \in \R^{\dimone \times \dimtwo}$ is called a \emph{Monge matrix}~\cite{Hof63} or a \emph{submodular matrix}~\cite{QueSpiTar98}, if 
\begin{equation}
\label{EQ:defMonge}
\theta_{i, j} + \theta_{k, \ell} \le \theta_{i, \ell} + \theta_{k, j} , \quad \text{ for all } 1 \le i \le k \le \dimone, \, 1 \le j \le \ell \le \dimtwo .
\end{equation}
In addition, a matrix $\theta \in \R^{\dimone \times \dimtwo}$ is called an \emph{anti-Monge matrix} or a \emph{supermodular matrix} if $-\theta$ is a Monge matrix. 
The Monge property dates back to Gaspard Monge's work on optimal transport~\cite{Mon81}.
Since then, it has been widely used and studied in optimization, discrete mathematics and computer science~\cite{Hof63, Aggetal87, CecSza90, RudWoe95, BurKliRud96, Bur07} as it allows for simple and fast algorithms in a variety of instances~\cite{Hof63, BurKliRud96, PfeRudWoe94, Par91, Buretal98}.
%For example, if the cost matrix in the Hitchcock transportation problem~\cite{Hit41} is a Monge matrix, then the so called north-west corner rule produces an optimal solution~\cite{Hof63, BurKliRud96}.
%Other problems which become easier with a Monge cost matrix include, but are far from being limited to, the balanced max-cut problem~\cite{PfeRudWoe94} and the traveling salesman problem~\cite{Par91, Buretal98}. 

Many of these problems turn out to be invariant under relabeling of the rows and columns of the Monge matrix. Consequently, we 
%In many of these applications, e.g., the optimal transportation and the traveling salesman problem,  there is no underlying ordering of the rows and columns of the cost matrix that changes the nature of the problem. Therefore, properties of the cost matrix should be invariant under permutations acting on its rows and columns. It is therefore natural to 
introduce the following definition. A matrix $\theta \in \R^{\dimone \times \dimtwo}$ is called \emph{pre-Monge} if there exist permutations $\pi_1 : [\dimone] \to [\dimone]$ and $\pi_2 : [\dimtwo] \to [\dimtwo]$ such that the matrix $\theta( \pi_1, \pi_2 )$ defined by
$$ 
\theta( \pi_1, \pi_2 )_{i, j} = \theta ( \pi_1(i), \pi_2(j) ), \quad \text{ for all } (i, j) \in [\dimone] \times [\dimtwo],
$$
is Monge. Note that the terminology \emph{permuted Monge} has also been used to define the same object~\cite{BurKliRud96}.
A pre-anti-Monge matrix is defined analogously. 
Like Monge matrices, pre-Monge matrices have also been studied in the context of optimization~\cite{BurDeiWoe98, CelDeiWoe18} where the latent permutation yields new computational challenges. For example, even checking that a matrix is pre-Monge is a nontrivial algorithmic task~\cite{KliRudWoe95, DeiRudWoe96}.

However, previous work on pre-Monge matrices has focused on the noiseless setting and algorithms typically fail when the pre-Monge matrix is contaminated by random noise.
This motivates us to take a statistical approach and study estimation of a pre-Monge matrix under random noise.

\subsection{Geometric interpretation and spectral ordering.}

The Monge property has strong ties with geometric properties of certain datasets, starting with the seminal work of Monge on optimal transport~\cite{Mon81}. In this subsection, we demonstrate how the Monge property arises in the context of seriation~\cite{Ken69, Ken70, AtkBomHen98, FogAspVoj14,FogJenBac13, FlaMaoRig19} where the goal is to recover the latent ordering of objects based on pairwise distances or inner products.

%To illustrate the role of the Monge property in seriation, 
Let $X \in \R^{n \times d}$ be a data matrix with rows $x_1^\top, \dots, x_n^\top \in \R^d$. Moreover, suppose that for all $i,j \in [n-1]$, we have that $(x_{i+1} - x_{i})^\top (x_{j+1} - x_{j}) \ge 0$. In other words, the differences between consecutive points form an acute angle so that the points $x_1, \ldots, x_n$ may be ordered along a common direction. In this case,  it is easy to check that the Gram matrix $\theta = X X^\top$ is an anti-Monge matrix, and the distance matrix $D$, defined by $D_{i,j} = \|x_{i} - x_{j} \|_2^2$ for $i,j \in [n]$, is a Monge matrix. 
%See also \cite{BurKliRud96} for an example of an asymmetric distance matrix with the Monge property.

Suppose that we do not know the order of the points, or equivalently, we observe $x_{\pi(1)}, \dots, x_{\pi(n)}$, where there is an unknown permutation $\pi : [n] \to [n]$. 
How can we reorder the points in order to recover the above geometric structure (i.e. so that the differences between consecutive points form an acute angle)? 
Intuitively, assuming that such a reordering exists suggests that the $n$ points should approximately lie along a hidden direction. Therefore, we can apply principal component analysis as follows. 
Let us assume without loss of generality that the points are centered so that $\sum_{i=1}^n x_i = 0$ and thus $\sum_{i=1}^n \theta_{i, j} = \sum_{j=1}^n \theta_{i, j} = 0$. 
Then the leading right singular vector of $X$ gives the hidden direction, and the leading left singular vector $v$ of $X$, i.e., the leading eigenvector of the Gram matrix $\theta$ is the first principal component of the data points. The entries of $v$ then give a one-dimensional embedding of the points, from which we easily recover the original order. 

Indeed, this intuition can be made rigorous using Corollary~2.6 of~\cite{Fie01}, which is a variant of the Perron-Frobenius theorem and states that the leading eigenvector of a doubly centered anti-Monge matrix (i.e. having row and column sums equal to zero) is monotone. 
Hence the leading eigenvector $v$ of the Gram matrix $\theta$ is monotone. 
If the unknown permutation $\pi$ relabels the points, then the leading eigenvector of the Gram matrix becomes $v_\pi$, defined by $(v_{\pi})_i = v_{\pi(i)}$. 
As a result, sorting the entries of $v_\pi$ recovers the permutation $\pi$ and, therefore, the latent order of the points. 
The above method for spectral ordering is similar to the one for seriation proposed in \cite{AtkBomHen98}. 

%The task of seriation consists in ordering items in a series so that nearby items exhibit higher similarities. 
%In short, the aforementioned work proves that sorting the entries of the Fiedler vector of a similarity matrix recovers the latent permutation of interest. 
%This is also equivalent to sorting the entries of the leading eigenvector of the Laplacian of a dissimilarity matrix. 
%It is beyond the scope of this work to provide a unified view of spectral ordering for various data structures, so we leave this task to future work. 
%
%While the method of spectral ordering has observed success in the noiseless setup, it only applies to a square (anti-)Monge matrix, and only exploits information from the leading eigenvector. 
%Therefore, it is not clear whether spectral ordering is robust under random noise when low-rank structure is missing. 
%Instead, we study algorithms in Section~\ref{sec:per} that are provably robust. 

\subsection{Our contribution}

In this work, we study the estimation of pre-(anti-)Monge matrices under additive sub-Gaussian noise. 
Statistically, we establish the minimax rates of estimation (up to logarithmic factors) for both Monge and pre-Monge matrices in Sections~\ref{sec:monge} and~\ref{sec:minimax} respectively, where the upper bounds are achieved by the least-squares estimators. 

Algorithmically, for estimating pre-Monge matrices, we further introduce two efficient estimators and study their rates of convergence.
The {\sf Variance Sorting} estimator introduced in Section~\ref{sec:v-sort}, as the name suggests, employs second-order information to estimate the latent permutation.
In Section~\ref{sec:svt}, we study the singular value thresholding estimator based on our result (Proposition~\ref{prop:low-rank-approx}) on the approximation of pre-Monge matrices by low-rank ones. 

Furthermore, we provide various numerical experiments in Section~\ref{sec:numerics} to corroborate the theoretically established rates of estimation. Using Dykstra's projection algorithm, we give a detailed implementation of the least-squares estimator for (anti-)Monge matrices, which is of practical interest. 
 
The proofs of all theorems and auxiliary lemmas can be found in Section~\ref{sec:proofs}.

\subsection{Related work}

This work connects to several lines of research that are described below.

\subsubsection*{Total positivity.}
The Monge property is closely related to the notion of \emph{total positivity}~\cite{KarRin80a}. An entrywise positive matrix $\theta \in \R^{\dimone \times \dimtwo}$ is called totally positive (of order $2$), if 
$$
\theta_{i, j}  \theta_{k, \ell} \ge \theta_{i, \ell}  \theta_{k, j} , \quad \text{ for all } 1 \le i \le k \le \dimone, \, 1 \le j \le \ell \le \dimtwo .
$$
Therefore, an entrywise positive matrix $\theta$ is totally positive if and only if $\log (\theta)$ is anti-Monge, where $\log (\cdot)$ is applied to each entry of $\theta$ individually. 
As a result, total positivity is also known as  \emph{log-supermodularity}. 
Total positivity plays an essential role in statistical physics via the FKG inequality~\cite{ForKasGin71} and appears frequently in many other areas of probability and statistics~\cite{KarRin80a, KarRin80b}.
More recently, there have been new developments in studying totally positive distributions and related estimation problems~\cite{Faletal17, LauUhlZwi17, RobStuTraUhl18}. In a companion paper~\cite{HutMaoRigRob19b}, we study minimax estimation of a totally positive distribution by employing mathematical tools that are closely related to those in the current paper. 

\subsubsection*{Latent permutation learning.}
Estimating a pre-Monge matrix from its noisy version falls into the category of matrix learning with latent permutations, which has recently observed a surge of interest. Models involving latent permutations include noisy sorting~\cite{BraMos08}, the strong stochastic transitivity model~\cite{Cha15, ShaBalGunWai17}, feature matching~\cite{ColDal16}, crowd labeling~\cite{ShaBalWai16a}, statistical seriation~\cite{FlaMaoRig19} and graph matching~\cite{LivRiz13, DinMaWuXu18}, to name a few. 
Many of the previous approaches for learning latent permutations under such models are based on sorting row or column sums of the observed matrix (or equivalently, degrees of vertices)~\cite{ShaBalWai16b, ChaMuk19, PanMaoMut17} or certain refinements~\cite{MaoWeeRig18, MaoPanWai18a}. However, since adding a constant to all entries in a row or column of a Monge matrix does not change its Monge property, first-order information such as row sums is uninformative for the Monge structure, and thus cannot be used to identify the latent permutation. Instead, we propose a new algorithm based on variance sorting. We show in Section~\ref{sec:v-sort} that this novel use of second order information is decisive when estimating pre-Monge matrices.

\subsubsection*{Graphon estimation.}
Another related, substantial body of literature is that on graphon estimation~\cite{WolOlh13, ChaAir14, GaoLuZho15, BorChaCoh15}, where the goal is to estimate a bivariate function $f : [0, 1]^2 \to \R$ from samples $\{ f (X_i, Y_j) : 1\le i \le \dimone, \, 1 \le j \le \dimtwo \}$. 
Unlike regression, the design points $(X_i, Y_j)$ are not observed in graphon estimation, so the observations can be viewed as an $\dimone \times \dimtwo$ matrix with latent permutations acting on its rows and columns. 
There have been extensive studies on graphon estimation with various structures, including block models~\cite{AirCosCha13}, smoothness~\cite{KloTsyVer17} and low-rank structure~\cite{ShaLee18}. While our setup is not about recovering an underlying function $f$, the current work can be viewed as a study of denoising observations in graphon estimation with the Monge structure.

\subsubsection*{Shape-constrained estimation.}
Estimation of a Monge matrix, which we study in Section~\ref{sec:monge}, falls in the scope of shape-constrained estimation. Closest to the present work is the estimation of a bivariate isotonic matrix under Gaussian noise~\cite{ChaGunSen18}. In fact, every anti-Monge matrix can be written as the sum of a rank-two matrix and a bivariate isotonic matrix (Lemma~\ref{lem:c}). However, our results suggest that the set of Monge matrices is in fact qualitatively different from the set of bivariate isotonic matrices. Particularly, the minimax rate of estimation in Theorem~\ref{thm:upper_bound} is different from that given by Theorem~2.1 of~\cite{ChaGunSen18}, and the low-rank approximation rate in Proposition~\ref{prop:low-rank-approx} is different from that given by Lemma~4 of~\cite{ShaBalGunWai17}. 

Shortly before completing the current work, we became aware of a concurrent work by Fang, Guntuboyina and Sen~\cite{FanGunSen19} that studies multivariate extensions of isotonic regression. The two-dimensional version almost coincides with the anti-Monge structure (without permutations) that we study, and the rate achieved by the least-squares estimator specialized to dimension two, as expected, coincides with the main term of the rate given by Theorem~\ref{thm:upper_bound} in our current paper. However, it is worth noting that the two proofs follow drastically different paths. 
While the proof in~\cite{FanGunSen19} relies on metric entropy estimates from~\cite{BleGaoLi07, Gao13}, our proof is based on spectral decomposition of the difference operator $D$ defined in~\eqref{eq:def-d}, a technique which has been used for example to study the performance of total variation regularization~\cite{WanShaSmo15a, HutRig16}.
Moreover, assuming \( n = \dimone = \dimtwo \), our upper bound given in Theorem~\ref{thm:upper_bound} contains a log factor of order \( \log(n) \), while the one in Theorem 4.1 of~\cite{FanGunSen19} potentially scales like \( \log(n)^3 \), a minor improvement which nonetheless shows the potential merits of our proof technique.

%\notecm{later section}

\medskip
\noindent
\textbf{Notation.}
For a positive integer $n$, let $[n] = \{1, 2, \ldots, n\}$. For a
finite set $S$, we use $|S|$ to denote its cardinality. For two
sequences $\{a_n\}_{n=1}^\infty$ and $\{b_n\}_{n=1}^\infty$ of real numbers, we write
$a_n \lesssim b_n$ if there is a universal constant $C$ such that $a_n
\leq C b_n$ for all $n \geq 1$. The relation $a_n \gtrsim b_n$ is
defined analogously.  We use $c$ and $C$ (possibly with subscripts) to denote
universal constants that may change from line to line. Let $\land$ and $\lor$ denote the min and the max operators between two real numbers respectively.
Given a matrix $M \in \R^{\dimone \times \dimtwo}$, we denote its $i$-th row by $M_{i, \cdot}$ and its $j$-th column by $M_{\cdot,j}$.
We denote by $\|M\|_F $ and $ \| M \|$ the Frobenius norm and the operator norm of \( M \), and by \( \| M \|_1 \) and \( \| M \|_\infty \) the \( \ell^1 \) and \( \ell^\infty \)-norm of $M$ when viewed as a vector in \( \R^{\dimone \dimtwo} \), respectively. 
We write \( M^{\dag} \) for the Moore-Penrose pseudoinverse of \( M \).
%For a close, convex set $\cM \subseteq \R^d$, let $\Pi_{\cM}(y)$ denote the projection of $y \in \R^d$ onto $\cM$.
Finally, let $\cS_n$ denote the set of permutations $\pi : [n] \to [n]$.

\section{Anti-Monge matrix estimation}
\label{sec:monge}

We start with estimation of a Monge matrix under sub-Gaussian noise, without latent permutations. 
It is mathematically equivalent to study estimation of an \emph{anti}-Monge matrix \( \theta^* \in \R^{\dimone \times \dimtwo} \), which we find more convenient for the presentation. 
Throughout this work, without loss of generality, we also assume that $\dimone \ge \dimtwo$. 
%To define the set of anti-Monge matrices, 

Consider the difference operator $\Done \in \R^{(\dimone-1)\times \dimone}$ defined by
\begin{equation}
\Done = \begin{bmatrix}
	-1 & 1 & 0 &  \dots & 0 & 0\\
	0 & -1 & 1 &  \dots & 0 & 0\\
	\vdots & \vdots & \vdots &  & \vdots & \vdots \\
	0 & 0 & 0 &  \dots & -1 & 1
\end{bmatrix} , \label{eq:def-d}
\end{equation}
and we define $\Dtwo \in \R^{(\dimtwo-1) \times \dimtwo}$ in the same way. 

Using a telescoping sum argument, it is easy to check that the set of anti-Monge matrices $\theta$ such that $-\theta$ satisfies~\eqref{EQ:defMonge} can be expressed as
\begin{align} \label{eq:Mdef}
\cM = \cM^{\dimone, \dimtwo} \defn \{\theta \in \R^{\dimone \times \dimtwo}: \Done \theta \Dtwo^\top \geq 0\} , 
\end{align}
where the symbol $\geq$ denotes entrywise inequality. 

%Note that for $\theta \in \cM$ and $1 \le i \le k \le \dimone, \, 1 \le j \le \ell \le \dimtwo$, it follows from a telescoping sum argument that
%\begin{align}
%\theta_{i, j} + \theta_{k, \ell} - \theta_{i, \ell} - \theta_{k, j} 
%= \sum_{a = i}^{k-1} \sum_{b = j}^{\ell - 1} (\theta_{a, b} + \theta_{a+1, b+1} - \theta_{a, b+1} - \theta_{a+1, b} )
%= \sum_{a = i}^{k-1}  \sum_{b = j}^{\ell - 1} ( \Done \theta \Dtwo^\top )_{a, b} \ge 0 ,  \quad 
%\label{eq:tele}
%\end{align}
%so definition~\eqref{eq:Mdef} is indeed equivalent to the usual definition of an anti-Monge matrix.
%
For each $\theta \in \cM$, we define the quantity
\begin{align} 
V(\theta) \defn \theta_{1,1} + \theta_{\dimone, \dimtwo} - \theta_{\dimone, 1} - \theta_{1, \dimtwo}=\| \Done \theta \Dtwo^\top \|_1 \,,
\label{eq:def-v}
\end{align}
where the last equality follows from a telescoping sum. 
We remark that $V(\theta)$ is a global seminorm of $\theta$, and turns out to play a role in the rate of estimation. 
%By equation~\eqref{eq:tele}, we see that 
%\begin{align} \label{eq:V-l1}
%V(\theta) 
%%= \sum_{i=1}^{\dimone-1} \sum_{j=1}^{\dimtwo-1} ( \theta_{i, j} + \theta_{i+1, j+1} - \theta_{i+1, j} - \theta_{i, j+1} ) 
%= \sum_{i=1}^{\dimone-1} \sum_{j=1}^{\dimtwo-1} (\Done \theta \Dtwo^\top)_{i, j} = \| \Done \theta \Dtwo^\top \|_1 ,
%\end{align}
%implying that $V(\theta)$ is indeed a seminorm.

In this work, we consider additive sub-Gaussian noise. Namely, for a zero-mean random matrix $\eps \in \R^{\dimone \times \dimtwo}$, we say that $\eps$ is sub-Gaussian with variance proxy $\sigma^2$, or simply $\eps \sim \subG_{\dimone \times \dimtwo} (\sigma^2)$, if for any matrix $M \in \R^{\dimone \times \dimtwo}$, it holds that
$$
\E[\exp(\tr(M^\top \eps)] \le \exp(\sigma^2 \|M\|_F^2/2) .
$$
Suppose that we observe
\begin{align}
\label{eq:GSmodel}
y = \theta^* + \varepsilon ,
\end{align}
where $\eps \sim \subG_{\dimone \times \dimtwo}(\sigma^2)$. We study the performance of the least-squares estimator
\begin{align}
\label{eq:LSE}
\thetals \defn %\Pi_{\cM}(y) = 
\argmin_{\theta \in \cM} \|\theta - y\|_F^2 ,
\end{align}
in terms of the mean squared error
\begin{align}
\label{eq:MSE}
\frac{1}{\dimone \dimtwo} \| \hat \theta - \theta^* \|_F^2  .
\end{align}
Our upper bound is stated in the following theorem.

\begin{theorem} \label{thm:upper_bound} 
Let \( \theta^* \in \cM^{\dimone, \dimtwo} \) be an anti-Monge matrix, and suppose that we observe $	y = \theta^* + \varepsilon$ where $\varepsilon \sim \subG_{\dimone \times \dimtwo}(\sigma^2).$
Then, the least-squares estimator $\thetals$ achieves the rate
\begin{align}
	\label{eq:ca}
	\frac{1}{\dimone \dimtwo} \| \thetals - \theta^* \|_F^2
	\lesssim \left[ \frac{\sigma^2 }{\dimtwo} + \left( \frac{ \sigma^2 V(\theta^*)}{ \dimone \dimtwo } \right)^{2/3} \log(\dimone)^{1/3} \log(\dimtwo)^{2/3} \right] \wedge \sigma^2
\end{align}
with probability at least $1 - \exp( - \dimone)$. 
Moreover, the same bound holds in expectation.
\end{theorem}

Assuming Gaussian noise, the following theorem provides a lower bound that matches the above upper bound up to a logarithmic factor. For $V_0 \ge 0$, let us define
$$
\cM_{V_0} = \cM_{V_0}^{\dimone, \dimtwo} \defn \{ \theta \in \cM^{\dimone, \dimtwo} : V(\theta) \le V_0 \}. 
$$

\begin{theorem}\label{thm:lower_bound} 
Consider the model
$
y = \theta^* + \varepsilon,
$
where $\theta^* \in \cM_{V_0}^{\dimone, \dimtwo}$ and $\varepsilon$ has i.i.d. $\cN(0, \sigma^2)$ entries.
For any $V_0 \ge 0$, 
%$$
%\frac{4 \sigma}{\sqrt{\dimone \dimtwo}} \le V \le 4 \sigma n_1n_2 .
%$$
it holds that
$$\underset{\tilde\theta}{\inf}\underset{\substack{\theta^* \in\cM_{V_0}}}{\sup} \E \Big[ \frac{1}{\dimone \dimtwo}  \| \tilde\theta - \theta^* \|_F^2 \Big] \gtrsim \Big[ \frac{\sigma^2}{\dimtwo} + \left( \frac{\sigma^2 V_0}{\dimone\dimtwo} \right)^{2/3} \Big] \land \sigma^2,$$
where the infimum is taken over all estimators measurable with respect to the observation $y$.
\end{theorem}

%Note that the lower bound is the minimum of the main term and $\sigma^2$. There is no loss of generality here since the trivial estimator $y$ itself achieves the (inconsistent) rate $\sigma^2$. 

%{The lower boundary of $V$ is clearly not a problem because if $V$ is too small, the second term is dominated by the first. The upper boundary gives rate $\sigma^2 \dimtwo/\dimone$, which is not a problem when $\dimone = \dimtwo$ since the rate becomes trivial then. If $\dimone > \dimtwo$, we probably need to consider different $k$ along the two dimensions. Shouldn't be hard to figure this out.}

\section{Pre-anti-Monge matrix estimation}
\label{sec:per}

In this section, we move on to study the estimation of a pre-anti-Monge matrix, that is, an anti-Monge matrix whose rows and columns have been shuffled by latent permutations. Let $\cS_n$ denote the set of permutations $\pi : [n] \to [n]$. For any matrix $\theta \in \R^{\dimone \times \dimtwo}$ and permutations $\pi_1 \in \cS_{\dimone}, \pi_2 \in \cS_{\dimtwo}$, recall that $\theta( \pi_1, \pi_2 )$ denotes the matrix defined by
$ 
\theta( \pi_1, \pi_2 )_{i, j} = \theta ( \pi_1(i), \pi_2(j) ) .
$
Define the sets
%\begin{align}
%& \cM (\pi_1, \pi_2) = \cM^{\dimone, \dimtwo} (\pi_1, \pi_2) \defn  \{ \theta( \pi_1, \pi_2 ) : \theta \in \cM^{\dimone, \dimtwo} \} \\
%\text{and } \quad  
%& \cM_{V_0} (\pi_1, \pi_2) = \cM_{V_0}^{\dimone, \dimtwo} (\pi_1, \pi_2)  \defn \{ \theta( \pi_1, \pi_2 ) : \theta \in \cM_{V_0}^{\dimone, \dimtwo} \}
%\end{align}
\begin{align}
\cM (\pi_1, \pi_2) \defn  \{ \theta( \pi_1, \pi_2 ) : \theta \in \cM \} 
\quad \text{ and } \quad  
\cM_{V_0} (\pi_1, \pi_2) \defn \{ \theta( \pi_1, \pi_2 ) : \theta \in \cM_{V_0} \}
\end{align}
of anti-Monge matrices shuffled by fixed permutations.

Suppose that we observe
\begin{align}
y = \theta^* (\pi_1^*, \pi_2^*) + \eps ,
\label{eq:model-per}
\end{align}
where $(\pi_1^*, \pi_2^*, \theta^*) \in \cS_{\dimone} \times \cS_{\dimtwo} \times \cM$ and $\eps \sim \subG_{\dimone \times \dimtwo} (\sigma^2 )$. Our goal is to estimate the pre-anti-Monge matrix $\theta^* (\pi_1^*, \pi_2^*)$. 

If two rows (or columns) of $\theta^*$ differ by a constant vector, then the matrix we obtain from switching these two rows is still anti-Monge. 
Therefore, even if the noise $\eps$ is zero, neither the pair of permutations $(\pi_1^*, \pi_2^*)$ nor the matrix $\theta^*$ can be inferred from \( y \).
As a result, measures of permutation and estimation errors such as $\| \theta^* (\hat \pi_1, \hat \pi_2 ) - \theta^* (  \pi_1^*, \pi_2^* ) \|_F$ and 
$\| \hat \theta - \theta^* \|_F,$
may be not be pertinent. This is why, instead of studying identifiability of the permutations and the anti-Monge matrix, we focus on the denoising error
\begin{equation}
	\label{eq:js}
	\| \tilde \theta - \theta^* (  \pi_1^*, \pi_2^* ) \|_F
\end{equation}
for any estimator $\tilde \theta$ of the pre-anti-Monge matrix. 

Depending on the application, it might be important to differentiate between \emph{proper} and \emph{improper} estimators \( \tilde \theta \).
In this context, a proper estimator is an estimator 
\begin{equation}
	\label{eq:jr}
	\tilde \theta \in \bar \cM \defn \bigcup_{\pi_1 \in \cS_{\dimone}, \, \pi_2 \in \cS_{\dimtwo}} \cM(\pi_1, \pi_2) ,
\end{equation}
that is, an estimator that needs to be a pre-anti-Monge matrix itself. 
By contrast, an improper estimator can be any matrix \( \tilde \theta \in \R^{n_1 \times n_2} \).

The rest of this section is organized as follows. We first establish the minimax rate for estimating a  pre-anti-Monge matrix in Section \ref{sec:minimax}. It is achieved by the global least-squares estimator, which is proper by nature, but is likely to be computationally infeasible.
Next, we give a computationally feasible proper estimator in Section \ref{sec:v-sort} under additional assumptions.
Finally, in Section \ref{sec:svt}, we present another computationally feasible estimator based on singular value thresholding that yields a better rate than the one in Section \ref{sec:v-sort}, but may be improper. This presents a shortcoming if one wants to leverage the Monge structure for downstream numerical computations.

\subsection{Minimax rates of estimation}
\label{sec:minimax}

% As a benchmark, we first establish the minimax rates of estimating a pre-anti-Monge matrix.
We work under the technical assumption that $\theta^* \in \cM_{V_0}$ where $V_0$ is known. 
Define
$$
\bar \cM_{V_0} \defn \bar \cM \cap \{\theta \in \R^
{n_1 \times n_2}\,:\, V(\theta) \le V_0\}\,.
$$
Our upper bound is achieved (up to a logarithmic factor) by the global least-squares estimator over the entire parameter space
%, using a projection onto \( \cM_{V_0} (\pi_1, \pi_2)  \) instead of \( \cM (\pi_1, \pi_2) \):
\begin{align}
\label{eq:least-sq}
\thetagls  \in \argmin_{ \theta \in\bar \cM_{V_0} } \| \theta - y\|_F^2 .
%\thetagls  \in \argmin_{ \theta \in \bigcup_{ \substack{ \scriptscriptstyle \pi_1 \in \cS_{\dimone} \\ \scriptscriptstyle \pi_2 \in \cS_{\dimtwo} } } \cM_{V_0} (\pi_1, \pi_2) } \| \theta - y\|_F^2 .
\end{align}
If the minimizer is not unique, an arbitrary one is chosen. 

\begin{theorem} \label{thm:minimax}
Suppose that we have $y = \theta^*( \pi_1^*, \pi_2^*) + \eps$, where $\theta^* \in \cM_{V_0}^{\dimone, \dimtwo}$ and $\eps \sim \subG(\sigma^2)$. 
Then the global least-squares estimator~\eqref{eq:least-sq} achieves the rate
\begin{align}
\frac 1{\dimone \dimtwo} \| \thetagls - \theta^* (  \pi_1^*, \pi_2^* ) \|_F^2 
\lesssim \left[ \frac{\sigma^2 \log(\dimone) }{\dimtwo} + \left( \frac{ \sigma^2 V_0}{ \dimone \dimtwo } \right)^{2/3}  \log(\dimone)^{1/3} \log(\dimtwo)^{2/3} \right] \wedge \sigma^2 \notag 
\end{align}
with probability at least $1 - \dimone^{-\dimone}$.
Moreover, the same bound holds in expectation.
\end{theorem}

Note that this rate is the same (up to a logarithmic factor in the first term) as that for estimating an anti-Monge matrix without latent permutations in view of Theorem~\ref{thm:upper_bound}. Therefore, the lower bound of Theorem~\ref{thm:lower_bound} for the smaller class implies minimax optimality of the above upper bound (up to a logarithmic factor).

We conjecture that a similar bound holds true for a version of the least-squares estimator where the projection onto \( \bar \cM_{V_0}  \) is replaced by the unrestricted version \(\bar \cM \), but our current proof technique does not allow us to conclude this.

\subsection{Efficient estimation via variance sorting}  
\label{sec:v-sort}

While the global least-squares estimator retains the minimax rate even in the presence of latent permutations, solving the optimization problem~\eqref{eq:least-sq} is unlikely to be computationally efficient. Thus we now discuss polynomial-time estimators. 
In this subsection, we assume that the noise matrix $\eps$ is homoscedastic with independent sub-Gaussian entries, i.e.,
\begin{align}
\eps_{i, j} \sim \subG( C \sigma^2 )
\quad \text{ and } \quad
\Var [ \eps_{i, j} ] = \sigma^2 .
\label{eq:iso}
\end{align}
%Our main goal is to give a triplet of estimators $( \hat \pi_1, \hat \pi_2, \hat \theta )$, for which we have a small denoising error 
%\begin{align}
%\| \hat \theta (\hat \pi_1, \hat \pi_2 ) - \theta^* (  \pi_1^*, \pi_2^* ) \|_F . 
%\label{eq:denoise}
%\end{align}

As in the previous section, the estimator is based on projecting a permuted version of the observations onto \( \cM_{V_0} \), but we use an efficient method to find estimators of the permutations with respect to which we project on.
Let us first focus on estimating the row permutation $\pi_1$. 
Since adding a constant to all entries in a row of the underlying matrix does not change its anti-Monge property, there is no first-order information that helps distinguish between the rows of $y$.
Instead, we exploit second-order information, namely, the variance of row differences of~$y$.

The intuition behind the following algorithm is that if we knew the index \( \pi_1^{-1}(1) \) corresponding to the first row of $\theta^*$, the anti-Monge property would imply that the variances between any other row \( i \in [\dimone] \) and row \( 1 \) in the unpermuted matrix \( \theta^\ast \),
\begin{equation}
	\label{eq:jt}
	\sum_{k = 1}^{\dimtwo} \left[ \theta^\ast_{i, k} - \theta^\ast_{1,k} - \frac{1}{\dimtwo} \sum_{\ell = 1}^{\dimtwo} (\theta^\ast_{i, \ell} - \theta^\ast_{1, \ell}) \right]^2,
\end{equation}
are monotonically increasing in \( i \).
Hence, given \( \pi_1(1) \), we could estimate these variances and sort the rows accordingly.
The precise method is given in the following {\sf Variance Sorting} Subroutine. 

% \medskip

\begin{algorithm}[H]
\caption{{\sf Variance Sorting}}
\label{alg:variance-sorting-sub}

\begin{enumerate}
\item
For each pair of rows $(i, j)$ of $y$, compute the variance of their difference
%$$
%\mu{(i, j)} = \frac{1}{\dimtwo} \sum_{k=1}^{\dimtwo} (y_{i, k} - y_{j, k}) 
%\quad \text{ and } \quad 
%\xi{(i, j)} = \sum_{k=1}^{\dimtwo} \big( y_{i, k} - y_{j, k} - \mu{(i, j)} \big)^2 ,
%$$
\begin{align}
\xi{(i, j)} = \sum_{k=1}^{\dimtwo} \Big[ y_{i, k} - y_{j, k} - \frac{1}{\dimtwo} \sum_{\ell=1}^{\dimtwo} (y_{i, \ell} - y_{j, \ell}) \Big]^2 ,  
\label{eq:xi-def}
\end{align}
and define 
\begin{align}
(i_0, j_0) = \argmax_{(i, j) \in [\dimone]^2, \, i < j} \, \xi (i , j) .
\label{eq:def-i0j0}
\end{align}

\item
%We denote 
%$$
%\mu_i = \mu(i, j_0)
%\quad \text{ and } \quad
%\xi_i = \xi(i, j_0) .
%$$
Define $\hat \pi_1 \in \cS_{\dimone}$ so that $\{ \xi ( i_0, \hat \pi_1^{-1}(i) ) \}_{i=1}^{\dimone}$ is nondecreasing in $i$. In particular, we can pick $\hat \pi_1(1) = i_0$ and $\hat \pi_1(\dimone) = j_0$.
% by the definition of $(i_0, j_0)$.
\end{enumerate}

\end{algorithm}

% \smallskip

Note that in Algorithm \ref{alg:variance-sorting-sub}, the pair \( (i_0, j_0) \) is an estimator for the extremal rows \( \pi_1^{-1} (1) \) and \( \pi_1^{-1} (\dimone) \), but the choice of which index corresponds to \( \pi_1^{-1} (1) \) is broken arbitrarily by the constraint \( i_0 < j_0 \).
In turn, the resulting estimator \( \hat \pi_1 \) can only be reliable up to a global flip of the coordinates.
In order to obtain denoising rates, this indeterminacy can be overcome by projecting \( y \) onto the set of anti-Monge matrices under both possible orientations and picking the best fit.

To facilitate our presentation, we define the reversal permutation $\pir \in \cS_{\dimone}$ by $\pir (i) = \dimone - i + 1$ for $i \in [\dimone]$, and define similarly $\pirtwo \in \cS_{\dimtwo}$ by $\pirtwo (i) = \dimtwo - i + 1$ for $i \in [\dimtwo]$.
% Note that the definition of $(i_0, j_0)$ in~\eqref{eq:def-i0j0} requires that $i_0 < j_0$, and consequently $\hat \pi_1(1) < \hat \pi_1(\dimone) $. 
% This condition forces us to choose a ``direction" for the permutation estimator $\hat \pi_1$, but it could be the case that the reversal $\pir \circ \hat \pi_1$, not $\hat \pi_1$ itself, is a good estimator of $\pi_1^*$. Hence it is important to incorporate this flexibility into the Main Algorithm, which we are ready to introduce. 
In short, Algorithm \ref{alg:variance-sorting-main} below applies the {\sf Variance Sorting} subroutine twice to estimate both row and column permutations, and then estimates $\theta$ by the (computationally efficient) least-squares estimator in the convex set of anti-Monge matrices along these estimated permutations.

% \medskip

\begin{algorithm}[H]
	\caption{{\sf Main Algorithm}}
	\label{alg:variance-sorting-main}
\begin{enumerate}
\item
	Find $\hat \pi_1$ using the {\sf Variance Sorting} subroutine, Algorithm~\ref{alg:variance-sorting-sub}.

\item
	With $y$ replaced by $y^\top$ and the roles of indices $1$ and $2$ switched, find $\hat \pi_2$ using the {\sf Variance Sorting} subroutine, Algorithm~\ref{alg:variance-sorting-sub}.

\item
Compute the least-squares estimator $\hat \theta$ as follows. If
$$
\min_{\theta \in \cM_{V_0}} \| \theta (\hat \pi_1, \hat \pi_2 ) - y \|_F^2 
\le \min_{\theta \in \cM_{V_0}} \| \theta (\pir \circ \hat \pi_1, \hat \pi_2 ) - y \|_F^2 ,
$$
then we define $\hat \pi_1' \defn \hat \pi_1$. Otherwise, we define $\hat \pi_1' \defn \pir \circ \hat \pi_1$. Finally, we set 
$$
\hat \theta %= \Pi_{\cM} \big( y (\hat \pi_1, \hat \pi_2 )  \big) 
\defn \argmin_{\theta \in \cM_{V_0}(\hat \pi_1', \hat \pi_2)} \| \theta - y \|_F^2 .
$$
\end{enumerate}
\end{algorithm}

% \smallskip

%The reason for the potential reversal $\pir$ applied to $\hat \pi_1$ is that the {\sf Variance Sorting} Subroutine below only estimates a permutation up to a reversal. 
%Furthermore, 
Note that we only allowed a potential flip $\pir$ for $\hat \pi_1$, although there is also such an ambiguity for $\hat \pi_2$. This suffices because if $\theta \in \cM_{V_0}$, then $\theta( \pir, \pirtwo ) \in \cM_{V_0}$, and as a result 
\begin{align}
	&\cM_{V_0}(\hat \pi_1, \hat \pi_2) \cup \cM_{V_0} (\pir \circ \hat \pi_1, \hat \pi_2) \notag \\
	&= \cM_{V_0}(\hat \pi_1, \hat \pi_2) \cup \cM_{V_0} (\pir \circ \hat \pi_1, \hat \pi_2) \cup \cM_{V_0} (\hat \pi_1, \pirtwo \circ \hat \pi_2) \cup \cM_{V_0} (\pir \circ \hat \pi_1, \pirtwo \circ \hat \pi_2) . \notag
\end{align}

The estimator computed by the Main Algorithm achieves the following rate of estimation. 

\begin{theorem} \label{thm:per-proj}
Suppose that $y = \theta^*( \pi_1^*, \pi_2^*) + \eps$, where $\theta^* \in \cM_{V_0}^{\dimone, \dimtwo}$ and $\eps$ has independent $\subG( C \sigma^2 )$ entries with variance $\sigma^2 .$
Let the estimator $\hat \theta$ be given by the Main Algorithm. Then it holds with probability at least $1 - \dimone^{- \dimone}$ that 
\begin{align}
	\frac{1}{\dimone \dimtwo} \| \hat \theta - \theta^*(\pi^\ast_1, \pi^\ast_2) \|_F^2 
	\lesssim \big( \sigma^2 + \sigma V_0 \big)  \Big( \frac{ \log  \dimone}{ \dimtwo} \Big)^{1/2} . 
\end{align}
Moreover, the same bound holds in expectation.
\end{theorem}

This rate achieved by our efficient estimator is consistent, but it is suboptimal in view of the minimax rate given by Theorem~\ref{thm:minimax}.

\subsection{Denoising via singular value thresholding}
\label{sec:svt}

While the {\sf Variance Sorting} algorithm above yields efficient estimators of the latent permutations, the rate of convergence it achieves is suboptimal. We now aim for the easier task of denoising the pre-anti-Monge matrix without learning the latent permutations, in the hope of obtaining an efficient estimator with a faster rate of convergence.
More precisely, under model~\eqref{eq:model-per}, we look for a possibly improper estimator $\tilde \theta \in \R^{\dimone \times \dimtwo}$ so that $\big\| \tilde \theta - \theta^* ( \pi_1^*, \pi_2^* ) \big\|_F^2$ is small.

To this end, we consider the well-studied singular value thresholding (SVT) estimator~\cite{Cha15, GavDon14}. Let the singular value decomposition of $y$  be
$$
y = \sum_{i=1}^{\dimtwo} \lambda_i u_i v_i^\top .
$$
Then the SVT (hard-thresholding) estimator is defined as
\begin{align}
\thetasvt \defn \sum_{i=1}^{\dimtwo} \1 \{ \lambda_i > \rho \}  \lambda_i u_i v_i^\top ,
\label{eq:svt}
\end{align}
where we choose the threshold to be $\rho \defn C \sigma \sqrt{\dimone}$ for a sufficiently large constant $C > 0$. The rate of estimation achieved by the SVT estimator is given in the following theorem. 

\begin{theorem} \label{thm:svt}
Suppose that we have $y = \theta^*( \pi_1^*, \pi_2^*) + \eps$, where $\theta^* \in \cM^{\dimone, \dimtwo}$ and $\eps \sim \subG(\sigma^2)$.  The singular value thresholding estimator $\thetasvt$ achieves the rate 
$$
\frac{1}{\dimone \dimtwo} \| \thetasvt - \theta^* ( \pi_1^*, \pi_2^*) \|_F^2 \lesssim  \left[ \frac{\sigma^2}{\dimtwo} + \frac{ \sigma^{3/2} V(\theta^*)^{1/2} }{ \dimtwo^{3/4} } \right] \wedge \sigma^2
$$
with probability at least $1 - \exp(- \dimone)$.
Moreover, the same bound holds in expectation.
\end{theorem}

This rate sits between the minimax rate given by Theorem~\ref{thm:minimax}, and the rate for the {\sf Variance Sorting} estimator given by Theorem~\ref{thm:per-proj}. Note that for this result, the noise $\eps$ needs not be homoscedastic, and moreover, no knowledge of $V_0$ is required, i.e., the SVT estimator adapts to the quantity $V( \theta^* )$. 

The proof technique leading to upper bounds for the SVT estimator is well developed~\cite{Cha15, ShaBalGunWai17}. Our contribution mainly lies in the following low-rank approximation result for an anti-Monge matrix, which is of independent interest. 

\begin{proposition} \label{prop:low-rank-approx}
For any $\theta \in \cM^{\dimone, \dimtwo}$ and positive integer $r$,  there exists a rank-$(3r+3)$ matrix $\tilde{\theta} \in \R^{\dimone \times \dimtwo}$ such that
$$
\|\tilde{\theta} - \theta \|_F^2 \le 2 \frac{\dimone \dimtwo}{r^3} V(\theta)^2 .
$$
\end{proposition}

Note that using a similar proof, the same rate as in Theorem~\ref{thm:svt} can be obtained for a soft-thresholding estimator as well, that is, for
\begin{equation}
	\label{eq:ju}
	\thetasoft \defn \sum_{i=1}^{\dimtwo} \big( (\lambda_i - \rho) \vee 0 \big) u_i v_i^\top,
\end{equation}
with a similar scaling for \( \rho \).

As the rate given in Theorem~\ref{thm:svt} does not match the minimax rate, it is natural to ask whether this suboptimality is an artifact of the proof or a true weakness of the SVT estimator. In Appendix~\ref{sec:svt-lower}, we present a worst-case anti-Monge matrix which cannot be approximated by any low-rank matrix at a rate better than that given by Proposition~\ref{prop:low-rank-approx}. This in turn gives evidence that the rate of convergence for the SVT estimator in Theorem~\ref{thm:svt} might be the best achievable by this method.

\section{Numerical experiments}
\label{sec:numerics}

In order to compare our theoretical guarantees with the empirical performance of the proposed estimators, we conducted experiments on synthetic data, using Dykstra's algorithm to project onto the cone of anti-Monge matrices.

We first present this projection algorithm in Section \ref{sec:dykstra}. We then show the experimental results of the  projection onto the cone of anti-Monge matrices in Section~\ref{sec:numerics-cone} and of the two efficient strategies for denoising pre-anti-Monge matrices in Section~\ref{sec:numerics-perm}.

\subsection{{Dykstra's algorithm for projecting onto the set of anti-Monge matrices}}
\label{sec:dykstra}

% To compute the projection of a matrix \( y \) to \( \cM \), we employed Dykstra's algorithm specified to the case of anti-Monge matrices, as outlined in the appendix [].

Since the set \( \cM \) is a convex cone specified by \( O(\dimone \dimtwo) \) constraints, the least-squares estimator \eqref{eq:LSE} can be calculated by a general purpose convex optimization software such as SCS \cite{DonChuPari16, scs} or EOCS \cite{DomChuBoy13}.
The most computationally intensive subroutine of these methods is usually solving linear systems associated with the constraints specifying \( \cM \).
Using direct methods to find these solutions results in a runtime that scales like \( (\dimone \dimtwo)^3 \), rendering calculations relatively slow even for moderate values of \( \dimone \) and \( \dimtwo \).
Hence, we chose to implement a specialized algorithm to calculate \( \theta \) based on Dykstra's projection algorithm \cite{BoyDyk86, ComPes11}.

In its general form (see Algorithm \ref{alg:dykstra}), this algorithm is designed to calculate the projection of a vector \( y \in \R^d \) onto the intersection of \( m \) convex sets \( \cM_1, \dots, \cM_m \) by iteratively projecting carefully chosen points to each individual set.
This is similar to alternate projections of a point to each of the sets \( \cM_1, \dots, \cM_m \), but when initialized with \( y \in \R^d \), Dykstra's algorithm not only finds a point in the intersection \( \bigcap_{j \in [m]} \cM_j \), but its iterates actually converge to the projection of \( y \) onto \( \bigcap_{j \in [m]} \cM_j  \).

\begin{algorithm}[H]
	\caption{{\sf Dykstra's algorithm}}
	\label{alg:dykstra}
	\begin{algorithmic}
		\Input \( y \in \R^d \), the point to project; $\cM_1, \ldots, \cM_m$ a collection of cones
		\Output \( \theta \), an approximation to the projection of $y$ onto $\cM_1 \cap \cdots \cap \cM_m$
		\Function{ProjectDykstra}{y}
		\For{\( i = 1, \dots, m \)}
			\State{\( p_i = \mathbf{0}_d \)}
			\Comment Initialize residuals
		\EndFor
		\State{\( \theta_m = y \)}
		\Comment Initialize iterates
		\While{not converged}
		\For{\( i = 1, \dots, m \)}
		\State \( \theta_i \gets \Pi_{\cM_i}(\theta_{(i - 2) \% m + 1} + p_i) \)
		\Comment Project shifted iterates
		\State \( p_i \gets \theta_{(i - 2) \% m + 1} + p_i - \theta_i \)
		\Comment Compute new residual
		\EndFor
		\EndWhile
		\State \Return \( \theta \)
		\EndFunction
	\end{algorithmic}
\end{algorithm}

To apply Dysktra's algorithm to the problem of projecting onto the cone of anti-Monge matrices, note that we can write \( \cM = \bigcap_{i_1=1}^{\dimone - 1} \bigcap_{i_2=1}^{\dimtwo - 1} \cM_{i_1, i_2} \) with
\begin{equation}
	\label{eq:fs}
	\cM_{i_1, i_2} \defn \Big\{ \theta \in \R^{\dimone, \dimtwo} : \sum_{j_1 \in \{0,1\}, \, j_2 \in \{0,1\}} (-1)^{j_1 + j_2} \theta_{i_1 + j_1, i_2 + j_2} \ge 0 \Big\} ,
\end{equation}
because a matrix is anti-Monge if and only if each contiguous \( 2 \times 2 \) submatrix is anti-Monge.
The projection of $y$ onto \( \cM_{i_1, i_2} \) can be explicitly calculated to be the matrix with entries
\begin{align}
	&\big[ \Pi_{\cM_{i_1, i_2}} (y) \big]_{i_1 + j_1, i_2 + j_2} \notag \\
	&=  y_{i_1 + j_1, i_2 + j_2} + \frac{(-1)^{j_1 + j_2}}{4} \max \left\{ -\sum_{k_1 \in \{0,1\}, \, k_2 \in \{0,1\}} (-1)^{k_1 + k_2} y_{i_1 + k_1, i_2 + k_2}, 0 \right\} \notag
\end{align}
for \( j_1, j_2 \in \{0,1\} \), and
\begin{equation}
	\label{eq:fu}
	\big[ \Pi_{\cM_{i_1, i_2}} (y) \big]_{\ell_1, \ell_2} =  y_{\ell_1, \ell_2},
\end{equation}
if \( (\ell_1, \ell_2) \notin (i_1 + \{0, 1\}) \times (i_2 + \{0, 1\}) \).

This leads to Algorithm \ref{alg:fast-proj} for projecting a matrix \( y \in \R^{\dimone \times \dimtwo} \) onto \( \cM \).

\begin{algorithm}[H]
	\caption{{\sf Fast Projection onto \( \cM \)}}
	\label{alg:fast-proj}
	\begin{algorithmic}
		\Input \( y \in \R^{\dimone \times \dimtwo} \)
		\Output \( \theta \approx \Pi_{\cM}(y) \)
		\Function{ProjetAntiMonge}{y}
		% \Input \( y \in \R^{\dimone \times \dimtwo} \)
		% \Output \( \theta \in \R^{\dimone \times \dimtwo} \)
		\State \( \eta \gets 0 \in \R^{(\dimone - 1) \times (\dimtwo - 1)} \)
		\Comment Initialize residuals
		\State \( \theta \gets y \),
		\Comment Initialize iterates
		\While{not converged}
		\For{\( i_1 = 1, \dots, \dimone - 1,\, i_2 = 1, \dots, \dimtwo - 1 \)}
			\State \( \tilde \eta \gets \max \left\{ -\sum_{j_1 \in \{0, 1\}, j_2 \in \{0, 1\}} (-1)^{j_1 + j_2} \theta_{i_1 + j_1, i_2 + j_2}/4 + \eta_{i_1, i_2}, 0 \right\} \)
			\Comment Compute new residuals
			\For{\( j_1 \in \{0, 1\}, \, j_2 \in \{0, 1\} \)}
				\State \( \theta_{i_1 + j_1, i_2 + j_2} \gets \theta_{i_1 + j_1, i_2 + j_2} + (-1)^{j_1 + j_2} (\tilde \eta - \eta_{i_1, i_2}) \)
				\Comment Project shifted iterates
			\EndFor
			\State \( \eta_{i_1, i_2} \gets \tilde \eta \)
			\Comment Store residuals
		\EndFor
		\EndWhile
		\State \Return \( \theta \)
		\EndFunction
	\end{algorithmic}
\end{algorithm}

The rate of convergence of Dykstra's method can be shown to be linearly exponential in the iterations \cite{DeuHun94}, that is, if we denote by \( \theta^{(k)} \) the \( k \)th iterate of \( \theta \) in Algorithm \ref{alg:fast-proj} and by \( \theta^\ast = \Pi_{\cM}(y) \), then \( \| \theta^{(k)} - \theta^\ast \|_2 \lesssim c^k \) for a constant \( c < 1 \).
However, note that the constant \( c \) may get closer to one with increasing \( \dimone \) and \( \dimtwo \), which is the case for isotonic regression as shown in \cite{DeuHun94} and matches our experience: simulations for larger values of \( \dimone \) and \( \dimtwo \) require more iterates before convergence.

In practice, convergence in Algorithm \ref{alg:fast-proj} can be checked by evaluating a measure of feasibility such as \( \| \Done \theta \Dtwo^\top \|_\infty \), or by checking when the distance between two successive iterates is small.

\subsection{Experiments for anti-Monge matrices}
\label{sec:numerics-cone}

In the following two sections, we assume \( n = \dimone = \dimtwo \) for simplicity.

For the estimation of anti-Monge matrices, we consider the following family of ground truth signals, motivated by the construction of the lower bounds in the proof of Theorem \ref{thm:lower_bound}.
First, for \( n \in \mathbb{N} \) and \( V, \sigma > 0 \), define \( \theta_{1, (n)} \in \R^{n \times n}\) as
\begin{equation}
	\label{eq:ji}
	(\theta_{1, (n)})_{i, j} = \frac{V}{\lfloor k \rfloor^2} \left\lfloor \frac{(i-1) k}{n-1} \right\rfloor \left\lfloor \frac{(j-1) k}{n-1} \right\rfloor, \quad i \in [n], j \in [n],
\end{equation}
where $k=(Vn/\sigma)^{1/3}$.
%\begin{equation}
%	\label{eq:jj}
%	k = \left( \frac{V n}{\sigma} \right)^{1/3}.
%\end{equation}
The ground truth \( \theta^\ast_{1, (n)} \) is obtained by centering \( \theta_{1, (n)} \) to have zero column and row sums.
Finally, we set $y = \theta^\ast_{1, (n)} + \varepsilon$ where  $\varepsilon_{i,j} \simiid \cN(0, \sigma)$
%\begin{equation}
%	\label{eq:jm}
%	y = \theta^\ast_{1, (n)} + \varepsilon, \quad \varepsilon_{i,j} \simiid N(0, \sigma),
%\end{equation}
and report the average denoising error \( \| \theta^{\mathsf{ls}} - \theta^\ast_{1, (n)} \|_F^2 / n^2 \) over 20 repetitions.

Our simulations recover the three regimes for \( n \) that appear in Theorem \ref{thm:upper_bound}, although at  different signal-to-noise ratios governed by \( V/\sigma \).
Namely, on the one hand, for \( V = \sigma = 1 \), we see in Figure~\ref{fig:supmod-n1} an error decay of \( n^{-1.02} \approx n^{-1} \) for \( n \) between \( 10 \) and \( 160 \), obtained by linearly regressing the logarithm of the errors onto the logarithm of the \( n \) values.
On the other hand, for \( V = 2\cdot 10^6 \), we can see both a plateau when the trivial \( \sigma^2 \) error bound in Theorem \ref{thm:upper_bound} is active, as well as a decay of \( n^{-1.34} \approx n^{-4/3} \) at the beginning of the decay becoming effective, where the slope in the doubly logarithmic plot is read off between two consecutive points as indicated in Figure~\ref{fig:supmod-n2}.

Similarly, fixing \( n = 200 \), \( \sigma = 1 \), and varying \( V \) between \( 10^{-2} \) and \( 10^7 \), we can observe a \( V^{0.65} \approx V^{2/3} \) scaling in Figure~\ref{fig:supmod-V}. %, but only between two consecutive values of \( V \).
The overall curve is shallower, plateauing both at the far low and high end of \( V \), corresponding to the \( \sigma^2/n \) and \( \sigma^2 \) rates becoming active, respectively.

Finally, in Figure~\ref{fig:supmod-sigma}, when setting \( n = 300 \), \( V = 1 \), and varying \( \sigma \) between \( 10^{-7} \) and \( 1 \), we obtain slopes of \( \sigma^{2.01} \) and \( \sigma^{1.84} \) on the low and high end, while the lowest slope between consecutive points in the curve is \( \sigma^{1.34} \), which matches the theoretical rates of \( \sigma^2 \), \( \sigma^2/n \) and \( (V \sigma^2/n^2)^{2/3} \), respectively.

\subsection{Experiments for pre-anti-Monge matrices}
\label{sec:numerics-perm}
To illustrate the practical performance of the efficient methods presented for denoising a pre-anti-Monge matrix, {\sf Variance Sorting} and singular value thresholding (see Sections~\ref{sec:v-sort} and~\ref{sec:svt} respectively), we further perform experiments by using both methods on the following family of ground truth matrices:
\begin{equation}
	\label{eq:jk}
	% (\theta_{2, (n)})_{i, j} = V \frac{i-1}{n-1} \frac{j-1}{n-1},
	\theta^\ast_{2, (n)} = \frac{V}{n-1} D^\dagger (D^{\dagger})^{\top}.
\end{equation}
These were chosen because the singular value decay we proved in Proposition \ref{prop:low-rank-approx}  is tight for these matrices (see Lemma \ref{lem:approx-lower}).
By contrast, each ground truth example in the previous subsection, \( \theta^\ast_{1, (n)} \), is a rank-one matrix, and hence should lead to an overall better performance of singular value thresholding that is independent of \( n \).

For the {\sf Variance Sorting} algorithm, we set \( V = 1 \), \( \sigma = 0.5 \) and report the approximation error induced by the estimated permutations, i.e.,
\begin{equation}
	\label{eq:jn}
	\min_{\substack{ \pi_1 \in \{\id, \pir\}\\  \pi_2 \in \{\id, \pirtwo\}}} \frac{1}{n^2} \| \theta^\ast( \pi_1 \circ \hat \pi_1,  \pi_2 \circ \hat \pi_2) - \theta^\ast \|_F^2
\end{equation}
for \( \theta^\ast = \theta^\ast_{2, (n)} \), averaged over \( 256 \) repetitions.
This measure of the approximation quality of the estimated permutations corresponds to the upper bound used in the proof of Proposition~\ref{prop:per-err} (see \eqref{eq:per-err}) and is applicable since by construction, \( \theta^\ast \) has row and column sums equal to zero.
It is the dominating part in the error analysis, leading to the rate reported in Theorem~\ref{thm:per-proj}, and it allows us to study a larger range of \( n \), avoiding the need for subsequent projection of the permuted \( y \) matrix.

In Figure~\ref{fig:variance-sort}, we observe that while for smaller \( n \), we see a slower decay than predicted, for larger \( n \), the decay scales like \( n^{-0.47} \approx n^{-1/2} \), close to the predicted rate.

Finally, we perform singular value thresholding on the same set of ground truth matrices, this time setting \( V = 1 \), \( \sigma = 0.1 \), and varying \( n \) between \( 20 \) and \( 500 \).
For this experiment, in Figure~\ref{fig:svt}, we plotted the full denoising error,
\begin{equation}
	\label{eq:jo}
	\frac{1}{n^2} \| \hat \theta - \theta^\ast \|_F^2,
\end{equation}
averaged over \( 64 \) repetitions.
As in the other experiments, we can see an error decay that is close to our theoretical guarantees, that is, \( n^{-0.73} \approx n^{-3/4} \).

\begin{figure}[ht]
	\label{fig:supmod}
	\begin{subfigure}{0.45\textwidth}
		\includegraphics[width=\textwidth]{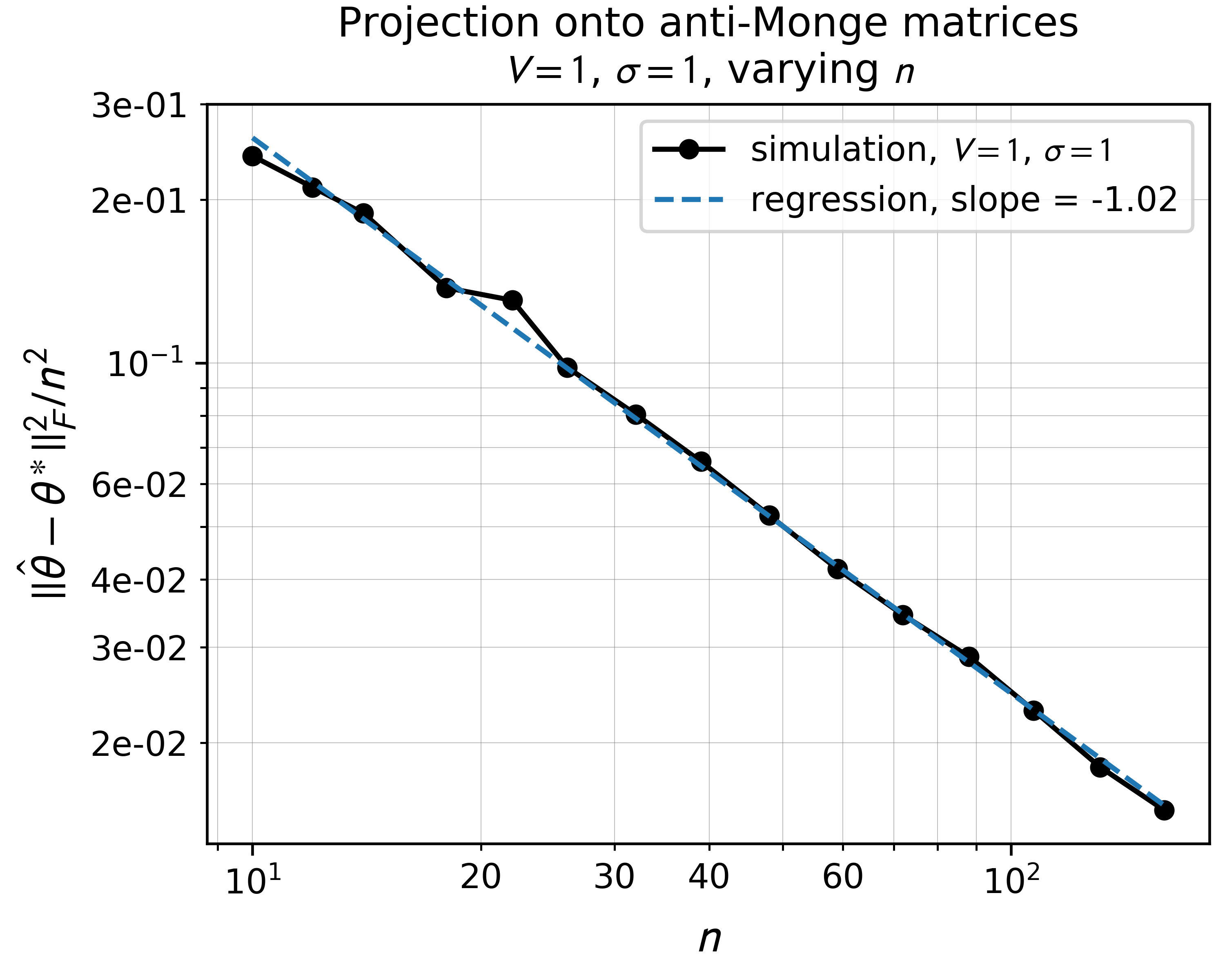}
		\caption{\( n^{-1} \) scaling}
		\label{fig:supmod-n1}
	\end{subfigure}
	\hspace{1em}
	\begin{subfigure}{0.45\textwidth}
		\includegraphics[width=\textwidth]{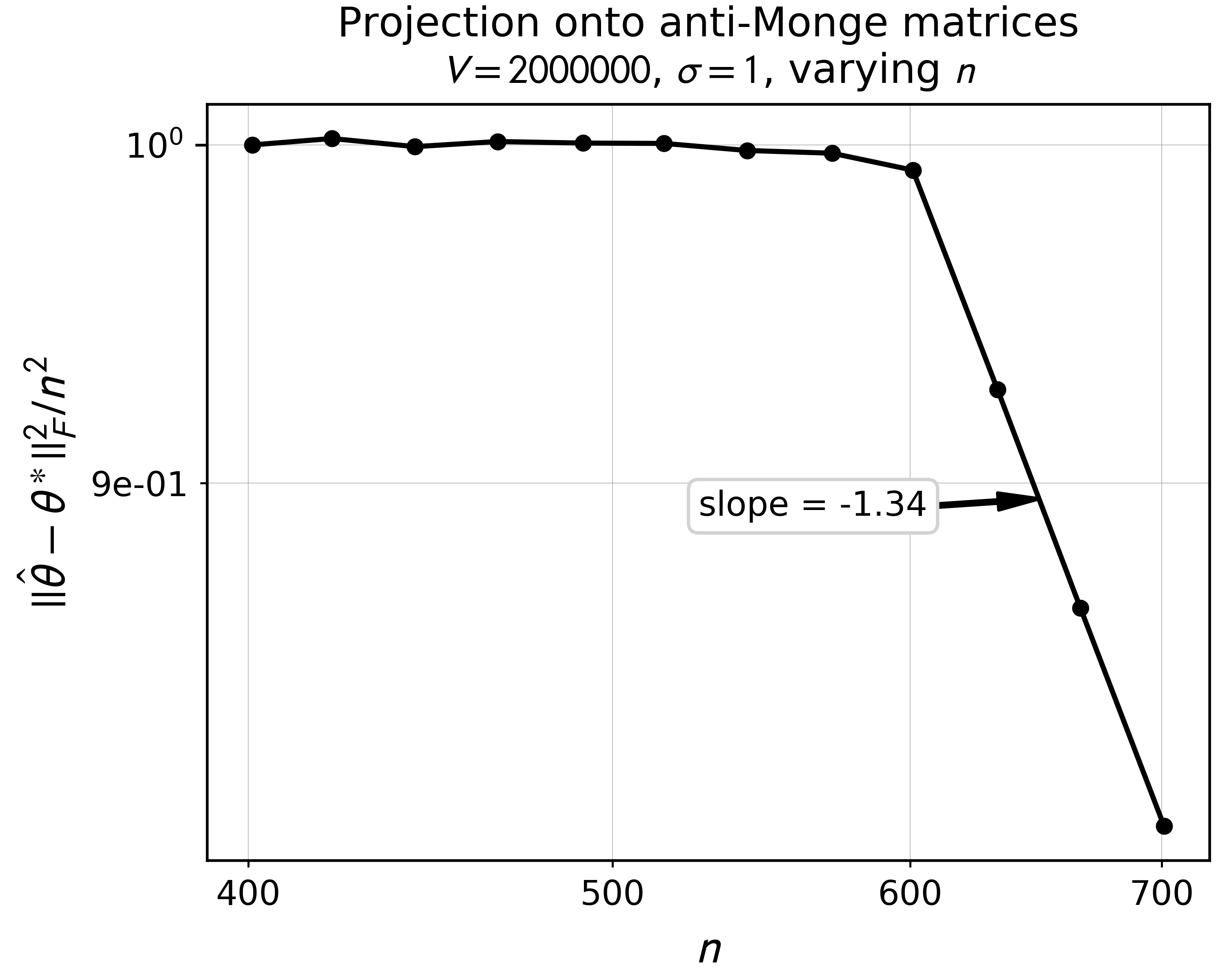}
		\caption{\( n^{- 4/3} \) scaling}
		\label{fig:supmod-n2}
	\end{subfigure}
	\caption{Varying \( n \) for projection onto \( \cM \).
	When an arrow is present, ``slope" indicates the slope between two consecutive points.}
\end{figure}
\begin{figure}[ht]
	\label{fig:supmod2}
	\begin{subfigure}{0.45\textwidth}
		\includegraphics[width=\textwidth]{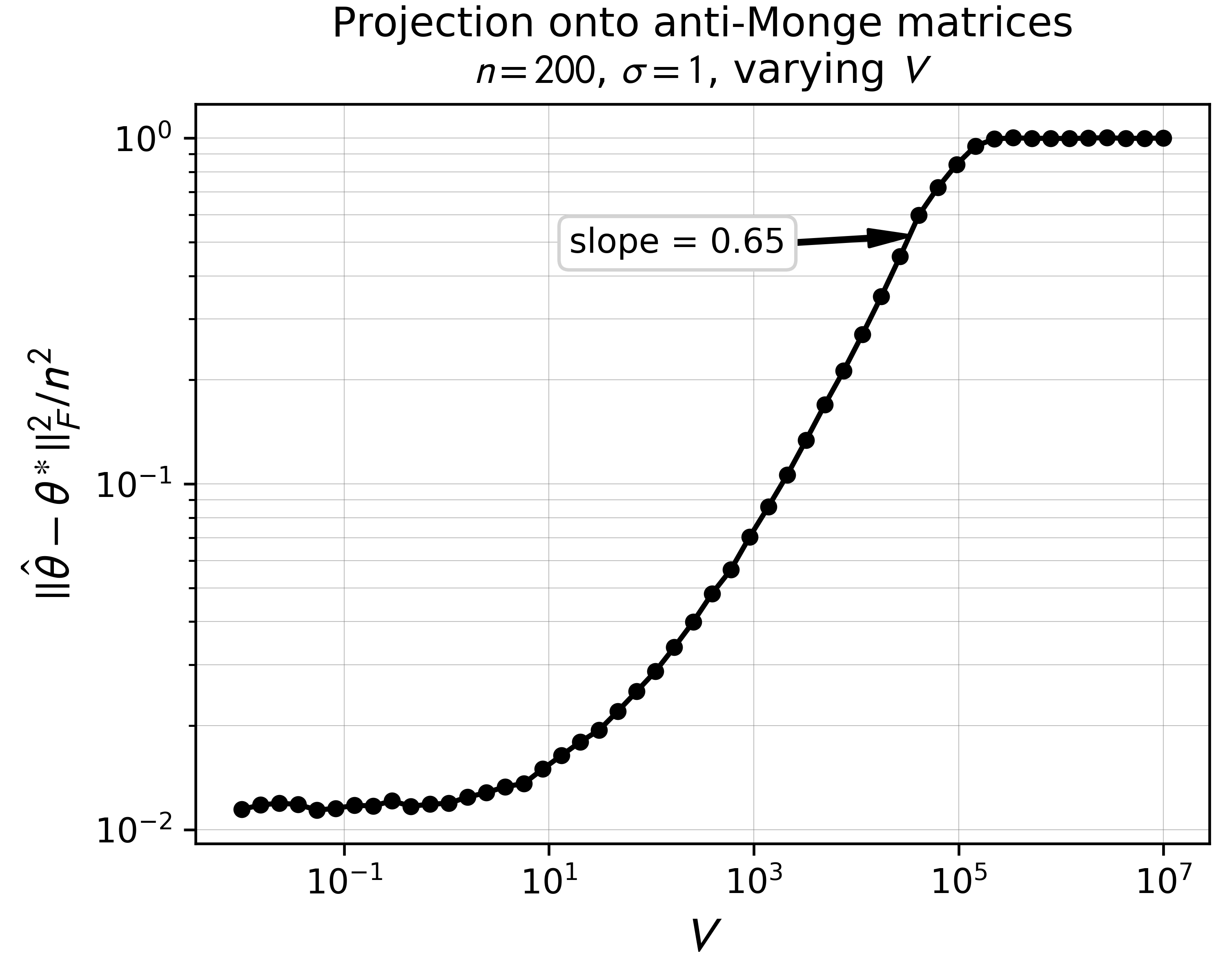}
		\caption{Scaling with respect to \( V \)}
		\label{fig:supmod-V}
	\end{subfigure}
	\hspace{1em}
	\begin{subfigure}{0.45\textwidth}
		\includegraphics[width=\textwidth]{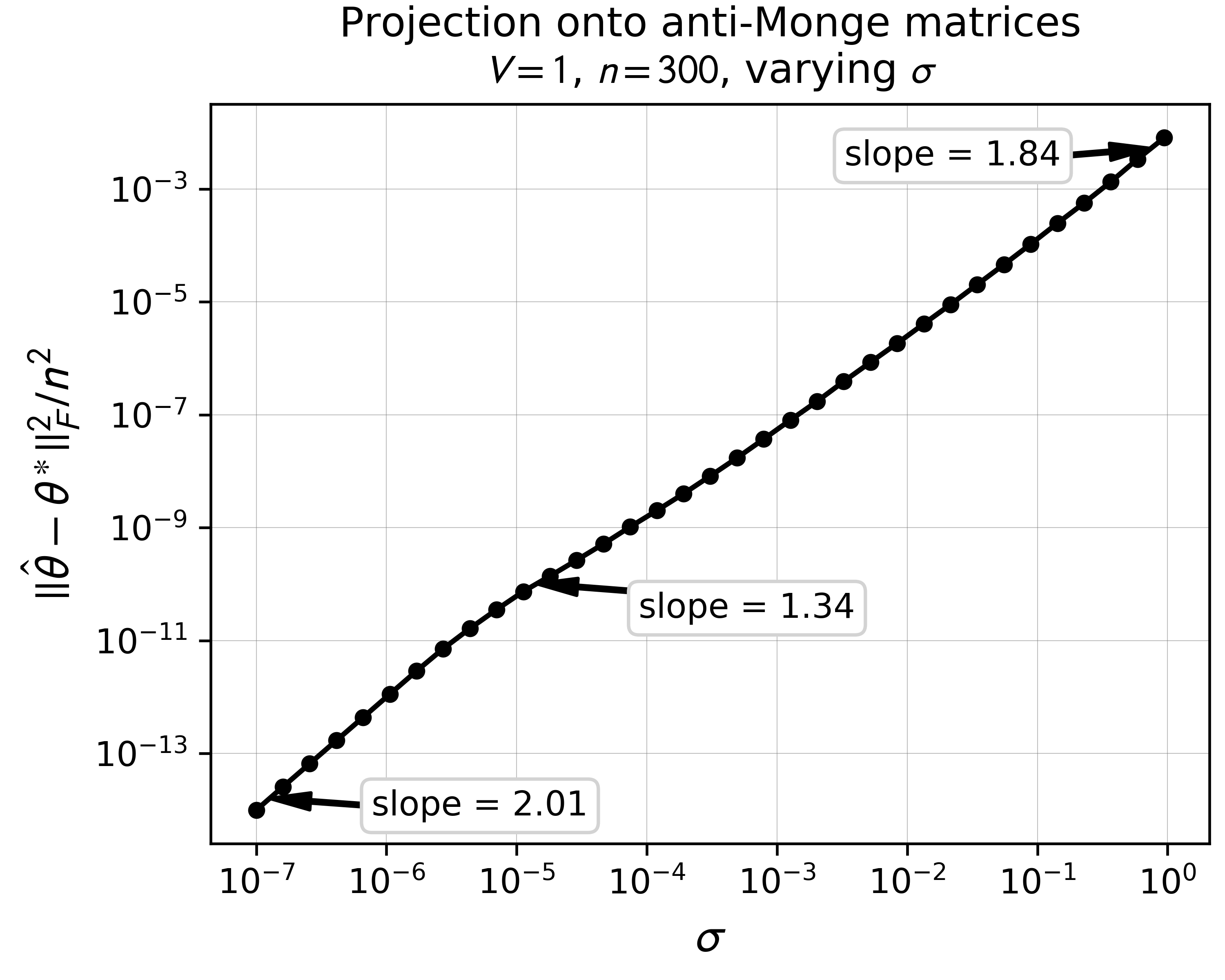}
		\caption{Scaling with respect to \( \sigma \)}
		\label{fig:supmod-sigma}
	\end{subfigure}
	\caption{Varying \( \sigma \), and \( V \) individually for projection onto \( \cM \).
	When an arrow is present, ``slope" indicates the slope between two consecutive points.}
\end{figure}
\begin{figure}[ht]
	\label{fig:perm}
	\begin{subfigure}{0.45\textwidth}
		\includegraphics[width=\textwidth]{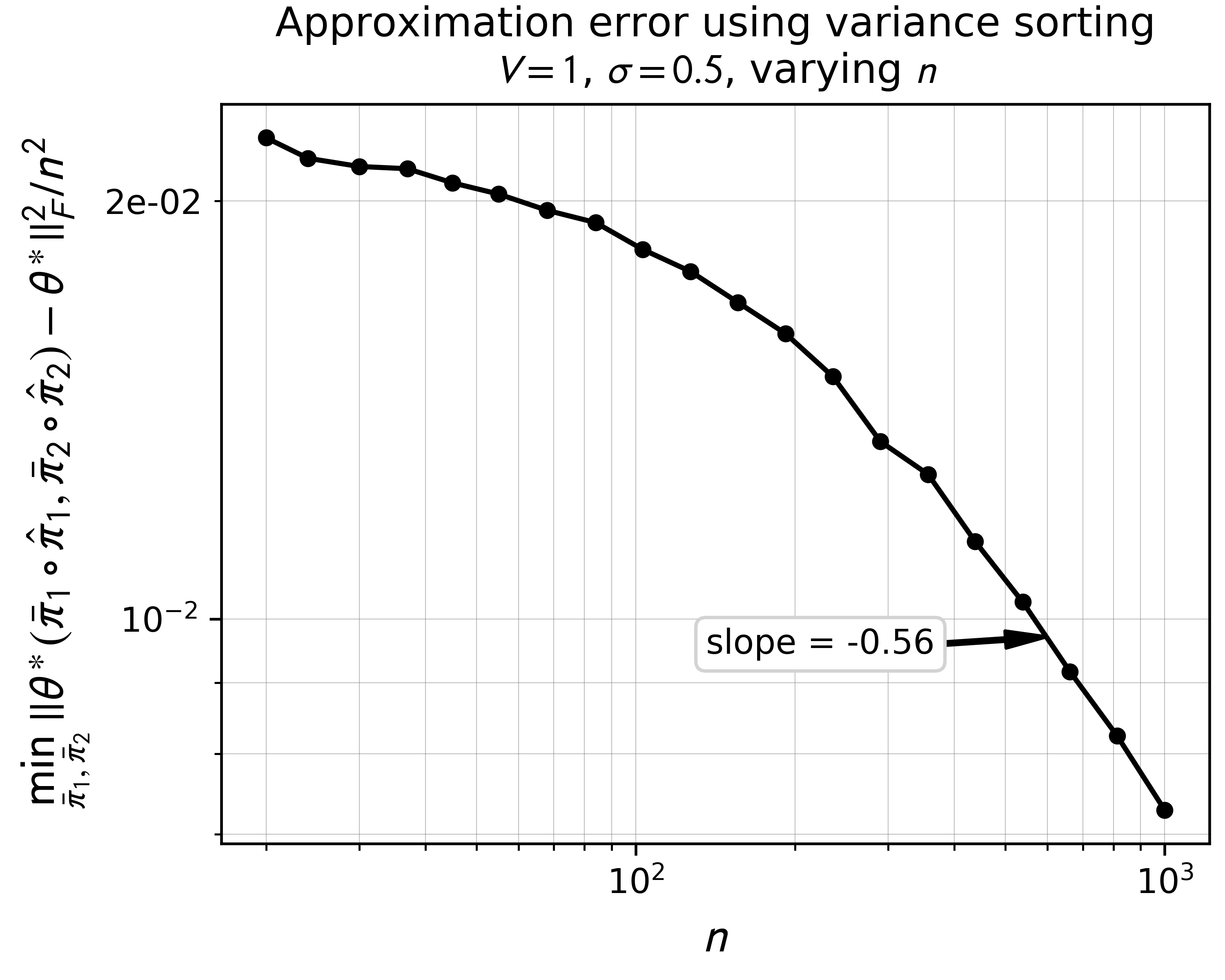}
		\caption{Variance sorting}
		\label{fig:variance-sort}
	\end{subfigure}
	\hspace{1em}
	\begin{subfigure}{0.45\textwidth}
		\includegraphics[width=\textwidth]{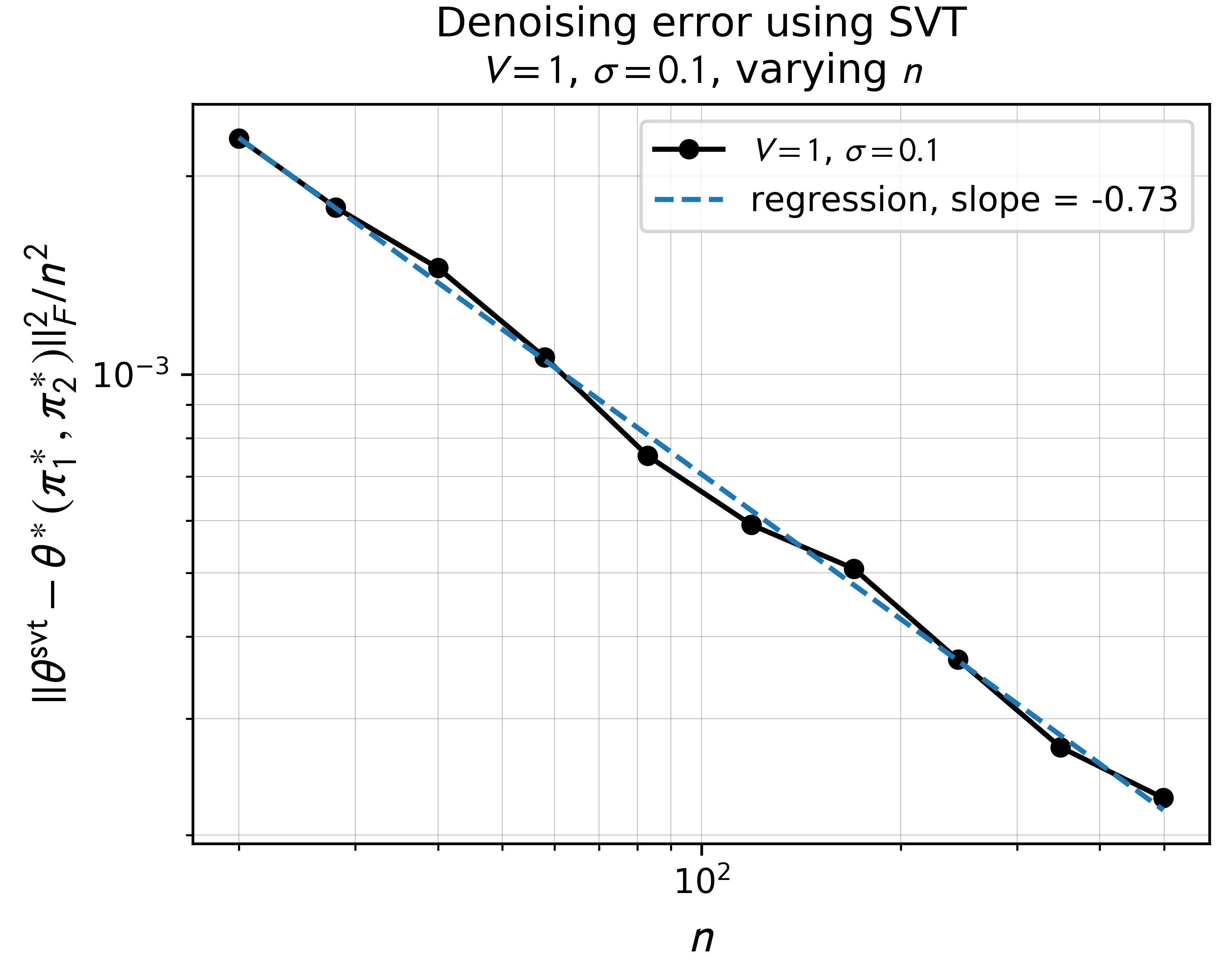}
		\caption{Singular value thresholding}
		\label{fig:svt}
	\end{subfigure}
	\caption{Algorithms for denoising pre-anti-Monge matrix.
	When an arrow is present, ``slope" indicates the slope between two consecutive points.}
\end{figure}

\section{Proofs}
\label{sec:proofs}

In this section, we provide the proofs of our results. 
% \notecm{organization of proofs}
Recall that $\Done$ is defined in~\eqref{eq:def-d} and $\Dtwo$ is defined analogously for dimension $\dimtwo$. 
In the sequel, whenever we introduce notation in dimension $\dimone$,  the analogous object in dimension $\dimtwo$ is denoted by the same symbol with a tilde.

\subsection{A structural lemma}

To gain insight into the set of anti-Monge matrices, we first state and prove the following structural lemma, which says that any anti-Monge matrix can be written as the sum of a constant-row matrix, a constant-column matrix, and an anti-Monge, bivariate isotonic matrix. 
Several structural decompositions of this type are known, 
%about decomposing an anti-Monge matrix into the sum of a constant-row matrix, a constant-column matrix, and another anti-Monge matrix with a particular structure, 
e.g., Lemma~2.1 of~\cite{BurKliRud96} and Lemma~2.5 of~\cite{Fie02}. 
Our lemma is stated in a form that is convenient for application, and we provide a self-contained proof which also facilitates understanding the structure of an anti-Monge matrix. 

\begin{lemma}
\label{lem:c}
For each \( \theta \in \cM^{\dimone, \dimtwo} \), there exists a unique triple \( (R, S, B) \) of matrices in \( \R^{\dimone \times \dimtwo} \) such that
\begin{itemize}
\item
$\theta = R + S + B$;

\item
$R \Dtwo^\top = 0$ and $\Done S = 0$, i.e., \( R \) has constant rows and \( S \) has constant columns;

\item
$S_{i, 1} = B_{i,1} = B_{1,j} = 0$ for $i \in [\dimone] , \, j \in [\dimtwo] . $
\end{itemize}
Moreover, we have that
\begin{itemize}
\item
$B$ is anti-Monge;

\item
%$\Done B \geq 0$ and $B \Dtwo^\top \geq 0$ (i.e., 
\( B  \) is bivariate isotonic, i.e., $B$ has nondecreasing rows and columns;

\item
$ \max\{ \| R \|_\infty, \| S \|_\infty, \| B \|_\infty \} \leq 4 \| \theta \|_\infty $;

\item
$\| B \|_\infty = B_{\dimone, \dimtwo} = V(\theta)$.
\end{itemize}
\end{lemma}

\begin{proof}
By the condition $S_{i, 1} = B_{i, 1} = 0$ for $i \in [\dimone]$, the first column of $R$ is equal to that of $\theta$. Since $R$ has constant rows, it is uniquely defined by $R = \theta_{\cdot,1} \AOtwo^\top$
%\begin{align}
%\label{eq:al}
%R = \theta_{\cdot,1} \AOtwo^\top ,
%\end{align}
where \( \AOtwo \) denotes the all-ones vector in \( \R^{\dimtwo} \).
Moreover, by the condition $B_{1, j} = 0$ for $j \in [\dimtwo]$, the first row of $S$ is equal to that of $\theta - R$. Since $S$ has constant columns, it is uniquely defined by $S = \AOone (\theta - R)_{1, \cdot}$
%\begin{align}
%\label{eq:am}
%S = \AOone (\theta - R)_{1, \cdot} ,
%\end{align}
where \( \AOone \) denotes the all-ones vector in \( \R^{\dimone} \).
Finally, define $B = \theta - R - S.$
%\begin{align}
%\label{eq:ao}
%B = \theta - R - S.
%\end{align}

Since \( \Done S = R \Dtwo^\top = 0 \), we have $\Done B \Dtwo^\top = \Done (\theta - R - S) \Dtwo^\top = \Done \theta \Dtwo^\top \geq 0$.
%\begin{align}
%\label{eq:an}
%\Done B \Dtwo^\top = \Done (\theta - R - S) \Dtwo^\top = \Done \theta \Dtwo^\top \geq 0 .
%\end{align}
so that  $B$ is anti-Monge. Moreover, this implies
\begin{align}
(\Done B)_{i, j} \leq {} & (\Done B)_{i, j+1}, \quad \text{for } i \in [\dimone], \, j \in [\dimtwo-1], \quad \text{and}\label{eq:aq}\\
(B \Dtwo^\top)_{i, j} \leq {} & (B \Dtwo^\top)_{i+1, j}, \quad \text{for } i \in [\dimone-1], \, j \in [\dimtwo]. \label{eq:ar}
\end{align}
The first column and first row of \( B \) are equal to \( 0 \) by construction.
Consequently, we obtain that \( (D B)_{\cdot, 1} = 0 \), and by \eqref{eq:aq}, this gives $D B \geq 0$ and similarly $\theta \Dtwo^\top \geq 0$ by~\eqref{eq:ar}\,.
%\begin{align}
%\label{eq:ap}
%D B \geq 0.
%\end{align}
%Similarly, by \eqref{eq:ar}, we obtain
%\begin{align}
%\label{eq:as}
%\theta \Dtwo^\top \geq 0.
%\end{align}
Together, these facts yield that \( B \) is bivariate isotonic.

Additionally, from the triangle inequality and the way we constructed \( R, S, \) and \( B \), we get
\begin{align}
\label{eq:at}
\| R \|_\infty \leq \| \theta \|_\infty, \quad \| S \|_\infty \leq 2 \| \theta \|_\infty, \quad \text{and} \quad \| B \|_\infty \leq 4 \| \theta \|_\infty.
\end{align}
Since \( B \) also inherits the variation of \( \theta \) and is nonnegative, $\| B \|_\infty = V(B) = V(\theta)$, 
%\begin{align}
%\label{eq:by}
%
%\end{align}
which completes the proof.
\end{proof}

\subsection{Proof of Theorem \ref{thm:upper_bound}}

To control the performance of a least-squares estimator, we employ Chatterjee's variational formula~\cite{Cha14} that we recall below. See, e.g., Lemma 6.1 of~\cite{FlaMaoRig19} for this deterministic form.

\begin{lemma}[Chatterjee's variational formula] \label{lem:variational}
Let $\cM$ be a closed subset of ${\R}^d$. Suppose that $y = \theta^* + \eps$ where $\theta^* \in \cM$ and  $\eps \in {\R}^d$. Let $\hat \theta \in \operatorname*{argmin}_{\theta \in \cM} \|y- \theta\|_2^2$ be a projection of $y$ onto $\cM$. Define the function $f_{\theta^*}: {\R}_+ \to {\R}$ by
\[
f_{\theta^*}(t) = \sup_{\theta \in \cM , \, \|\theta - \theta^*\|_2 \le t} \langle \eps, \theta - \theta^* \rangle - \frac{t^2}2 .
\]
Then we have
\begin{align} \label{eq:variational}  
\|\hat \theta - \theta^*\|_2 \in \argmax_{t \ge 0} f_{\theta^*}(t). 
\end{align}
Moreover, if there exists $t^* > 0$ such that $f_{\theta^*}(t) < 0$ for all $t \ge t^*$, then
$\|\hat \theta - \theta^*\|_2 \le t^*.$
\end{lemma}

%It is well known~\cite{Bir83, Cha14} that the performance of a least-squares estimator is determined by the supremum of a certain empirical process. For 
%Theorem~\ref{thm:upper_bound} follows from the above the relevant supremum is controlled in expectation by the following proposition, which presents the main technical challenge and whose proof is deferred to Section \ref{sec:pf-sup-id}.
To control the supremum in Lemma~\ref{lem:variational}, note that it suffices to consider Gaussian noise here, since the generalization to sub-Gaussian noise is taken care of by Theorem \ref{thm:tail}.

%\ndpr{Why do you keep a general $k$ in the proposition below? If not necessary, present the optimized bound}

\begin{proposition} \label{prop:sup-id}
Fix an anti-Monge matrix $\theta^* \in \cM$, and suppose that $\gaussiannoise \in \R^{\dimone \times \dimtwo}$ has i.i.d. $\cN(0, 1)$ entries.
Then for any integer $k \in [\dimone \dimtwo]$ and any \( t > 0 \), we have
%$$
%\E \Big[ \sup_{\substack{\theta \in \cM\\ \|\theta - \theta^* \|_F \leq t}} \langle \varepsilon , \theta - \theta^* \rangle  \Big] \lesssim
%$$
\begin{align*}
	\E \Big[ \sup_{\substack{\theta \in \cM\\ \|\theta - \theta^* \|_F \leq t}} \langle \gaussiannoise , \theta - \theta^* \rangle \Big]
	& \lesssim t \left[ \sqrt{\dimone} + \sqrt{ k \log (\dimtwo)} + \sqrt{ \log (\dimone) \log (\dimtwo) } \left( \frac{\dimone \dimtwo}{k} \right)^{1/4} \right] \\ %+ \log (\dimone) \log (\dimtwo) \right] \\
	& \quad \  + \sqrt{\frac{\dimone \dimtwo }{k}} \sqrt{\log (\dimone) \log (\dimtwo)} \, V(\theta^*) .
\end{align*}
\end{proposition}

To show Theorem \ref{thm:upper_bound} taking Proposition \ref{prop:sup-id} as given, let \( t > 0 \) and \( 1 \le k \le \dimone \dimtwo \) to be chosen later.
Note that by Theorem~\ref{thm:maj} and Proposition~\ref{prop:sup-id}, we obtain 
\begin{align*}
	&\gamma_2 \big( \{ \theta - \theta^* : \theta \in \cM, \, \|\theta - \theta^* \|_F \leq t \} \big)
	\asymp \E_{\gaussiannoise_{i,j} \simiid \cN(0, 1) } \Big[ \sup_{\substack{\theta \in \cM\\ \|\theta - \theta^* \|_F \leq t}} \langle \gaussiannoise , \theta - \theta^* \rangle \Big] \\
	& \lesssim t \left[ \sqrt{\dimone} + \sqrt{ k \log (\dimtwo)} + \sqrt{ \log (\dimone) \log (\dimtwo) } \left( \frac{\dimone \dimtwo}{k} \right)^{1/4} \right] %+ \log (\dimone) \log (\dimtwo) \right] \\
	+ \sqrt{\frac{\dimone \dimtwo }{k}} \sqrt{\log (\dimone) \log (\dimtwo)} \, V(\theta^*),
\end{align*}
where \( \gamma_2 \) denotes Talagrand's \( \gamma_2 \) functional.
Therefore, Theorem~\ref{thm:tail} yields that with probability $1 - 4 \exp(-s^2)$, 
\begin{align}
	\sup_{\substack{\theta \in \cM\\ \|\theta - \theta^* \|_F \leq t}} \langle \eps , \theta - \theta^* \rangle 
	& \lesssim t \sigma \left[ \sqrt{\dimone} + \sqrt{ k \log (\dimtwo)} + \sqrt{ \log (\dimone) \log (\dimtwo) } \left( \frac{\dimone \dimtwo}{k} \right)^{1/4} \right]  \notag  \\ %+ \log (\dimone) \log (\dimtwo) \right] \\
	& \quad \  + \sigma \sqrt{\frac{\dimone \dimtwo }{k}} \sqrt{\log (\dimone) \log (\dimtwo)} \, V(\theta^*) 
	+ \sigma s t . \label{eq:s-prob}
\end{align}

Let us define 
\[
f_{\theta^*}(t) = \sup_{\substack{ \theta \in \cM \\ \| \theta  - \theta^* \|_F \le t }} \langle \eps,  \theta - \theta^* \rangle - \frac{t^2}2 .
\]
we obtain that for any 
\begin{align}
	\label{eq:dl}
	t > t^*_s & \defn C \sigma \left[ \sqrt{\dimone} + \sqrt{ k \log (\dimtwo) } + \sqrt{ \log (\dimone) \log (\dimtwo) } \left( \frac{\dimone \dimtwo}{k} \right)^{1/4} \right] \\ %+ \log (\dimone) \log (\dimtwo) \right] \\
	&\quad \  + C \Big[ \sigma \sqrt{\frac{\dimone \dimtwo }{k}} \sqrt{\log (\dimone) \log (\dimtwo)} \, V(\theta^*) \Big]^{1/2}
	+ C \sigma s ,
\end{align}
where \( C \) is a sufficiently large constant, it holds with probability at least $1 - 4 \exp ( - s^2 )$ that 
\begin{align}
	\label{eq:dn}
	f_{\theta^*}( t ) < 0.
\end{align}
Therefore by Lemma~\ref{lem:variational}, we obtain that with probability at least $1 - 4 \exp( - s^2 )$, 
\begin{align}
	& \frac{1}{\dimone \dimtwo} \| \thetals - \theta^* \|_F^2 
	\le \frac{(t^*)^2}{\dimone \dimtwo}  \label{eq:do} \\
	& \lesssim \sigma^2 \left[ \frac{1}{\dimtwo} + \frac{k \log (\dimtwo)}{\dimone \dimtwo}  + \frac{ \log (\dimone) \log (\dimtwo) }{\sqrt{ \dimone \dimtwo k }} \right] %+ \frac{\log^2 (\dimone) \log^2 (\dimtwo)}{\dimone \dimtwo}  \right] \\
	+ \sigma \sqrt{\frac{\log (\dimone) \log (\dimtwo)}{\dimone \dimtwo k }} \, V(\theta^*) 
	+ \sigma^2 \frac{ s^2 }{ \dimone \dimtwo } . \notag
\end{align}

We now choose $s = 2 \sqrt{\dimone}$. 
Balancing the terms that depend on \( k \) leads to the choice
\begin{align}
	\label{eq:dp}
	k^* \defn (\dimone \dimtwo)^{1/3} \log(\dimone)^{1/3} \left[ \sqrt{\log(\dimone)} + \frac{V(\theta^*)}{\sigma \sqrt{\log(\dimtwo)}} \right]^{2/3},
\end{align}
in addition to possibly rounding \( k^\ast \) to an integer which we omitted to simplify the presentation.
Therefore, we obtain that with probability $1 - \exp(- \dimone)$, if \( 1 \le k^\ast \le \dimone \dimtwo \), then 
\begin{align}
	\frac{1}{\dimone \dimtwo} \| \thetals - \theta^* \|_F^2
	\lesssim {} & \frac{\sigma^2 }{\dimtwo} + \frac{ \sigma^2 \log(\dimtwo) \log(\dimone)^{1/3} \big[ \sqrt{\log(\dimone)} + V(\theta^\ast)/ \big(\sigma \sqrt{\log(\dimtwo) } \, \big) 			\big]^{2/3}}{ (\dimone \dimtwo)^{2/3} } \notag \\
	\lesssim {} & \frac{\sigma^2 }{\dimtwo} + \left( \frac{ \sigma^2 V(\theta^*)}{ \dimone \dimtwo } \right)^{2/3}  \log(\dimone)^{1/3} \log(\dimtwo)^{2/3}.
	\label{eq:dq}
\end{align}
If \( k^\ast < 1 \), we replace it by \( 1 \), increasing the \( k \log(\dimtwo)/(\dimone \dimtwo) \) term while decreasing the ones with \( 1/\sqrt{k} \) in \eqref{eq:do}, hence leading to the same rate as in \eqref{eq:dq}.
If \( k^\ast > \dimone \dimtwo \), note that the \( k/(\dimone \dimtwo) \) term is already of the order \( \sigma^2 \), so a basic bound of \( \sigma^2 \) on the empirical process term yields the rate
\begin{equation}
	\label{eq:jl}
	\frac{1}{\dimone \dimtwo} \| \thetals - \theta^* \|_F^2
	\le \sigma^2.
\end{equation}
Combined, this yields that with probability at least $1 - \exp( - \dimone)$, 
\begin{align}
	\frac{1}{\dimone \dimtwo} \| \thetals - \theta^* \|_F^2
	\lesssim \left[ \frac{\sigma^2 }{\dimtwo} + \left( \frac{ \sigma^2 V(\theta^*)}{ \dimone \dimtwo } \right)^{2/3} \log(\dimone)^{1/3} \log(\dimtwo)^{2/3} \right] \wedge \sigma^2 .
\end{align}

To obtain the bound in expectation, we can first integrate the exponentially decaying tale of~\eqref{eq:do}, and then choose the optimal $k$ in the same way. 

%\ndpr{It is unclear that you can indeed integrate the tales since you substitute $s$ before applying chatterjee. Please rewrite.}

\subsection{Proof of Proposition \ref{prop:sup-id}}
\label{sec:pf-sup-id}

%We prove Theorem~\ref{thm:upper_bound} in this section.

%the Gaussian width
%\begin{align}
%	\label{eq:db}
%	\E \Big[ \sup_{\substack{\theta \in \cM\\ \|\theta - \theta^* \|_F \leq t}} \langle \gaussiannoise , \theta - \theta^* \rangle \Big]
%\end{align}
%for $t > 0$. 

% \subsubsection{Proof of Proposition~\ref{prop:sup-id}.}

Our main strategy consists in decomposing the noise matrix $\gaussiannoise$ into three terms according to the spectral decomposition of the linear map $\cD$, defined by $\cD(A) = \Done A \Dtwo^\top$ for $A \in \R^{\dimone \times \dimtwo}$.

\paragraph{Spectral decomposition of the difference operator.}	
Denote the (reduced) singular value decomposition of \( \Done \) by $\Done =  U \Sigma W^\top$,
%\begin{align}
%\label{eq:dc}
%	
%\end{align}
where we order the singular values in \( \Sigma \) in ascending magnitude.
%Furthermore, for $k \in [\dimone]$ define $\Sigma_{-k}$ to be the matrix of singular values the smallest \( k \) of which are set to zero, i.e.,
%\begin{align}
%	\label{eq:df}
%	(\Sigma_{-k})_{i, j} =
%	\left\{
%	\begin{aligned}
%		\Sigma_{i, i}, \quad & k + 1 \leq i = j \leq \dimone - 1,\\
%		0, \quad & 1 \leq i = j \leq k,\\
%		0, \quad & i \neq j ,
%	\end{aligned}
%	\right.
%\end{align}
%and consider its pseudo-inverse
%\begin{align}
%	\label{eq:dd}
%	(\Sigma_{-k}^\dag)_{i, j} =
%	\left\{
%	\begin{aligned}
%		\Sigma_{i, i}^{-1}, \quad & k + 1 \leq i = j \leq \dimone - 1,\\
%		0, \quad & 1 \leq i = j \leq k,\\
%		0, \quad & i \neq j .
%	\end{aligned}
%	\right.
%\end{align}
In addition, we write $W = \left[
	\begin{array}{@{}c|c|c@{}}
	w_1 & \cdots & w_{\dimone-1}
\end{array}\right]$.

	First, let \( \Pi_1 \) denote the projection onto \( \ker \cD \).
Moreover, let \( J = \{ (l, r) \in [\dimone] \times [\dimtwo] : l r \leq k \} \) and \( J^c = [\dimone] \times [\dimtwo] \setminus J \). 
Define the projection $\Pi_2$ by
\begin{align}
	\label{eq:dr}
	\Pi_2 (A) = {} & \sum_{(l, r) \in J} w_l w_l^\top A \tilde w_r \tilde w_r^\top , \text{ and so }\\
	(I - \Pi_2) (A) = {} & \sum_{(l, r) \in J^c} w_l w_l^\top A \tilde w_r \tilde w_r^\top.
\end{align}
With these two projections, we decompose
\begin{align}
	\E[\sup_{\substack{\theta \in \cM\\ \|\theta - \theta^* \|_F \leq t}} \langle \gaussiannoise , \theta - \theta^* \rangle]
	\leq {} & \E[\sup_{\substack{\theta \in \cM\\ \|\theta - \theta^* \|_F \leq t}} \langle \Pi_1 (\gaussiannoise) , \theta - \theta^* \rangle]
	+ \E[\sup_{\substack{\theta \in \cM\\ \|\theta - \theta^* \|_F \leq t}} \langle  (I - \Pi_1) \Pi_2 (\gaussiannoise) , \theta - \theta^* \rangle] \notag \\
	{} & + \E[\sup_{\substack{\theta \in \cM\\ \|\theta - \theta^* \|_F \leq t}} \langle (I - \Pi_1) (I - \Pi_2) (\gaussiannoise) , \theta - \theta^* \rangle] .
	\label{eq:cd}
\end{align}
We now bound the three terms in \eqref{eq:cd} separately.

\paragraph{Bounding the first term in \eqref{eq:cd}.}
Recall that \( \Pi_1 \) be the projection onto \( \ker \cD \). We claim that $\dim ( \ker \cD ) = \dimone + \dimtwo - 1$. 
%	It is easy to obtain the dimension of \( \ker \mathcal{D} \) from Lemma~\ref{lem:c}.
%	Although there may be simpler ways to bound this dimension, we provide the following lemma which, in addition, sheds light on the structure of the parameter space $\cM$.
Given a matrix $\theta \in \ker \cD$, i.e., $\Done \theta \Dtwo^\top = 0$, we apply Lemma~\ref{lem:c} to obtain the unique decomposition $\theta = R+S+B$, where $R \Dtwo^\top = 0$ and $\Done S = 0$. It follows that $\Done B \Dtwo^\top = 0$. Since the first column and the first row of $B$ are both identically zero, it is easy to see from an inductive argument that $B = 0$ so that $\ker \cD$ contains only matrices of the form $\theta = R+S$. The set of constant-row matrices $R$ has dimension $\dimone$; the set of constant-column matrices $S$ with $S_{i, 1} = 0$ for $i \in [\dimone]$ has dimension $\dimtwo - 1$. 
%Since every $\theta$ in $\ker \mathcal{D}$ can be uniquely decomposed as $\theta = R+S$, we conclude that 
Thus $\dim(\ker \mathcal{D}) = \dimone + \dimtwo - 1$.
Consequently, we have
\begin{align}
	\label{eq:dg}
	\E[\sup_{\substack{\theta \in \cM\\ \|\theta - \theta^* \|_F \leq t}} \langle \Pi_1 (\gaussiannoise) , \theta - \theta^* \rangle]
	\leq t \, \E[ \| \Pi_1 (\gaussiannoise) \|_F ] \le t \sqrt{n_1+n_2-1}\lesssim t \sqrt{\dimone} .
\end{align}

\paragraph{Bounding the second term in \eqref{eq:cd}.}
Similarly, it suffices to compute the rank of $\Pi_2$, which is bounded as follows
\begin{align}
	\label{eq:dw}
	|J| = \sum_{(l, r) \in J} 1
	\le {} & \sum_{r=1}^{\dimtwo}  k/r 
	\le k \log(\dimtwo).
\end{align}
Therefore, we obtain
\begin{align}
	\E[\sup_{\substack{\theta \in \cM\\ \|\theta - \theta^* \|_F \leq t}} \langle  (I - \Pi_1) \Pi_2 (\gaussiannoise) , \theta - \theta^* \rangle]
	\leq {} & t \E[ \| (I - \Pi_1) \Pi_2 (\gaussiannoise) \|_F ]
	\lesssim t \sqrt{k \log(\dimtwo)} .  \notag 
\end{align}

\paragraph{Bounding the third term in \eqref{eq:cd}.}
Note that $I - \Pi_1$ is the projection onto the image of the linear map $\mathcal{D}^\top$, defined by $\mathcal{D}^\top (A) = \Done^\top A \Dtwo$. Hence we have
\begin{align}
	\label{eq:ce}
	\langle (I - \Pi_1) (I - \Pi_2) (\gaussiannoise) , \theta - \theta^* \rangle
	= {} & \langle \Done^\top (\Done^\top)^\dag (I - \Pi_2) (\gaussiannoise) \Dtwo^\dag \Dtwo, \theta - \theta^* \rangle \\
	= {} & \langle (\Done^\dag)^\top (I - \Pi_2) (\gaussiannoise) \Dtwo^\dag, \Done (\theta - \theta^*) \Dtwo^\top \rangle . 
%		= {} & \langle (I - \Pi_2) (\gaussiannoise) , (I - \Pi_1) \theta - \theta^* \rangle\\
%		= {} & \langle (I - \Pi_2) (\gaussiannoise) , \mathcal{D}^{\dag} \mathcal{D} (\theta - \theta^*) \rangle\\
%		= {} & \langle (\mathcal{D}^{\dag})^{\top}((I - \Pi_2) (\gaussiannoise)) , \mathcal{D} (\theta - \theta^*) \rangle\\
%		= {} & \langle (D^\dag)^\top((I - \Pi_2) (\gaussiannoise)) D^\dag, D (\theta - \theta^*) D^\top \rangle,
\end{align}
%	Setting
%	\begin{align}
%		\label{eq:du}
%		\tilde \gaussiannoise_{l, r} = w_l^\top \gaussiannoise \tilde w_r
%	\end{align}
%	which is $\subG(\sigma^2)$,
%	we see that \( \tilde \gaussiannoise \) is again composed of independent Gaussians
Since \( \gaussiannoise \) has mean zero, it is sufficient to control
\begin{align}
%	\E[\sup_{\substack{\theta \in \cM\\ \|\theta - \theta^* \|_F \leq t}} \langle (I - \Pi_1) (I - \Pi_2) (\gaussiannoise) , \theta - \theta^* \rangle]
%	= 
	\E[\sup_{\substack{\theta \in \mathcal{M}\\ \| \theta - \theta^* \|_F \leq t}}\langle (D^\dag)^\top (I - \Pi_2) (\gaussiannoise) \Dtwo^\dag, D \theta \Dtwo^\top \rangle]. \label{eq:s1}
\end{align}
To bound this quantity, we need the following lemma, whose proof is deferred to Section \ref{sec:pf-var-bd}.
	
	\begin{lemma} \label{lem:var-bd}
	For any $i \in [\dimone], \, j \in [\dimtwo]$, the quantity
	$
	[ (D^\dag)^\top (I - \Pi_2) (\gaussiannoise) (\tilde D^\dag) ]_{i, j}
	$
	is sub-Gaussian with variance proxy
	$$
	O \Big( \log (\dimtwo) \left[ \big( i \wedge (\dimone - i) \big) \big( j \wedge (\dimtwo - j) \big)  \wedge \frac{\dimone \dimtwo}{k} \right] \Big) .
	$$
	\end{lemma}
	
	Let us define $\Phi \in \R^{(\dimone - 1) \times (\dimtwo - 1)}$ by
		\begin{align}
			\label{eq:cp}
			\Phi_{i, j} = \sqrt{\log(\dimone) \log (\dimtwo)} \left[ \left( \sqrt{i} \wedge \sqrt{\dimone - i} \right) \left( \sqrt{j} \wedge \sqrt{\dimtwo - j} \right) \wedge \sqrt{ \frac{\dimone \dimtwo}{k} } \, \right] ,
		\end{align}
	and let \( \oslash \) denote element-wise division.
%	Note that there is an extra $\sqrt{\log(\dimone)}$ factor in the definition of $\Phi$ against the sub-Gaussian variance proxy from 
	Lemma~\ref{lem:var-bd}, together with a union bound  readily yields
	\begin{align}
		\label{eq:cz}
		\E[\| (D^\dag)^\top (I - \Pi_2) (\gaussiannoise) \Dtwo^\dag \oslash \Phi \|_\infty] \lesssim 1.
	\end{align}
							
	In addition, it holds for every \( \theta \) that
	\begin{align}
%		\label{eq:cj}
		\langle (D^\dag)^\top (I - \Pi_2) (\gaussiannoise) \Dtwo^\dag, D \theta \Dtwo^\top \rangle
		= {} & \langle (D^\dag)^\top (I - \Pi_2) (\gaussiannoise) \Dtwo^\dag \oslash \Phi, \Phi \odot D \theta \Dtwo^\top \rangle  \notag \\
		\leq {} & \| (D^\dag)^\top (I - \Pi_2) (\gaussiannoise) \Dtwo^\dag \oslash \Phi \|_\infty \| \Phi \odot D \theta \Dtwo^\top \|_1  \notag 
	\end{align}
	by H\"older's inequality. 
%		, i.e.,
%	\begin{align}
%		\label{eq:dj}
%		(A \oslash B)_{i, j} = \frac{A_{i,j}}{B_{i, j}}.
%	\end{align}
We therefore obtain
\begin{align}
	& \E[\sup_{\substack{\theta \in \mathcal{M}\\ \| \theta - \theta^* \|_F \leq t}}\langle (D^\dag)^\top (I - \Pi_2) (\gaussiannoise) \Dtwo^\dag, D (\theta - \theta^*) \Dtwo^\top \rangle] \\
	& \le \E[\| (D^\dag)^\top (I - \Pi_2) (\gaussiannoise) \Dtwo^\dag \oslash \Phi \|_\infty] \sup_{\substack{\theta \in \mathcal{M}\\ \| \theta - \theta^* \|_F \leq t}} \| \Phi \odot D \theta \Dtwo^\top \|_1 \\
	&\lesssim \sup_{\substack{\theta \in \mathcal{M}\\ \| \theta - \theta^* \|_F \leq t}} \| \Phi \odot (D \theta \Dtwo^\top) \|_1 . \quad \label{eq:s3}
\end{align}

It remains to bound this supremum. For \( \theta \in \mathcal{M} \), because \( D \theta \Dtwo^\top \geq 0 \) we can write
\begin{align}
	\| \Phi \odot (D \theta \Dtwo^\top) \|_1
%		= {} & \langle \AOone \AOtwo^\top , \Phi \odot (D \theta \Dtwo^\top) \rangle\\
%		= {} & \langle \Phi \odot (\AOone \AOtwo^\top) , D \theta \Dtwo^\top \rangle\\
	= \langle \Phi , D \theta \Dtwo^\top \rangle
	= \langle \Phi , D (\theta - \theta^*) \Dtwo^\top \rangle + \langle \Phi , D \theta^* \Dtwo^\top \rangle .
	\label{eq:cl}
\end{align}
The second term in \eqref{eq:cl} can be bounded by 
\begin{align}
	\label{eq:cm}
	\langle \Phi , D \theta^* \Dtwo^\top \rangle
	\leq \| \Phi \|_\infty \| D \theta^* \Dtwo^\top \|_1
	\lesssim \sqrt{\frac{\dimone \dimtwo }{k}} \sqrt{\log (\dimone) \log (\dimtwo)} \, V(\theta^*).
\end{align}
For the first term in \eqref{eq:cl}, we need the following lemma, whose proof is deferred to Section \ref{sec:pf-phi-fro}.

\begin{lemma}
	\label{lem:phi-fro}
We have the estimate
$$
\| D^\top \Phi \Dtwo \|_F^2 \lesssim \log (\dimone) \log (\dimtwo) \sqrt{ \frac{\dimone \dimtwo}{k} }  + \log^2 (\dimone) \log^2 (\dimtwo) .
$$
\end{lemma}
\noindent If $\|\theta - \theta^*\|_F \le t$, then the above lemma together with the Cauchy-Schwarz inequality yields
\begin{align}
	\label{eq:cn}
	{}& \langle \Phi , D (\theta - \theta^*) \Dtwo^\top \rangle \\
	= {} & \langle D^\top \Phi \Dtwo , \theta - \theta^* \rangle
	\leq \| D^\top \Phi \Dtwo \|_F \| \theta - \theta^* \|_F \\
	\lesssim {} & t \left[ \sqrt{ \log (\dimone) \log (\dimtwo) } \left( \frac{\dimone \dimtwo}{k} \right)^{1/4}  + \log (\dimone) \log (\dimtwo) \right] .
\end{align}

Combining~\eqref{eq:s1}--\eqref{eq:cn}, we conclude that
\begin{align} 
& \E[\sup_{\substack{\theta \in \cM\\ \|\theta - \theta^* \|_F \leq t}} \langle (I - \Pi_1) (I - \Pi_2) (\gaussiannoise) , \theta - \theta^* \rangle]  \notag \\
& \lesssim \sqrt{\frac{\dimone \dimtwo }{k}} \sqrt{\log (\dimone) \log (\dimtwo)} \, V(\theta^*) 
+ t \left[ \sqrt{ \log (\dimone) \log (\dimtwo) } \left( \frac{\dimone \dimtwo}{k} \right)^{1/4}  + \log (\dimone) \log (\dimtwo) \right] . \notag 
\end{align}

%\paragraph{Finishing the proof of Proposition~\ref{prop:sup-id}.}
%
The bounds on the three terms of \eqref{eq:cd} together yield the desired result.

\subsection{Proof of Lemma~\ref{lem:var-bd}}
\label{sec:pf-var-bd}

By  definition of $\Pi_2$, it holds that
\begin{align}
	\label{eq:dt}
	(\Done^\dag)^\top (I - \Pi_2) (\gaussiannoise) \Dtwo^\dag
	= {} & \sum_{(\ell, r) \in J^c} U \Sigma^\dag W^\top w_\ell w_\ell^\top \gaussiannoise \tilde w_r \tilde w_r^\top \tilde W \tilde \Sigma^\dag \tilde U^\top\\
	= {} & \sum_{(\ell, r) \in J^c} ( w_\ell^\top \gaussiannoise \tilde w_r )  U \Sigma^\dag e_\ell \tilde e_r^\top \tilde \Sigma^\dag \tilde U^\top\\
	= {} & \sum_{(\ell, r) \in J^c} ( w_\ell^\top \gaussiannoise \tilde w_r ) \Sigma_{\ell, \ell}^{-1} \tilde \Sigma_{r, r}^{-1} U e_\ell \tilde e_r^\top \tilde U^\top\\
	= {} & \sum_{(\ell, r) \in J^c} ( w_\ell^\top \gaussiannoise \tilde w_r ) \Sigma_{\ell, \ell}^{-1} \tilde \Sigma_{r, r}^{-1} U_{\cdot, \ell} \tilde U_{\cdot, r}^\top . \label{eq:t1}
\end{align}
We now study the sub-Gaussianity of the $(i, j)$-th entry of this quantity.
Since $\gaussiannoise$ has i.i.d. $\cN(0, 1)$ entries, it holds for each $\lambda > 0$ that
\begin{align}
\E \exp \Big( \lambda \sum_{(\ell, r) \in J^c} ( w_\ell^\top \gaussiannoise \tilde w_r ) \Sigma_{\ell, \ell}^{-1} \tilde \Sigma_{r, r}^{-1} U_{i, \ell} \tilde U_{j, r} \Big)
&= \E \exp \Big( \lambda \sum_{(\ell, r) \in J^c} \tr(\gaussiannoise \tilde w_r w_\ell^\top ) \Sigma_{\ell, \ell}^{-1} \tilde \Sigma_{r, r}^{-1} U_{i, \ell} \tilde U_{j, r} \Big)  \notag \\
&= \E \exp \Big\{ \tr \Big[ \gaussiannoise  \Big( \lambda \sum_{(\ell, r) \in J^c}  \Sigma_{\ell, \ell}^{-1} \tilde \Sigma_{r, r}^{-1} U_{i, \ell} \tilde U_{j, r} \tilde w_r  w_\ell^\top \Big) \Big] \Big\} \\
&\le \exp \Big\{ \frac{\lambda^2}{2} \Big\| \sum_{(\ell, r) \in J^c}  \Sigma_{\ell, \ell}^{-1} \tilde \Sigma_{r, r}^{-1} U_{i, \ell} \tilde U_{j, r} \tilde w_r  w_\ell^\top \Big\|_F^2 \Big\} . \label{eq:t2}
\end{align}
Note that $ \| \tilde w_r  w_\ell^\top \|_F = 1$, and	
$\langle \tilde w_r  w_\ell^\top, \tilde w_{r'}  w_{\ell'}^\top \rangle = 0$ for any pairs $(r, \ell) \ne (r', \ell')$, so we have
\begin{align}
\Big\| \sum_{(\ell, r) \in J^c}  \Sigma_{\ell, \ell}^{-1} \tilde \Sigma_{r, r}^{-1} U_{i, \ell} \tilde U_{j, r} \tilde w_r  w_\ell^\top \Big\|_F^2 = \sum_{(\ell, r) \in J^c} \Sigma_{\ell, \ell}^{-2} \tilde \Sigma_{r, r}^{-2} U_{i, \ell}^2 \tilde U_{j, r}^2. \label{eq:t3}
\end{align}

It remains to bound this quantity. 
Without loss of generality, assume  that \( \dimone \) is odd, so \( \dimone - 1 \) is even.
The matrix \( \Done \) has the same left-singular vectors as
	\begin{align}
		\label{eq:ct}
		\Done \Done^\top =
		\begin{bmatrix}
			2 & -1 & 0 & \dots & 0 & 0\\
			-1 & 2 & -1 & \dots & 0 & 0\\
			\vdots & & \ddots & & \vdots \\
			0 & 0 & \dots & -1 & 2 & -1\\
			0 & 0 & \dots & 0 & -1 & 2
		\end{bmatrix},
	\end{align}
	which are known~\cite{Str07} to be
	\begin{align}
		\label{eq:cu}
		U_{i, j} = \sqrt{\frac{2}{\dimone}} \sin \left( \frac{\pi i j}{\dimone} \right), \quad i, j = 1, \dots, \dimone-1.
	\end{align}
	Moreover, the matrix $D$ has (non-zero) singular values
	\begin{align}
		\label{eq:cv}
		\Sigma_{i, i} = 2 \left| \sin \left( \frac{\pi i}{2 \dimone} \right) \right|, \quad i = 1, \dots, \dimone-1 .
	\end{align}
	Note that because of the symmetry
	\begin{align}
		\label{eq:da}
		\sin \left( \frac{\pi i j}{\dimone} \right)
		= \sin \left( \frac{\pi j (\dimone - i)}{\dimone} \right), \quad i = 1, \dots, \dimone-1,
	\end{align}
	it is enough to consider \( i = 1, \dots, \frac{\dimone-1}{2} \).
	We make use of the following inequalities to control the \( \sin \) terms involved:
%	in these expressions: 
%	\( \sin(x) \leq x \) for \( x \in [0, \infty) \), \( | \sin(x) | \leq 1 \) for all \( x \), \( \sin(x) \geq \frac{2}{\pi} x \geq \frac{1}{2} x\) for \( x \in [0, 2/\pi] \)
	\begin{alignat}{2}
		| \sin(x) | \leq {} & 1, \quad && \text{for all } x \in \R; \label{eq:di}
\\
		\sin(x) \leq {} & x, \quad && \text{for } x \in [0, \infty); \label{eq:jp}\\
		\sin(x) \geq {} & \frac{2}{\pi} x \geq \frac{1}{2} x, \quad && \text{for } x \in [0, \frac{\pi}{2}]. \label{eq:jq}
	\end{alignat}
	
Plugging in the entries of $U$ and $\Sigma$ yields
\begin{align}
	\sum_{(\ell, r) \in J^c} \Sigma_{\ell, \ell}^{-2} \tilde \Sigma_{r, r}^{-2} U_{i, \ell}^2 \tilde U_{j, r}^2
	= {} & \sum_{(\ell, r) \in J^c} \frac{4 \sin \left( \frac{\pi i \ell}{\dimone} \right)^2 \sin \left( \frac{\pi j r}{\dimtwo} \right)^2}{16 \dimone \dimtwo \sin \left( \frac{\pi \ell}{2 \dimone} \right)^2 \sin \left( \frac{\pi r}{2 \dimtwo} \right)^2} \notag \\
	\stackrel{(i)}{\lesssim} {} & \frac{1}{\dimone \dimtwo} \sum_{(\ell, r) \in J^c} \frac{\dimone^2 \dimtwo^2}{\ell^2 r^2}\\
	\lesssim {} & \dimone \dimtwo \sum_{r=1}^{\dimtwo} \left( \frac{1}{r^2} \sum_{\ell = \lceil k/r \rceil}^{\dimone} \frac{1}{\ell^2} \right)\\
	\lesssim {} & \dimone \dimtwo \sum_{r=1}^{\dimtwo} \left( \frac{1}{r^2} \sum_{\ell = \lceil k/r \rceil + 1}^{\dimone} \frac{1}{\ell^2}\right) + \dimone \dimtwo \sum_{r=1}^{\dimtwo} \frac{1}{r^2} \frac{1}{\lceil k/r \rceil^2}\\
	\stackrel{(ii)}{\lesssim} {} & \dimone \dimtwo \sum_{r=1}^{\dimtwo} \frac{1}{r^2} \frac{r}{k} + \dimone \dimtwo \sum_{r=1}^{k} \frac{1}{k^2} + \dimone \dimtwo \sum_{r=k+1}^{\dimtwo} \frac{1}{r^2} \\
	\stackrel{(iii)}{\lesssim} {} & \frac{\dimone \dimtwo}{k} \log (\dimtwo) + \frac{\dimone \dimtwo}{k} + \frac{\dimone \dimtwo}{k} \lesssim \frac{\dimone \dimtwo}{k} \log (\dimtwo) , \label{eq:k-split}
%		\lesssim {} & \sum_{\ell=1}^{n} \left( \frac{n}{\ell^2} \int_{\lceil k/\ell \rceil /n}^{1} \frac{1}{x^2} d x \right) + \sum_{\ell=1}^{n} \frac{n^2}{k^2} \\
%		\lesssim {} & \frac{n^2}{k} \sum_{\ell = 1}^{n} \frac{1}{\ell} + \frac{n^3}{k^2}\\
%		\lesssim {} & \frac{n^2}{k} \log n.
\end{align}
where we used \eqref{eq:di} on the numerator and \eqref{eq:jq} on the denominator in \( (i) \) and the bound $\sum_{r=k+1}^\infty \frac{1}{r^2} \le \frac{1}{k}$ for any $k \ge 1$ in \( (ii) \) and \( (iii) \).

On the other hand, 
%	for $i \le \frac{\dimone - 1}{2}$ and $j \le \frac{\dimtwo - 1}{2}$, 
even without using the constraint $(\ell, r) \in J^c$, 
we have
\begin{align}
	\sum_{(\ell, r) \in J^c} \Sigma_{\ell, \ell}^{-2} \tilde \Sigma_{r, r}^{-2} U_{i, \ell}^2 \tilde U_{j, r}^2
	\lesssim {} & \sum_{\ell r \le \frac{\dimone \dimtwo}{i j} } \frac{\sin \left( \frac{\pi i \ell}{\dimone} \right)^2 \sin \left( \frac{\pi j r}{\dimtwo} \right)^2}{\dimone \dimtwo \sin \left( \frac{\pi \ell}{2 \dimone} \right)^2 \sin \left( \frac{\pi r}{2 \dimtwo} \right)^2}
	+ \sum_{\ell r > \frac{\dimone \dimtwo}{i j} } \frac{\sin \left( \frac{\pi i \ell}{\dimone} \right)^2 \sin \left( \frac{\pi j r}{\dimtwo} \right)^2}{\dimone \dimtwo \sin \left( \frac{\pi \ell}{2 \dimone} \right)^2 \sin \left( \frac{\pi r}{2 \dimtwo} \right)^2}  \notag  \\
	\stackrel{(i)}{\lesssim} {} & \sum_{\ell r \le \frac{\dimone \dimtwo}{i j}} \frac{ (i \ell j r)^2 }{ \dimone \dimtwo (\ell r)^2 } + \frac{\dimone \dimtwo i j}{\dimone \dimtwo} \log (\dimtwo)  \notag \\
	\lesssim {} & \sum_{\ell r \le \frac{\dimone \dimtwo}{i j}} \frac{ (i j )^2 }{ \dimone \dimtwo } + i j \log (\dimtwo)  \notag \\
	\stackrel{(ii)}{\lesssim} {} & \frac{ (i j )^2 }{ \dimone \dimtwo } \frac{\dimone \dimtwo}{i j} \log (\dimtwo) + i j \log (\dimtwo) \lesssim i j \log (\dimtwo),  \label{eq:t5}
\end{align}
where in \( (i) \), we used \eqref{eq:jp} for the numerator and \eqref{eq:jq} for the denominator as well as \eqref{eq:k-split} with $k$ replaced by $\frac{\dimone \dimtwo}{ i j }$, and \( (ii) \) follows from counting integer points in the set $\{(\ell, r): \ell r \le \frac{\dimone \dimtwo}{i j}\}$ as in \eqref{eq:dw}.

A similar argument yields bounds with $i$ replaced by $\dimone - i$, or $j$ replaced by $\dimtwo - j$. Combining this observation with~\eqref{eq:t1},~\eqref{eq:t2},~\eqref{eq:t3},~\eqref{eq:k-split} and~\eqref{eq:t5} completes the proof.

\subsection{Proof of Lemma~\ref{lem:phi-fro}}
\label{sec:pf-phi-fro}

Define $\Phi' = \frac{\Phi}{\sqrt{\log(\dimone) \log(\dimtwo)}}$, i.e.,
$$
\Phi'_{i, j} = \left( \sqrt{i} \wedge \sqrt{\dimone - i} \right) \left( \sqrt{j} \wedge \sqrt{\dimtwo - j} \right) \wedge \sqrt{ \frac{\dimone \dimtwo}{k} } .
$$
To simplify the notation, let $\Phi'_{i, 0} = \Phi'_{i, \dimtwo} = \Phi'_{0, j} = \Phi'_{\dimone, j} = 0$ for all $i \in [\dimone], \, j \in [\dimtwo]$.
We need to bound
$$
\| \Done^\top \Phi' \Dtwo \|_F^2 = \sum_{i=1}^{\dimone} \sum_{j=1}^{\dimtwo} ( \Phi'_{i-1, j-1} + \Phi'_{i, j} - \Phi'_{i-1, j} - \Phi'_{i, j-1} )^2 . 
$$ 
By symmetry, it suffices to consider summation over $i \in [\frac{\dimone - 1}{2}], \, j \in [\frac{\dimtwo-1}{2}]$ where $\Phi' = \sqrt{i j} \land \sqrt{\frac{\dimone \dimtwo}{k}}$. Moreover, note that the summand vanishes if $(i-1)(j-1) \ge \frac{\dimone \dimtwo}{k}$. Hence we can further split the sum into two parts: 
\begin{enumerate}
\item
over $\{ i \in [\frac{\dimone - 1}{2}], \, j \in [\frac{\dimtwo-1}{2}] : (i-1)(j-1) < \frac{\dimone \dimtwo}{k} < i j \}$, and
\item
over $\{ i \in [\frac{\dimone - 1}{2}], \, j \in [\frac{\dimtwo-1}{2}] : i j \le \frac{\dimone \dimtwo}{k} \}$.
\end{enumerate}

To bound the first part of the sum, first consider the case $i \le j$, where we have
$$
(i-1)(j-1) < (i-1) j \le i (j-1) < i j .
$$
Adjusting signs to ensure both differences in the following expression are positive, it is easily checked that
\begin{align*}
\left( \Phi'_{i-1, j-1} + \Phi'_{i, j} - \Phi'_{i-1, j} - \Phi'_{i, j-1} \right)^2 \le \left[ \sqrt{(i-1) j} - \sqrt{(i-1)(j-1)} + \sqrt{i j} - \sqrt{i (j-1)} \right]^2 
%	& = \left( \sqrt{(i-1)(j-1)} \land \sqrt{\frac{\dimone \dimtwo}{k}} + \sqrt{i j} \land \sqrt{\frac{\dimone \dimtwo}{k}} - \sqrt{i (j-1)}  \land \sqrt{\frac{\dimone \dimtwo}{k}} - \sqrt{(i-1) j}  \land \sqrt{\frac{\dimone \dimtwo}{k}} \, \right)^2
\lesssim i/j .
\end{align*}
By the conditions $(i-1)(j-1) < \frac{\dimone \dimtwo}{k} < i j$ and $i \le j$, we obtain
\begin{align*}
& i < \sqrt{\frac{\dimone \dimtwo}{k}} + 1, \\
& \frac{\dimone \dimtwo}{k i} < j < \frac{\dimone \dimtwo}{k (i-1)} + 1 , \text{ and} \\
& i/j < \frac{i^2 k}{\dimone \dimtwo} . 
\end{align*}
Therefore, it remains to bound
\begin{align*}
\sum_{i=1}^{\sqrt{ \frac{\dimone \dimtwo}{k} } } \sum_{j = \frac{\dimone \dimtwo}{k i}}^{\frac{\dimone \dimtwo}{k (i-1)}} \frac{i^2 k}{\dimone \dimtwo}
\lesssim \sum_{i=1}^{\sqrt{ \frac{\dimone \dimtwo}{k} } } \frac{\dimone \dimtwo}{k i^2} \frac{i^2 k}{\dimone \dimtwo} \le \sqrt{ \frac{\dimone \dimtwo}{k} } .
\end{align*}
An analogous argument yields the same bound for the case $i > j$.

Next, we consider the sum of $\big( \Phi'_{i-1, j-1} + \Phi'_{i, j} - \Phi'_{i-1, j} - \Phi'_{i, j-1} \big)^2$ over $(i, j)$ such that $i j \le \frac{\dimone \dimtwo}{k}$, where we have
\begin{align}
	\left( \Phi'_{i-1, j-1} + \Phi'_{i, j} - \Phi'_{i-1, j} - \Phi'_{i, j-1} \right)^2 
	& = \left( \sqrt{(i-1)(j-1)} + \sqrt{i j} - \sqrt{(i-1) j} - \sqrt{i (j-1)} \right)^2  \notag \\
	& = \left( \sqrt{i} - \sqrt{i-1} \right)^2 \left( \sqrt{j} - \sqrt{j-1} \right)^2 .  \notag 
\end{align}
Now even summing over all indices $i \in [\dimone], \, j \in [\dimtwo]$ yields
\begin{align}
	\sum_{i=1}^{\dimone} \sum_{j = 1}^{\dimtwo} \left( \sqrt{i} - \sqrt{i-1} \right)^2 \left( \sqrt{j} - \sqrt{j-1} \right)^2
	\lesssim \sum_{i=1}^{\dimone} \sum_{j = 1}^{\dimtwo} \frac{1}{i} \frac{1}{j} 
	\lesssim \log (\dimone) \log (\dimtwo).
\end{align}

Combining all the pieces, we obtain
\begin{align}
\| \Done^\top \Phi \Dtwo \|_F^2 
&= \log (\dimone) \log (\dimtwo) \| \Done^\top \Phi' \Dtwo \|_F^2 \\
&\lesssim \log (\dimone) \log (\dimtwo) \left( \sqrt{ \frac{\dimone \dimtwo}{k} }  + \log (\dimone) \log (\dimtwo) \right) .
\end{align}

\subsection{Proof of Theorem \ref{thm:lower_bound}}

%Before we begin the proof, we state Assouad's Lemma, which is our main tool for proving Theorem~\ref{thm:lower_bound}.

%We prove Theorem~\ref{thm:lower_bound} in this section. 
First, consider the set $\cC$ of constant-row matrices, which is a subset of $\cM$. Note that $V(\theta^*) = 0$ for any $\theta^* \in \cC$. Since there are $\dimone$ rows, it is not hard to see that the minimax rate of estimation over $\cC$ is $\sigma^2 \, \dimone$ in the squared Frobenius norm. The lower bound of order $\sigma^2/\dimtwo$ then follows from normalization by \( \dimone \dimtwo \).

Next, we turn to the second term in the lower bound, which is based on Assouad's Lemma (Theorem \ref{thm:assouad}).
To this end, we construct an embedding of the hypercube into \( \cM_{V_0} \).

Consider integers $k_1 \in [\dimone]$ and $k_2 \in [\dimtwo]$, and let $m_1 = \dimone/k_1$ and $m_2 =  \dimtwo/k_2$.
Assume without loss of generality that \( m_1 \) and \( m_2 \) are integer.
Denote the elements of the hypercube $\{-1, 1\}^{k_1 \times k_2}$ by $(\tau_{u,v}: (u, v)\in [k_1]\times [k_2])$. For each $\tau\in\{-1,1\}^{k_1 \times k_2}$,  define $\theta^\tau\in\mathcal M$ in the following way. For $i\in[\dimone]$ and $j\in[\dimtwo]$, first  identify the unique $u\in[k_1],v\in[k_2]$ for which $(u-1)m_1<i\leq um_1$ and $(v-1)m_2<j\leq vm_2$, and then  take
$$\theta_{i, j}^\tau = V_0\left(\frac{uv}{2k_1k_2} + \frac{\tau_{u,v}}{8k_1k_2}\right).$$

First, we check that $\theta^\tau\in\mathcal M$ and that $V(\theta^\tau)\leq V_0$. For the former, it is enough to check that for each $1\leq i < \dimone, 1\leq j <\dimtwo$, we have that $\theta^\tau_{i, j} + \theta^\tau_{i+1,j+1}-\theta^\tau_{i,j+1}-\theta^\tau_{i+1,j} \geq 0$.
We distinguish the following two cases:
\begin{enumerate}
\item 
There exists $u\in[k_1]$ such that $(u-1)m_1<i<i+1\leq um_1$, or there exists $v\in[k_2]$ such that $(v-1)m_2<j<j+1\leq vm_2$. Then, either $\theta^\tau_{i, j} = \theta^\tau_{i+1,j}, \,  \theta^\tau_{i,j+1}=\theta^\tau_{i+1,j+1}$ or $\theta^\tau_{i, j} = \theta^\tau_{i,j+1}, \,  \theta^\tau_{i+1,j}=\theta^\tau_{i+1,j+1}$ respectively. In both cases, the difference above is $0$.

\item 
There exist $u\in[k_1],v\in[k_2]$ such that $(u-1)m_1<i\leq um_1<i+1\leq (u+1)m_1$ and $(v-1)m_2<j\leq vm_2<j+1\leq (v+1)m_2$. In this case, for any \( \tau \), we have
\begin{align*}
&\quad \  \theta^\tau_{i, j} + \theta^\tau_{i+1,j+1}-\theta^\tau_{i,j+1}-\theta^\tau_{i+1,j} \\
&= V_0\left(\frac{uv + (u+1)(v+1) - u(v+1)-(u+1)v}{2k_1k_2}+\frac{\tau_{u,v}+\tau_{u+1,v+1}-\tau_{u,v+1}-\tau_{u+1,v}}{8k_1k_2}\right) \\
& \geq V_0\left(\frac1{2k_1k_2} - \frac4{8k_1k_2} \right)= 0.
\end{align*}
\end{enumerate}
Thus, $\theta^\tau\in\mathcal M$. We now check that $V(\theta^\tau)\leq V_0$. Note that $V(\theta^\tau)$ can be written as the sum $\sum_{i,j}(\theta^\tau_{i, j} + \theta^\tau_{i+1,j+1}-\theta^\tau_{i,j+1}-\theta^\tau_{i+1,j})$. As we have seen above, this sum is nonzero in only $(k_1-1)(k_2-1)$ cases, and it equals exactly
\begin{align}
V(\theta^\tau) &= \sum_{u\in[k_1-1],v\in[k_2-1]} V_0\left(\frac1{2k_1k_2}+\frac{\tau_{u,v}+\tau_{u+1,v+1}-\tau_{u,v+1}-\tau_{u+1,v}}{8k_1k_2}\right) \\
&= V_0\left(\frac{(k_1-1)(k_2-1)}{2k_1k_2} + \frac{\tau_{11} + \tau_{k_1k_2} - \tau_{1k_2} - \tau_{k_11}}{8k_1k_2}\right) \\
&\leq V_0\left(\frac{(k_1-1)(k_2-1)}{2k_1k_2} + \frac4{8k_1k_2}\right) = V_0\frac{(k_1-1)(k_2-1)+1}{2k_1k_2}\leq V_0.
\end{align}

We now proceed to show our lower bound using Theorem~\ref{thm:assouad}, which states that
$$\underset{\tilde\theta}{\inf}\underset{\theta\in \cM_{V_0}}{\sup} R(\tilde\theta, \theta^*)\geq \frac d8 \, \underset{\tau\neq\tau'}{\min}\frac{\ell^2(\theta^{\tau}, \theta^{\tau'})}{d_H(\tau, \tau')} \, \underset{d_H(\tau, \tau') = 1}{\min}(1 - \|\mathbb P_{\theta^\tau} - \mathbb P_{\theta^{\tau'}}\|_{TV}),$$
where $d_H$ denotes the Hamming distance and  $\ell^2 (\theta, \theta') = \frac{1}{\dimone \dimtwo} \|\theta - \theta'\|_F^2$.
Note that
\begin{align*}
\ell^2(\theta^\tau, \theta^{\tau'}) &=\frac1{\dimone\dimtwo} \sum_{i,j}(\theta^\tau_{i, j}-\theta^{\tau'}_{i, j})^2 \\
&= \frac1{\dimone\dimtwo}\sum_{u\in[k_1],v\in[k_2]} \sum_{(u-1)m_1<i\leq um_1}\sum_{(v-1)m_2<j\leq vm_2}(\theta^{\tau}_{i, j} - \theta^{\tau'}_{i,j})^2 \\
&=\frac{V_0^2}{\dimone\dimtwo}\sum_{u\in[k_1],v\in[k_2]} \sum_{(u-1)m_1<i\leq um_1}\sum_{(v-1)m_2<j\leq vm_2}\left(\frac{\tau_{u,v} - \tau'_{u,v}}{8k_1k_2}\right)^2 \\
&= \frac{V_0^2}{\dimone\dimtwo}\sum_{u\in[k_1],v\in[k_2]}\frac{m_1m_2(\tau_{u,v}-\tau'_{u,v})^2}{64k_1^2k_2^2}=\frac{V_0^2m_1m_2}{16\dimone\dimtwo k_1^2k_2^2}d_H(\tau, \tau') = \frac{V_0^2}{16k_1^3k_2^3}d_H(\tau, \tau').
\end{align*}
Thus, we have
$$\underset{\tau\neq\tau'}{\min}\frac{\ell^2(\theta^{\tau}, \theta^{\tau'})}{d_H(\tau, \tau')} = \frac{V_0^2}{16k_1^3k_2^3}.$$

To bound $\|\mathbb P_{\theta^\tau} - \mathbb P_{\theta^{\tau'}}\|_{TV}$, we use Pinsker's inequality:
% because the KL divergence $D(\mathbb P_{\theta^\tau}\|\mathbb P_{\theta^{\tau'}})$ has a simple expression in terms of $\ell^2(\theta^\tau, \theta^{\tau'})$:
$$\|\mathbb P_{\theta^\tau} - \mathbb P_{\theta^{\tau'}}\|^2_{TV}\leq \frac12D(\mathbb P_{\theta^\tau}\|\mathbb P_{\theta^{\tau'}}) = \frac {\dimone\dimtwo}{4\sigma^2}\ell^2(\theta^\tau, \theta^{\tau'}) = \frac{V_0^2 \dimone \dimtwo}{64k_1^3k_2^3\sigma^2}d_H(\tau, \tau').$$
It follows that
$$\underset{d_H(\tau, \tau')=1}{\min}(1-\|\mathbb P_{\theta^\tau} - \mathbb P_{\theta^{\tau'}}\|_{TV})\geq\left(1-\frac{V_0}{8 \sigma} \sqrt{ \frac{\dimone \dimtwo }{ k_1^3k_2^3 } } \right).$$
Putting things together in Assouad's lemma, we obtain
$$\underset{\tilde\theta}{\inf}\underset{\theta\in \cM_{V_0}}{\sup} R(\tilde\theta, \theta^*)\geq 
%\frac{k_1k_2V_0^2}{8\cdot 16k_1^3k_2^3}\left(1- \frac{V_0}{8 \sigma} \sqrt{ \frac{\dimone \dimtwo }{ k_1^3k_2^3 } } \right) =
 \frac{V_0^2}{128k_1^2k_2^2}\left(1 - \frac{V_0}{8 \sigma} \sqrt{ \frac{\dimone \dimtwo }{ k_1^3k_2^3 } } \right).$$

If $\frac{4 \sigma}{\sqrt{\dimone \dimtwo}} \le V_0 \le 4 \sigma \dimone\dimtwo$, then we can choose $k_1$ and $k_2$ such that $k_1k_2$ is of order $ \big( \frac{V_0\sqrt{\dimone\dimtwo}}{4\sigma} \big)^{\frac23}$, which yields that
$$\underset{\tilde\theta}{\inf}\underset{\theta\in \cM_{V_0}}{\sup} R(\tilde\theta, \theta^*)\gtrsim \left( \frac{\sigma^2 V_0}{\dimone\dimtwo} \right)^{2/3}.$$
%Note moreover, that $\dimone\dimtwo\geq 16\sigma^2/V_0^2$ also implies
%$$\frac{V_0\sigma}{64\sqrt{\dimone\dimtwo}}\geq \frac{V_0^{\frac23}\sigma^{\frac43}}{64(\dimone\dimtwo)^{\frac23}},$$
%which concludes the proof.
If $V_0 \le \frac{4 \sigma}{\sqrt{\dimone \dimtwo}}$, then $\big( \frac{\sigma^2 V_0}{\dimone\dimtwo} \big)^{2/3} \lesssim \frac{\sigma^2}{ \dimone \dimtwo} $, so the second term in the statement of Theorem \ref{thm:lower_bound} is dominated by the first term. Finally, if $V_0 \ge 4 \sigma \dimone\dimtwo$, then $\big( \frac{\sigma^2 V_0}{\dimone\dimtwo} \big)^{2/3} \ge \sigma^2$, so the rate becomes trivial. This completes the proof.

\subsection{Proof of Theorem \ref{thm:minimax}}

Similar to the proof of Theorem \ref{thm:upper_bound}, the main technical step in this proof is to bound the supremum of an empirical process.
This is dealt with in the following proposition, whose proof is deferred to Section \ref{sec:pf-sup-per}.
Note that both the statement and the proof are very similar to Proposition \ref{prop:sup-id}, and that we  can restrict our attention to noise matrices that are Gaussian.

\begin{proposition} \label{prop:sup-per}
Fix any matrix $\theta^* \in \R^{\dimone \times \dimtwo}$ and permutations $\pi_1 \in \cS_{\dimone}$ and $\pi_2 \in \cS_{\dimtwo}$. Suppose that $\gaussiannoise \in \R^{\dimone \times \dimtwo}$ has \iid $\cN(0, 1)$ entries.
Then for any integer $k \in [ \dimone \dimtwo ]$ and any $t > 0$, we have
$$
\E \Big[ \sup_{\substack{\theta \in \cM_{V_0} \\ \|\theta (\pi_1, \pi_2) - \theta^* \|_F \leq t}} \langle \gaussiannoise , \theta (\pi_1, \pi_2) - \theta^* \rangle  \Big] 
\lesssim t \left[ \sqrt{\dimone} + \sqrt{ k \log (\dimtwo)} \right] 
+ \sqrt{\frac{\dimone \dimtwo }{k}} \sqrt{\log (\dimone) \log (\dimtwo)} \, V_0,
$$
where we use the convention that the supremum over the empty set is \( -\infty \).
\end{proposition}

We assume without loss of generality that the underlying true permutations $\pi_1^*$ and $\pi_2^*$ are the identities throughout the proof. For fixed permutations $\pi_1 \in \cS_{\dimone}$ and $\pi_2 \in \cS_{\dimtwo}$, we obtain from Theorem~\ref{thm:maj} and Proposition~\ref{prop:sup-per} that
\begin{align*}
	\leadeq{\gamma_2 \big( \{ \theta (\pi_1, \pi_2) - \theta^* : \theta \in \cM_{V_0} , \, \|\theta (\pi_1, \pi_2) - \theta^* \|_F \leq t \} \big)} \\
	\asymp {} & \E \Big[ \sup_{\substack{\theta \in \cM_{V_0} \\ \|\theta (\pi_1, \pi_2) - \theta^* \|_F \leq t}} \langle \varepsilon , \theta (\pi_1, \pi_2) - \theta^* \rangle  \Big] \\
	\lesssim {} & t \left[ \sqrt{\dimone} + \sqrt{ k \log (\dimtwo)} \right] 
	+ \sqrt{\frac{\dimone \dimtwo }{k}} \sqrt{\log (\dimone) \log (\dimtwo)} \, V_0 .
\end{align*}
Therefore, Theorem~\ref{thm:tail} yields that with probability $1 - 4 \exp(-s^2)$, 
\begin{align*}
	\sup_{\substack{\theta \in \cM_{V_0} \\ \|\theta (\pi_1, \pi_2) - \theta^* \|_F \leq t}} \langle \varepsilon , \theta (\pi_1, \pi_2) - \theta^* \rangle
	& \lesssim t \sigma \left[ \sqrt{\dimone} + \sqrt{ k \log (\dimtwo)} \right]
	+ \sigma \sqrt{\frac{\dimone \dimtwo }{k}} \sqrt{\log (\dimone) \log (\dimtwo)} \, V_0
	+ \sigma s t . 
\end{align*}
Taking $s = 2 \sqrt{ \dimone \log (\dimone) }$ and applying a union bound over all $(\pi_1, \pi_2) \in \cS_{\dimone} \times \cS_{\dimtwo} $ (which has log-cardinality $\log (\dimone ! \dimtwo !) \le 2 \dimone \log \dimone$), we get that with probability at least $1 - \dimone^{-\dimone}$, 
\begin{align}
	\leadeq{\sup_{\substack{(\pi_1, \pi_2, \theta) \in \cS_{\dimone} \times \cS_{\dimtwo} \times \cM_{V_0} \\ \|\theta (\pi_1, \pi_2) - \theta^* \|_F \leq t}} \langle \varepsilon , \theta (\pi_1, \pi_2) - \theta^* \rangle} \\
	\lesssim {} & t \sigma \left[ \sqrt{\dimone \log(\dimone) } + \sqrt{ k \log (\dimtwo)} \right]
	+ \sigma \sqrt{\frac{\dimone \dimtwo }{k}} \sqrt{\log (\dimone) \log (\dimtwo)} \, V_0 . 
\end{align}

Let us define 
\[
f_{\theta^*}(t) = \sup_{\substack{(\pi_1, \pi_2, \theta) \in \cS_{\dimone} \times \cS_{\dimtwo} \times \cM_{V_0} \\ \| \theta( \pi_1, \pi_2 ) - \theta^* \|_F \le t }} \langle \eps,  \theta( \pi_1, \pi_2 ) - \theta^* \rangle - \frac{t^2}2 .
\]
Then for any 
\begin{align*}
	t > t^* & \defn C \sigma \left[ \sqrt{\dimone \log(\dimone)} + \sqrt{ k \log (\dimtwo) } \right]
	+ C \Big[ \sigma \sqrt{\frac{\dimone \dimtwo }{k}} \sqrt{\log (\dimone) \log (\dimtwo)} \, V_0 \Big]^{1/2}
\end{align*}
where \( C \) is a sufficiently large constant, it holds with probability at least $1 - \dimone^{- \dimone}$ that 
\begin{align}
	f_{\theta^*}( t ) < 0.
\end{align}
Therefore by Lemma~\ref{lem:variational}, we obtain
\begin{align*}
	\frac{1}{\dimone \dimtwo} \| \thetagls - \theta^* \|_F^2 \le \frac{(t^*)^2}{\dimone \dimtwo} 
	& \lesssim \sigma^2 \left[ \frac{\log(\dimone) }{\dimtwo} + \frac{k \log (\dimtwo)}{\dimone \dimtwo}  \right] 
	+ \sigma \frac{\sqrt{\log (\dimone) \log (\dimtwo)}}{\sqrt{ \dimone \dimtwo k }} V_0,
\end{align*}
noting that by assumption, \( \theta^\ast \in \bigcup_{ \substack{ \scriptscriptstyle \pi_1 \in \cS_{\dimone} \\ \scriptscriptstyle \pi_2 \in \cS_{\dimtwo} } } \cM_{V_0} (\pi_1, \pi_2) \).

Balancing the terms that depend on \( k \) leads to the choice
\begin{align}
	k^* = (\dimone \dimtwo)^{1/3} \left( \frac{\log(\dimone)}{\log(\dimtwo)} \right)^{1/3} \left( \frac{V_0}{\sigma} \right)^{2/3} ,
\end{align}
and therefore we obtain that with probability $1 - \dimone^{- \dimone}$, 
\begin{align}
	\frac{1}{\dimone \dimtwo} \| \thetagls - \theta^* \|_F^2 \lesssim \frac{\sigma^2 \log(\dimone) }{\dimtwo} + \left( \frac{ \sigma^2 V_0 }{ \dimone \dimtwo } \right)^{2/3}  \log(\dimone)^{1/3} \log(\dimtwo)^{2/3},
\end{align}
if \( 1 \le k^\ast \le \dimone \dimtwo \).
We conclude by arguing similar to the proof of Theorem \ref{thm:upper_bound} to handle the other possible cases of \( k^\ast \) and to get an error bound in expectation instead of with high probability.

\subsection{Proof of Proposition~\ref{prop:sup-per}}
\label{sec:pf-sup-per}

Note that Proposition~\ref{prop:sup-per} is very similar to Proposition~\ref{prop:sup-id}, and so are their proofs.
The difference is that Proposition~\ref{prop:sup-per} has extra complications arising from the presence of permutations, while at the same time it is simpler because we restrict $\theta$ to $\cM_{V_0} \subset \cM$.
Hence, we focus only the differences to the proof of Proposition \ref{prop:sup-id} here.

Since $\gaussiannoise$ is equal in distribution to $\gaussiannoise( \pi_1^{-1}, \pi_2^{-1} )$, we have
$$
\langle \gaussiannoise , \theta (\pi_1, \pi_2) - \theta^* \rangle \stackrel{d}{=} \langle \gaussiannoise , \theta  - \theta^* (\pi_1^{-1}, \pi_2^{-1}) \rangle
$$
in distribution for any permutations $\pi_1 \in \cS_{\dimone}, \pi_2 \in \cS_{\dimtwo}$. Therefore, by replacing $\theta^* (\pi_1^{-1}, \pi_2^{-1})$ with $\theta^*$,  it suffices to prove that for any matrix $\theta^* \in \R^{\dimone \times \dimtwo}$, it holds that 
%for any $\pi_1 \in \cS_{\dimone}, \pi_2 \in \cS_{\dimtwo}$, it holds that
\begin{align*}
\E \Big[ \sup_{\substack{\theta \in \cM_{V_0} \\ \|\theta - \theta^* \|_F \leq t}} \langle \gaussiannoise , \theta  - \theta^* \rangle  \Big] 
\lesssim t \left[ \sqrt{\dimone} + \sqrt{ k \log (\dimtwo)} \right] 
+ \sqrt{\frac{\dimone \dimtwo }{k}} \log (\dimone) \log (\dimtwo) \, V_0 .
\end{align*}

Note that this supremum is very similar to that studied in Proposition~\ref{prop:sup-id}, with the main differences being that $\theta^*$ can be any matrix in $\R^{\dimone \times \dimtwo}$ while $\theta$ is restricted to $\cM_{V_0}$.
% Hence we focus on the differences in the proof.
As in the proof of Proposition~\ref{prop:sup-id}, \eqref{eq:cd}, the supremum can be split into three terms:
\begin{align*}
	\E[\sup_{\substack{\theta \in \cM_{V_0} \\ \|\theta - \theta^* \|_F \leq t}} \langle \gaussiannoise , \theta - \theta^* \rangle]
	\leq {} & \E[\sup_{\substack{\theta \in \cM_{V_0} \\ \|\theta - \theta^* \|_F \leq t}} \langle \Pi_1 (\gaussiannoise) , \theta - \theta^* \rangle] 
	+ \E[\sup_{\substack{\theta \in \cM_{V_0} \\ \|\theta - \theta^* \|_F \leq t}} \langle  (I - \Pi_1) \Pi_2 (\gaussiannoise) , \theta - \theta^* \rangle]\\
	{} & + \E[\sup_{\substack{\theta \in \cM_{V_0} \\ \|\theta - \theta^* \|_F \leq t}} \langle (I - \Pi_1) (I - \Pi_2) (\gaussiannoise) , \theta - \theta^* \rangle].
%	\label{eq:three-terms}
\end{align*}

The first two terms can be bounded exactly as before, because we only need the condition $\| \theta - \theta^* \|_F \le t$ but not any other property of $\theta^*$. Up to a constant factor, the third term can be bounded by (recall~\eqref{eq:s1} and~\eqref{eq:s3} in the proof of Proposition~\ref{prop:sup-id})
$$
\sup_{\substack{\theta \in \mathcal{M}_{V_0} \\ \| \theta - \theta^* \|_F \leq t}} \| \Phi \odot (D \theta \Dtwo^\top) \|_1 .
$$
Thanks to the constraint $\theta \in \cM_{V_0}$ in this case, we immediately obtain 
$$
\| \Phi \odot (D \theta \Dtwo^\top) \|_1 
= \langle \Phi , D \theta \Dtwo^\top \rangle
\leq \| \Phi \|_\infty \| D \theta \Dtwo^\top \|_1
\lesssim \sqrt{\frac{\dimone \dimtwo }{k}} \sqrt{\log (\dimone) \log (\dimtwo)} \, V_0 .
$$
Hence, it holds 
\begin{align*} 
\E \Big[ \sup_{\substack{\theta \in \cM\\ \|\theta - \theta^* \|_F \leq t}} \langle (I - \Pi_1) (I - \Pi_2) (\gaussiannoise) , \theta - \theta^* \rangle \Big] 
\lesssim \sqrt{\frac{\dimone \dimtwo }{k}} \sqrt{\log (\dimone) \log (\dimtwo)} \, V_0 .
\end{align*}
Combining the bounds on the three terms completes the proof.

\subsection{Proof of Theorem~\ref{thm:per-proj}}

We assume without loss of generality that the underlying true permutations $\pi_1^*$ and $\pi_2^*$ are the identities throughout the proof, except where the notations \( \pi_1^\ast \) and \( \pi_2^\ast \) are explicitly used.
Recall that we defined the reversal permutation $\pir  \in \cS_{\dimone}$ by $\pir (i) = \dimone - i + 1$ for $i \in [\dimone]$. 
Given estimators $\hat \pi_1, \hat \pi_2$, let us define 
\begin{align}
\tilde \theta \defn \argmin_{\theta \in \cM_{V_0} ( \hat \pi_1, \hat \pi_2 ) \cup \cM_{V_0} ( \pir \circ \hat \pi_1, \hat \pi_2 ) } \| \theta - \theta^* \|_F^2 .
\label{eq:t-theta-def}
\end{align}

The theorem follows from the next two propositions combined.
The first proposition says that the final denoising error can be bounded by the sum of the minimax rate (the error rate incurred by the projection step of the algorithm), and the error incurred by the permutation estimators. 
The second proposition controls the error incurred by the permutation estimators. 

\begin{proposition} \label{prop:per-proj}
Suppose that we have $y = \theta^* + \eps$, where the noise matrix $\eps$ has independent $\subG( C \sigma^2 )$ entries with $\Var [ \eps_{i, j} ] = \sigma^2 .$
Let the estimators $(\hat \pi_1, \hat \pi_1' , \hat \pi_2, \hat \theta)$ be given by Algorithm \ref{alg:variance-sorting-main}, and define $\tilde \theta$ according to~\eqref{eq:t-theta-def}. Then it holds with probability at least $1 - \dimone^{- \dimone}$ that 
\begin{align*}
	\frac{1}{\dimone \dimtwo} \| \hat \theta - \theta^* \|_F^2 
	\lesssim \left[\frac{\sigma^2 \log(\dimone) }{\dimtwo} + \left( \frac{ \sigma^2 V_0 }{ \dimone \dimtwo } \right)^{2/3}  \log(\dimone)^{1/3} \log(\dimtwo)^{2/3} \right] \wedge \sigma^2
	+ \frac{1}{\dimone \dimtwo} \| \tilde \theta - \theta^* \|_F^2 .
\end{align*}
\end{proposition}

\begin{proposition} \label{prop:per-err}
Suppose that we have $y = \theta^* + \eps$, where $\theta^* \in \cM_{V_0}^{\dimone, \dimtwo}$ and the noise matrix $\eps$ has independent $\subG( C \sigma^2 )$ entries. 
For the permutation estimators $\hat \pi_1$ and $\hat \pi_2$ given by the {\sf Variance Sorting} subroutine, Algorithm \ref{alg:variance-sorting-sub}, let $\tilde \theta$ be defined by~\eqref{eq:t-theta-def}. Then it holds with probability at least $1 - \dimone^{-9}$ that
\begin{align}
\frac{1}{\dimone \dimtwo} \| \tilde \theta - \theta^* \|_F^2
\lesssim \big( \sigma^2 + \sigma V_0 \big)  \Big( \frac{ \log  \dimone}{ \dimtwo} \Big)^{1/2} . 
\end{align}
\end{proposition}

The error bound in Proposition \ref{prop:per-err} is clearly dominating, and combining the two propositions yields the statement of the theorem in probability.
Taking into account that we can do the same analysis keeping track of the failure probability independently of \( \dimone \) and integrate, we obtain bounds in expectation instead of probability as well.

\subsection{Proof of Proposition~\ref{prop:per-proj}}

To employ the variational formula in Lemma~\ref{lem:variational}, let us view $\tilde \theta$ defined by~\eqref{eq:t-theta-def} as the ground truth, and view $y - \tilde \theta = \eps + \theta^* - \tilde \theta$ as the noise. Correspondingly, we define 
\begin{equation}
	\label{eq:jv}
f_{\tilde \theta} (t) \defn \sup_{\substack{ \theta \in \cM_{V_0} ( \hat \pi_1, \hat \pi_2 ) \cup \cM_{V_0} ( \pir \circ \hat \pi_1, \hat \pi_2 ) \\ \| \theta - \tilde \theta \|_F \le t }} \langle \eps + \theta^* - \tilde \theta,  \theta - \tilde \theta \rangle - \frac{t^2}2 .
\end{equation}
% Note that Lemma~\ref{lem:variational} is applicable since \( \tilde \theta \in  \cM_{V_0} ( \hat \pi_1, \hat \pi_2 ) \cup \cM_{V_0} ( \pir \circ \hat \pi_1, \hat \pi_2 ) \) by definition.
To facilitate our analysis, for each pair of permutations $(\pi_1, \pi_2) \in \cS_{\dimone} \times \cS_{\dimtwo}$, we define
$$
\tilde \theta^{\pi_1, \pi_2} \defn \argmin_{\theta \in \cM_{V_0} ( \pi_1, \pi_2 )} \| \theta - \theta^* \|_F^2,
$$
and note that we have that either $\tilde \theta = \tilde \theta^{\hat \pi_1, \hat \pi_2}$ or $\tilde \theta = \tilde \theta^{\pir \circ \hat \pi_1, \hat \pi_2}$, so \( \tilde \theta \in \cM_{V_0} ( \hat \pi_1, \hat \pi_2 ) \cup \cM_{V_0} ( \pir \circ \hat \pi_1, \hat \pi_2 ) \) and Lemma~\ref{lem:variational} is applicable.

We further estimate the supremum in \eqref{eq:jv} by
\begin{align}
	f_{\tilde \theta} (t) \le {} & \sup_{\substack{ \theta \in \cM_{V_0} ( \hat \pi_1, \hat \pi_2 ) \cup \cM_{V_0} ( \pir \circ \hat \pi_1, \hat \pi_2 ) \\ \| \theta - \tilde \theta \|_F \le t }} \langle \eps ,  \theta - \tilde \theta \rangle  \notag \\
	{} & \quad +  \sup_{\substack{ \theta \in \cM_{V_0} ( \hat \pi_1, \hat \pi_2 ) \cup \cM_{V_0} ( \pir \circ \hat \pi_1, \hat \pi_2 ) \\ \| \theta - \tilde \theta \|_F \le t }} \langle \theta^* - \tilde \theta,  \theta - \tilde \theta \rangle  - \frac{t^2}2   \label{eq:sup-1} \\
	\le {} & \sup_{\substack{ (\pi_1, \pi_1', \pi_2, \pi_2', \theta) \in \cS_{\dimone}^2 \times \cS_{\dimtwo}^2 \times \cM_{V_0} \\ \| \theta (\pi_1', \pi_2') - \tilde \theta^{\pi_1, \pi_2} \|_F \le t }} \langle \eps ,  \theta (\pi_1', \pi_2') - \tilde \theta^{\pi_1, \pi_2} \rangle +
t \| \theta^* - \tilde \theta \|_F - \frac{t^2}2 .   \label{eq:sup-2}
\end{align}
Note that the random variables $\hat \pi_1, \hat \pi_2$ and $\tilde \theta$ depend on $\eps$, so it is not clear how to control the first supremum in~\eqref{eq:sup-1}. Instead, in~\eqref{eq:sup-2} we take a supremum over all $(\pi_1, \pi_1', \pi_2, \pi_2', \theta) \in \cS_{\dimone}^2 \times \cS_{\dimtwo}^2 \times \cM_{V_0}$, where each individual quantity $\theta (\pi_1', \pi_2') - \tilde \theta^{\pi_1, \pi_2}$ is deterministic. 
%To study this larger supremum, using that $\tilde \theta^{\pi_1, \pi_2} \in \cM_{V_0} (\pi_1, \pi_2)$, we may further derive 
%$$
%f_{\tilde \theta} (t) \le \sup_{\substack{ (\pi_1, \pi_1', \pi_2, \pi_2', \theta) \in \cS_{\dimone}^2 \times \cS_{\dimtwo}^2 \times \cM_{V_0} \\ \| \theta (\pi_1', \pi_2') - \tilde \theta^{\pi_1, \pi_2} \|_F \le t }} \langle \eps ,  \theta (\pi_1', \pi_2') - \tilde \theta^{\pi_1, \pi_2} \rangle +
%t \| \theta^* - \tilde \theta \|_F - \frac{t^2}2
%$$

For fixed permutations $\pi_1, \pi_1' \in \cS_{\dimone}$ and $\pi_2, \pi_2' \in \cS_{\dimtwo}$, we obtain from Theorem~\ref{thm:maj} and Proposition~\ref{prop:sup-per} that
\begin{align*}
	\leadeq{\gamma_2 \big( \{ \theta (\pi_1', \pi_2') - \tilde \theta^{\pi_1, \pi_2} : \theta \in \cM_{V_0} , \, \|\theta (\pi_1', \pi_2') - \tilde \theta^{\pi_1, \pi_2} \|_F \leq t \} \big)} \\
	\asymp {} & \E_{\gaussiannoise_{i,j} \simiid \cN(0, 1)} \Big[ \sup_{\substack{\theta \in \cM_{V_0} \\ \|\theta (\pi_1', \pi_2') - \tilde \theta^{\pi_1, \pi_2} \|_F \leq t}} \langle \gaussiannoise , \theta (\pi_1', \pi_2') - \tilde \theta^{\pi_1, \pi_2} \rangle  \Big] \\
	\lesssim {} & t \left[ \sqrt{\dimone} + \sqrt{ k \log (\dimtwo)} \right] 
	+ \sqrt{\frac{\dimone \dimtwo }{k}} \sqrt{\log (\dimone) \log (\dimtwo)} \, V_0.
\end{align*}
Therefore, Theorem~\ref{thm:tail} yields that with probability $1 - 4 \exp(-s^2)$, 
\begin{align*}
	& \sup_{\substack{\theta \in \cM_{V_0} \\ \|\theta (\pi_1', \pi_2') - \tilde \theta^{\pi_1, \pi_2} \|_F \leq t}} \langle \varepsilon , \theta (\pi_1', \pi_2') - \tilde \theta^{\pi_1, \pi_2} \rangle \\
	& \lesssim t \sigma \left[ \sqrt{\dimone} + \sqrt{ k \log (\dimtwo)} \right]
	+ \sigma \sqrt{\frac{\dimone \dimtwo }{k}} \sqrt{\log (\dimone) \log (\dimtwo)} \, V_0
	+ \sigma s t . 
\end{align*}
Taking $s = 3 \sqrt{ \dimone \log (\dimone) }$ and applying a union bound over all $\pi_1, \pi_1' \in \cS_{\dimone}$ and $\pi_2, \pi_2' \in \cS_{\dimtwo}$, we see that with probability at least $1 - \dimone^{-\dimone}$, 
\begin{align}
	\leadeq{\sup_{\substack{ (\pi_1, \pi_1', \pi_2, \pi_2', \theta) \in \cS_{\dimone}^2 \times \cS_{\dimtwo}^2 \times \cM_{V_0} \\ \| \theta (\pi_1', \pi_2') - \tilde \theta^{\pi_1, \pi_2} \|_F \le t }} \langle \eps ,  \theta (\pi_1', \pi_2') - \tilde \theta^{\pi_1, \pi_2} \rangle} \\
	\lesssim {} & t \sigma \left[ \sqrt{\dimone \log(\dimone) } + \sqrt{ k \log (\dimtwo)} \right]
	+ \sigma \sqrt{\frac{\dimone \dimtwo }{k}} \sqrt{\log (\dimone) \log (\dimtwo)} \, V_0 . 
\end{align}

This together with inequality~\eqref{eq:sup-2} yields that for any 
\begin{align*}
	t > t^* & \defn C \sigma \left[ \sqrt{\dimone \log(\dimone)} + \sqrt{ k \log (\dimtwo) } \right]
	+ C \Big[ \sigma \sqrt{\frac{\dimone \dimtwo }{k}} \sqrt{\log (\dimone) \log (\dimtwo)} \, V_0 \Big]^{1/2} + C \| \tilde \theta - \theta^* \|_F 
\end{align*}
where \( C \) is a sufficiently large constant, it holds with probability at least $1 - \dimone^{- \dimone}$ that 
\begin{align}
	f_{\tilde \theta}( t ) < 0.
\end{align}
Therefore, by Lemma~\ref{lem:variational} applied to the set $\cM_{V_0} ( \hat \pi_1, \hat \pi_2 ) \cup \cM_{V_0} ( \pir \circ \hat \pi_1, \hat \pi_2 )$, we obtain
\begin{align*}
	\frac{1}{\dimone \dimtwo} \| \hat \theta - \tilde \theta \|_F^2 \le \frac{(t^*)^2}{\dimone \dimtwo} 
	& \lesssim \sigma^2 \left[ \frac{\log(\dimone) }{\dimtwo} + \frac{k \log (\dimtwo)}{\dimone \dimtwo}  \right] 
	+ \sigma \frac{\sqrt{\log (\dimone) \log (\dimtwo)}}{\sqrt{ \dimone \dimtwo k }} V_0 
	+ \frac{1}{\dimone \dimtwo} \| \tilde \theta - \theta^* \|_F^2 .
\end{align*}
Balancing the terms that depend on \( k \) yields that with probability $1 - \dimone^{- \dimone}$, 
\begin{align*}
	\frac{1}{\dimone \dimtwo} \| \hat \theta - \theta^* \|_F^2 
	&\lesssim \frac{1}{\dimone \dimtwo} \| \hat \theta - \tilde \theta \|_F^2 + \frac{1}{\dimone \dimtwo} \| \tilde \theta - \theta^* \|_F^2 \\
	&\lesssim \frac{\sigma^2 \log(\dimone) }{\dimtwo} + \left( \frac{ \sigma^2 V_0}{ \dimone \dimtwo } \right)^{2/3}  \log(\dimone)^{1/3} \log(\dimtwo)^{2/3}
	+ \frac{1}{\dimone \dimtwo} \| \tilde \theta - \theta^* \|_F^2,
\end{align*}
if for the optimal \( k^\ast \), we have \( 1 \le k^\ast \le \dimone \dimtwo \).
The other cases can be handled as in the proof of Theorem \ref{thm:upper_bound}.

\subsection{Proof of Proposition \ref{prop:per-err}}

\subsubsection{Reduction to the row/column-centered case.}
First, we reduce the problem to the case where the underlying matrix $\theta^*$ has centered rows and columns. If $\theta^*$ is not centered, we may choose a matrix $R \in \R^{\dimone \times \dimtwo}$ with constant rows and a matrix $S \in \R^{\dimone \times \dimtwo}$ with constant columns, so that if $\bar \theta \defn \theta^* - R - S$, then for all $i \in [\dimone], \, j \in [\dimtwo]$, 
$$
\sum_{k = 1}^{\dimone} \bar \theta_{k, j} = \sum_{\ell = 1}^{\dimtwo} \bar \theta_{i, \ell} = 0 .
$$
Since $R$ and $S$ have constant rows and columns respectively, we have $V(\bar \theta) = V(\theta^*)$ by definition~\eqref{eq:def-v}, and thus $\bar \theta \in \cM_{V_0}$. 
More importantly, according to the definition of $\xi(i, j)$ in~\eqref{eq:xi-def}, its value does not change if we replace $y$ by $y - R - S$. Therefore, we may assume without loss of generality that 
$$ y = \bar \theta + \eps , $$
which does not change the estimators $\hat \pi_1$, $\hat \pi_1'$ and $\hat \pi_2$ output by the algorithm. 

Furthermore, we have
\begin{align}
\min_{\theta \in \cM_{V_0} ( \hat \pi_1, \hat \pi_2 ) \cup \cM_{V_0} ( \pir \circ \hat \pi_1, \hat \pi_2 ) } \| \theta - \theta^* \|_F^2 
= \min_{\theta \in \cM_{V_0} ( \hat \pi_1, \hat \pi_2 ) \cup \cM_{V_0} ( \pir \circ \hat \pi_1, \hat \pi_2 ) } \| \theta - \bar \theta \|_F^2 , \label{eq:min}
\end{align}
since if $\tilde \theta$ minimizes the left-hand side, then $\tilde \theta - R - S$ minimizes the right-hand side. Hence it suffices to show that the right-hand side of~\eqref{eq:min} is bounded by the desired rate. 

Note that by symmetry of the anti-Monge constraint,
\begin{align}
	\label{eq:jw}
	\cM_{V_0}(\pir \circ \hat \pi_1, \hat \pi_2) = {} & \cM_{V_0}(\hat \pi_1, \pirtwo \circ \hat \pi_2) \quad \text{and}\\
	\cM_{V_0}(\hat \pi_1, \hat \pi_2)
	= {} & \cM_{V_0}(\pir \circ \hat \pi_1, \pir \circ \hat \pi_2).
 \end{align}
In light of this, it is sufficient to show
\begin{align}
	\min_{\substack{ \pi_1 \in \{\id, \pir\}\\  \pi_2 \in \{\id, \pirtwo\}}} \| \bar \theta( \pi_1 \circ \hat \pi_1 ,  \pi_2 \circ \hat \pi_2) - \bar \theta \|_F^2
	\lesssim \big[ \sigma^2 + \sigma V(\bar \theta) \big]  \dimone \sqrt{\dimtwo \log  \dimone } ,  \label{eq:per-err}
\end{align}
which we do in the sequel. This gives an upper bound on~\eqref{eq:min} and thus completes the proof.

\subsubsection{Preliminaries.}

Before proceeding to proving \eqref{eq:per-err}, we start with some lemmas.

\begin{lemma} \label{lem:cen-sq}
For $\gamma \in \R^n$ such that $\sum_{k=1}^n \gamma_k = 0$, it holds that
$$
\sum_{k < \ell} (\gamma_k - \gamma_\ell)^2 
= n \sum_{k=1}^n \gamma_k^2 .
$$
\end{lemma}
The proof follows by inspection.
%\begin{proof}
%It holds that
%$$
%\sum_{k, \ell = 1}^n (\gamma_k - \gamma_\ell)^2 
%= \sum_{k, \ell = 1}^n ( \gamma_k^2 + \gamma_\ell^2 - 2 \gamma_k \gamma_\ell ) 
%= 2 n \sum_{k=1}^n \gamma_k^2 - 2 \Big( \sum_{k=1}^n \gamma_k \Big)^2 ,
%$$
%from which the lemma follows immediately.
%\end{proof}

%The next deterministic lemma is regarding a permutation and a bivariate function.

\begin{lemma} \label{lem:per-sort}
Let $f:[n] \times [n] \to \R$ be a symmetric bivariate function such that 
\begin{align}
f(i, m) \lor f(m, j) \le f(i, j) , \quad \text{ for all } 1 \le i \le m \le j \le n.
\label{eq:assum}
\end{align}
%\begin{itemize}
%\item
%$f(i, i) = 0$ and $f(i, j) = f(j, i)$ for $i, j \in [n]$;
%
%\item
%$f(i, m) \lor f(m, j) \le f(i, j)$ for all $1 \le i \le m \le j \le n$.
%\end{itemize} 
Let $\pi:[n] \to [n]$ be a permutation and $\tau \in \R$. Suppose that 
\begin{align}
f (i, j) \le \tau , \quad \text{ if } i < j \text{ and } \pi(i) > \pi(j) . \label{eq:assum-2}
\end{align}
Then we have $ f(\pi(i), i) \le \tau$ for all $i \in [n]$.
\end{lemma}

\begin{proof}
Suppose that $f( \pi(j), j ) > \tau$ and $\pi(j) < j$ for some index $j \in [n]$. Since $\pi$ is a bijection, there must exist an index $i \le \pi(j) < j$ such that $\pi(i) > \pi(j)$. However, by~\eqref{eq:assum} we then have $$
f(i, j) \ge f( \pi(j), j ) > \tau , 
$$
which contradicts assumption~\eqref{eq:assum-2}. A similar argument yields a contradiction in the case that $f( \pi(j), j ) > \tau$ and $\pi(j) > j$. Therefore, we obtain that
$f( \pi(j), j ) \le \tau$ for all $j \in [n]$.
\end{proof}

%\subsubsection{Properties of $\xi (i, j)$.}

Next, we study the quantity $\xi(i, j)$ used in the algorithm defined by~\eqref{eq:xi-def}.
Throughout the rest of the proof, we use the notation
\begin{align}
f(i, j) \defn \sum_{k=1}^{\dimtwo} ( \bar \theta_{i, k} - \bar \theta_{j, k}  )^2 .
\label{eq:f-def}
\end{align}

\begin{lemma} \label{lem:xi-prop}
%For any distinct $i, j \in [\dimone]$, it holds that
%\begin{align}
%\E[ \mu_i ] = \frac{1}{\dimtwo} \sum_{k=1}^{\dimtwo} ( \bar \theta_{i,k} - \bar \theta_{j,k} )  , \quad 
%\Var[ \mu_i ] = \frac{2}{\dimtwo}, 
%\end{align}
%and that with probability $1 -  \dimone^{-10}$, 
%\begin{align}
%\big| \mu_i - \E[ \mu_i ] \big| \lesssim \Big( \frac{ \log \dimone }{ \dimtwo } \Big)^{1/2} .
%\end{align}
%Furthermore, let ${\theta}_{i, k} = \bar \theta_{i, k} - \frac{1}{\dimtwo} \sum_{k=1}^{\dimtwo} \bar \theta_{i, k}$. 
Suppose that $y = \bar \theta + \eps$, where $\bar \theta$ has centered rows and columns, and $\eps$ has independent $\subG(C \sigma^2)$ entries with $\Var[ \eps_{i, j} ] = \sigma^2$.
Then it holds that for all distinct $i, j \in [\dimone]$, 
\begin{align}
\E[ \xi ( i, j ) ] = f(i, j) + 2(\dimtwo - 1) \sigma^2 , \label{eq:xi-exp}
\end{align}
and that with probability $1 -  \dimone^{-10}$, for all $i, j \in [\dimone]$,
\begin{align}
\big| \xi ( i, j ) - \E[ \xi ( i, j ) ] \big| 
%&\lesssim \sigma^2 \sqrt{\dimtwo \log  \dimone} + \sigma 
%%\Big[ \sum_{k=1}^{\dimtwo} ( \bar \theta_{i, k} - \bar \theta_{j, k}  )^2  \Big]^{1/2} 
%\sqrt{f(i, j) \log  \dimone } \\
\le \tau \defn C \big[ \sigma^2 + \sigma V(\bar \theta) \big] \sqrt{\dimtwo \log  \dimone } ,
\label{eq:xi-dev}
\end{align}
where $C>0$ is a universal constant and $V(\bar \theta)$ is defined as in~\eqref{eq:def-v}.
\end{lemma}
The proof of the lemma is deferred to Section~\ref{sec:pf-xi}.

Next, we study properties of the expectation $\E[ \xi(i, j) ]$, or equivalently $f(i, j)$, thanks to~\eqref{eq:xi-exp}.

\begin{lemma} \label{lem:f-prop}
It holds for all $1 \le i \le m \le j \le \dimone$ that
$
f(i, m) + f(m, j) \le f(i, j) .
$
\end{lemma}

\begin{proof}
Fix indices $1 \le i \le m \le j \le \dimone$. 
%For $1 \le k \le \ell \le \dimtwo$, let us define 
%$$
%\Delta(i, j; k, \ell) \defn \bar \theta_{i, k} + \bar \theta_{j, \ell} - \bar \theta_{i, \ell} - \bar \theta_{j, k} \ge 0 ,
%$$
%where the nonnegativity holds because $\bar \theta$ is anti-Monge.
Since all the row sums of $\bar \theta$ are zero by assumption, it follows from Lemma~\ref{lem:cen-sq} that
\begin{align}
f(i, j) = \sum_{k=1}^{\dimtwo} \big( \bar \theta_{i, k} - \bar \theta_{j, k}  \big)^2
= \frac{1}{\dimtwo} \sum_{k < \ell} \big( \bar \theta_{i, k} - \bar \theta_{j, k} - \bar \theta_{i, \ell} + \bar \theta_{j, \ell}  \big)^2 .
%= \sum_{k < \ell} \Delta(i, j; k, \ell)^2 .
\label{eq:f-sum}
\end{align}
Note that we have 
$$
\big( \bar \theta_{i, k} - \bar \theta_{m, k} - \bar \theta_{i, \ell} + \bar \theta_{m, \ell} \big)
+ \big( \bar \theta_{m, k} - \bar \theta_{j, k} - \bar \theta_{m, \ell} + \bar \theta_{j, \ell} \big)
= \big( \bar \theta_{i, k} - \bar \theta_{j, k} - \bar \theta_{i, \ell} + \bar \theta_{j, \ell} \big) ,
$$
where each of the three bracketed terms is nonnegative because $\bar \theta$ is anti-Monge. 
Therefore, we obtain
$$
\big( \bar \theta_{i, k} - \bar \theta_{m, k} - \bar \theta_{i, \ell} + \bar \theta_{m, \ell} \big)^2
+ \big( \bar \theta_{m, k} - \bar \theta_{j, k} - \bar \theta_{m, \ell} + \bar \theta_{j, \ell} \big)^2
\le \big( \bar \theta_{i, k} - \bar \theta_{j, k} - \bar \theta_{i, \ell} + \bar \theta_{j, \ell} \big)^2 .
$$
This together with~\eqref{eq:f-sum} completes the proof.
\end{proof}

\subsubsection{Proof of main bound \eqref{eq:per-err}.}

We condition on the event of probability $1 - \dimone^{-10}$ that~\eqref{eq:xi-dev} holds. 
%Define $\tau \defn C [ \sigma^2 + \sigma V(\bar \theta) ] \sqrt{\dimtwo \log  \dimone } $ for a large constant $C$ so that we have
%\begin{align} 
%\big|\xi(i, j) - \E[ \xi(i, j) ] \big| \le \tau
%\label{eq:xi-dev}
%\end{align}
%for all $i, j \in [\dimone]$.
Consider the permutation estimator $\hat \pi_1$ defined as in the algorithm so that $\{ \xi ( i_0, \hat \pi_1^{-1}(i) ) \}_{i=1}^{\dimone}$ is nondecreasing.
Note that we made the decision of whether to consider \( i_0 \) or \( j_0 \) as an estimator for \( \hat \pi_1^{-1} (1) \) arbitrarily by demanding \( i_0 < j_0 \).
In the following, we make use of the fact that this orientation aligns with the assumption of \( \pi_1^\ast = \id \), that is, we use \( \pi_1^\ast (i_0) < \pi_1^\ast (j_0) \).
If the reverse inequality holds, then we may repeat the same proof with \( \hat \pi_1 \) replaced by \( \pir \circ \hat \pi_1 \), to obtain permutation guarantees for the reversed permutation instead.

% In fact, we will simply prove that $\| \bar \theta(\hat \pi_1 , \hat \pi_2) - \bar \theta \|_F^2$ is bounded as desired, because we assumed that the true permutations $\pi_1^*$ and $\pi_2^*$ are the identities, and defined $\hat \pi_1$ and $\hat \pi_2$ so that $\hat \pi_1 (1) \le \hat \pi_1 (\dimone)$ and $\hat \pi_2 (1) \le \hat \pi_2 (\dimtwo)$.
% In the general case, the bound~\eqref{eq:per-err} holds only when we allow a potential reversal of the permutation $\hat \pi_1$ on the left-hand side. 

We claim that for $f$ defined in~\eqref{eq:f-def},
\begin{align}
f(i, j) \le 12 \tau , \quad \text{ if } i < j \text{ and } \hat \pi_1(i) > \hat \pi_1(j) .
\label{eq:claim-1}
\end{align}

To establish the claim, we first consider any pair $(i, j)$ for which $i_0 < i < j$ and
%$\{ \E [ \xi ( i_0, i ) ] \}_{i=i_0}^{\dimone}$ is nondecreasing, and  Therefore, 
$\hat \pi_1 (i) > \hat \pi_1 (j)$. Thus by the definition of $\hat \pi_1$, we have $\xi( i_0, i ) > \xi ( i_0, j )$. Then it follows from Lemma~\ref{lem:f-prop},~\eqref{eq:xi-exp} and~\eqref{eq:xi-dev} that 
$$
f(i, j) \le f(i_0, j) - f(i_0, i) = \E[ \xi( i_0, j ) ] - \E[ \xi ( i_0, i ) ]
\le \xi( i_0, j ) - \xi ( i_0, i ) + 2 \tau \le 2 \tau .
$$
%$$
%\sum_{k=1}^{\dimtwo} \big( \bar \theta_{i, k} - \bar \theta_{j, k}  \big)^2 \le \E[ \xi( i_0, j ) ] - \E[ \xi ( i_0, i ) ]
%\le \xi( i_0, j ) - \xi ( i_0, i ) + 2 \tau \le 2 \tau .
%$$

Next, consider any pair $(i, j)$ where $i \le i_0$ and $\hat \pi_1 (i) > \hat \pi_1 (j)$.
By the definition of $(i_0, j_0)$ in~\eqref{eq:def-i0j0}, we have that $i_0 < j_0$ and $\xi( i_0, j_0 ) \ge \xi ( i, j_0 )$.
Together with Lemma~\ref{lem:f-prop},~\eqref{eq:xi-exp} and~\eqref{eq:xi-dev}, this implies that
\begin{align}
f(i, i_0)  \le f(i, j_0) - f(i_0, j_0) = \E [ \xi ( i, j_0 ) ] - \E [ \xi ( i_0, j_0 ) ] \le \xi( i, j_0 ) - \xi ( i_0, j_0 ) + 2 \tau \le 2 \tau .
\label{eq:i-i0}
\end{align}
Moreover, we have $\xi ( i, i_0 ) \ge \xi( i_0, j )$ since $\hat \pi_1 (i) >\hat \pi_1 (j)$. Therefore, by~\eqref{eq:xi-exp} and~\eqref{eq:xi-dev} it holds 
\begin{align}
f(i_0, j) - f(i, i_0) =  \E [ \xi ( i_0, j ) ] - \E [ \xi ( i, i_0 ) ] 
\le \xi ( i_0, j ) - \xi ( i, i_0 ) + 2 \tau \le 2 \tau . \qquad \label{eq:i-j}
\end{align}
Combining~\eqref{eq:i-i0} and~\eqref{eq:i-j}, we obtain
\begin{align*}
f(i, j) = \sum_{k=1}^{\dimtwo} \big( \bar \theta_{j, k} - \bar \theta_{i, k}  \big)^2 
&\le 2 \sum_{k=1}^{\dimtwo} \Big[ \big( \bar \theta_{j, k} - \bar \theta_{i_0, k}  \big)^2 + \big( \bar \theta_{i_0, k} - \bar \theta_{i, k}  \big)^2 \Big] \\
&= 2 f(i_0, j) + 2 f(i, i_0) \le 12 \tau .
\end{align*}
Therefore, claim~\eqref{eq:claim-1} is established.

Note that assumption~\eqref{eq:assum} of Lemma~\ref{lem:per-sort} holds in view of Lemma~\ref{lem:f-prop} and the fact that $f(i, j) \ge 0$. Hence claim~\eqref{eq:claim-1} and Lemma~\ref{lem:per-sort} together yield that for all $i \in [\dimone]$,
\begin{align}
f(\hat \pi_1(i), i) = \sum_{k=1}^{\dimtwo} \big( \bar \theta_{\hat \pi_1 (i), k} - \bar \theta_{i, k}  \big)^2 \le 12 \tau .
\end{align}
Summing over $i$, we conclude that
\begin{align}
\| \bar \theta(\hat \pi_1, \id) - \bar \theta \|_F^2 \le 12 \tau \dimone,
\end{align}
contingent on the assumption that for the ground truth permutation \( \pi_1^\ast \), we have \( (\pi_1^\ast)^{-1}(i_0) < (\pi_1^\ast)^{-1}(j_0) \).
If not, repeat the same proof with \( \hat \pi_1 \) replaced by \( \pir \circ \hat \pi_1 \) to obtain
\begin{align}
\| \bar \theta(\pir \circ \hat \pi_1, \id) - \bar \theta \|_F^2 \le 12 \tau \dimone.
\end{align}

Finally, the proof remains valid if we replace $\bar \theta$ and $y$ by their transposes, and switch the roles of row and column indices. Hence it also holds with probability $1 - \dimone^{-10}$ that
\begin{align}
	\| \bar \theta(\id, \hat \pi_2) - \bar \theta \|_F^2 \wedge \| \bar \theta(\id, \pirtwo \circ \hat \pi_2) - \bar \theta \|_F^2
	\lesssim \big[ \sigma^2 + \sigma V(\bar \theta) \big] \dimtwo \sqrt{\dimone \log  \dimone }  .
\end{align}
We then complete the proof of~\eqref{eq:per-err} by using the triangle inequality to see that for some choice of \( \bar \pi_1 \in \{ \id, \pir \} \), \( \bar \pi_2 \in \{ \id, \pirtwo \} \), it holds that
\begin{align*}
\| \bar \theta(\bar \pi_1 \circ \hat \pi_1, \bar \pi_2 \circ \hat \pi_2) - \bar \theta \|_F^2 
&\le 2 \| \bar \theta(\bar \pi_1 \circ \hat \pi_1, \bar \pi_2 \circ \hat \pi_2) - \bar \theta(\bar \pi_1 \circ \hat \pi_1, \id) \|_F^2 + 2 \| \bar \theta(\bar \pi_1 \circ \hat \pi_1, \id) - \bar \theta \|_F^2 \\
&= 2 \| \bar \theta(\id, \bar \pi_2 \circ \hat \pi_2) - \bar \theta \|_F^2 + 2 \| \bar \theta(\bar \pi_1 \circ \hat \pi_1, \id) - \bar \theta \|_F^2 \\
&\lesssim \big[ \sigma^2 + \sigma V(\bar \theta) \big]  \dimone \sqrt{\dimtwo \log  \dimone } .
\end{align*}

\subsection{Proof of Lemma \ref{lem:xi-prop}} \label{sec:pf-xi}

Recall that $y = \bar \theta + \eps$ where $\Var[ \eps_{i, j} ] = \sigma^2$, and recall the notation
\begin{align}
\xi{(i, j)} = \sum_{k=1}^{\dimtwo} \Big[ y_{i, k} - y_{j, k} - \frac{1}{\dimtwo} \sum_{\ell=1}^{\dimtwo} (y_{i, \ell} - y_{j, \ell}) \Big]^2 .
\end{align}
For the statement about its expectation, \eqref{eq:xi-exp}, we need to prove that for distinct $i, j \in [\dimone]$,
\begin{align}
\E[ \xi ( i, j ) ] = \sum_{k=1}^{\dimtwo} ( \bar \theta_{i, k} - \bar \theta_{j, k}  )^2  + 2(\dimtwo - 1) \sigma^2 .
\label{eq:toprove1}
\end{align}
For the statement about its deviation, \eqref{eq:xi-dev}, we claim that it suffices to prove that with probability $1 -  \dimone^{-12}$,
\begin{align}
\big| \xi ( i, j ) - \E[ \xi ( i, j ) ] \big| 
\lesssim \sigma^2 \sqrt{\dimtwo \log  \dimone} + \sigma 
\Big[ \sum_{k=1}^{\dimtwo} ( \bar \theta_{i, k} - \bar \theta_{j, k}  )^2  \Big]^{1/2} 
\sqrt{ \log  \dimone }  . \label{eq:toprove}
\end{align}
To see this, note that Lemma~\ref{lem:cen-sq} gives 
\begin{align*}
\sum_{k=1}^{\dimtwo} ( \bar \theta_{i, k} - \bar \theta_{j, k}  )^2 
= \frac 1{\dimtwo} \sum_{k < \ell} ( \bar \theta_{i, k} - \bar \theta_{j, k}  - \bar \theta_{i, \ell} + \bar \theta_{j, \ell} )^2  
\le \sum_{k < \ell} \frac 1{\dimtwo} V_0^2 \le \dimtwo V(\bar \theta)^2 ,
\end{align*}
where we used that $| \bar \theta_{i, k} + \bar \theta_{j, \ell} - \bar \theta_{j, k}  - \bar \theta_{i, \ell} | \le V_0$. 
Plugging this bound into~\eqref{eq:toprove} and applying a union bound over all $i, j \in [\dimone]$ then completes the proof.

The claims~\eqref{eq:toprove1} and~\eqref{eq:toprove} can be simplified as follows. For distinct $i, j \in [\dimone]$, we let
$$
n = \dimtwo, \quad
x = y_{i, \cdot} - y_{j, \cdot} , \quad
\gamma = \bar \theta_{i, \cdot} - \bar \theta_{j, \cdot} , \quad
\delta = \eps_{i, \cdot} - \eps_{j, \cdot} ,  \quad
\text{and} \quad \zeta = \sum_{k=1}^n \Big( x_k - \frac 1n \sum_{\ell=1}^n x_\ell \Big)^2 .
$$
Note that $\delta$ has independent $\subG(C \sigma^2)$ entries and $\E[ \delta_k^2 ] = 2 \sigma^2$.
%Hence it holds that
%\begin{align}
%\E[ \nu ] = \frac 1n \sum_{k=1}^n \gamma_k  , \quad 
%\Var[ \nu ] = \frac{2 \sigma^2}{n}, 
%\end{align}
%and that with probability $1 - n^{-10}$, 
%\begin{align}
%\big| \nu - \E[ \nu ] \big| \lesssim \sigma \Big( \frac{ \log n }{ n } \Big)^{1/2} .
%\end{align}
We need to prove that
\begin{align}
\E[ \zeta ] 
= \sum_{k=1}^n \gamma_k^2 + 2(n - 1) \sigma^2 ,
\end{align}
and that with probability $1 - \dimone^{-12}$, 
\begin{align}
\big| \zeta - \E[ \zeta ] \big| \lesssim \sigma^2 \sqrt{n \log  \dimone} + \sigma \Big( \sum_{k=1}^{n} \gamma_k^2  \Big)^{1/2} \sqrt{\log  \dimone } .
\end{align}

%For the expectation, we compute 
%\begin{align}
%\E[ \zeta ] 
%& = \sum_{k=1}^n \E \big[ ( x_k - \nu )^2 \big] \\
%& = \sum_{k=1}^n \E \big[ \gamma_k^2 + \delta_k^2 + \nu^2 + 2 \gamma_k \delta_k - 2 \nu \gamma_k - 2 \nu \delta_k  \big] \\
%& = \sum_{k=1}^n \Big( \gamma_k^2 + 2 \sigma^2 + \frac{1}{n^2} \sum_{\ell, m=1}^n \gamma_\ell \gamma_m + \frac{2 \sigma^2}{n} - \frac{2 \gamma_k}{n} \sum_{\ell=1}^n \gamma_\ell - \frac{4 \sigma^2}{n} \Big) \\
%& = \sum_{k=1}^n \gamma_k^2 - \frac 1n \sum_{k, \ell = 1}^n \gamma_k \gamma_\ell + 2(n - 1) \sigma^2 .
%\end{align}

%Then we have
%\begin{align}
%\E [ \zeta ] & = \sum_{k=1}^n \E \Big[ ( x_k - \nu )^2 \Big] \\
%& = \sum_{k=1}^n \E \bigg[ \Big( 1 - \frac 1n \Big)^2 x_k^2 + \frac 1{n^2} \Big( \sum_{i \ne k} x_i \Big)^2 - \frac{2}{n} \Big( 1 - \frac 1n \Big) x_k \sum_{i \ne k} x_i \bigg] \\
%& = \Big( 1 - \frac 1n \Big)^2 \sum_{k=1}^n ( \gamma_k^2 + 2 \sigma^2 ) + \frac 1{n^2} \sum_{k=1}^n \Big[ \Big( \sum_{i \ne k} \gamma_i \Big)^2 + 2 (n-1) \sigma^2 \Big] - \frac{2}{n} \Big( 1 - \frac 1n \Big) \sum_{k=1}^n \Big( \gamma_k \sum_{i \ne k} \gamma_i \Big) \\
%& = \Big( 1 - \frac 1n \Big)^2 \sum_{k=1}^n \gamma_k^2 +  \frac{2 (n-1)^2}{n} \sigma^2  
%+ \frac 1{n^2} \sum_{k=1}^n \gamma_k^2 + \frac { 2 (n-1) }{n } \sigma^2 + \frac{2(n-1)}{n^2} \sum_{k=1}^n \gamma_k^2 \\
%& = \sum_{k=1}^n \gamma_k^2 + 2 (n-1) \sigma^2 .
%\end{align}

Recall that all the rows and columns of $\bar \theta$ are centered by assumption, so we have $\sum_{k=1}^n \gamma_k = 0$. Using this, we get the following Hoeffding decomposition
\begin{align}
\zeta &= \sum_{k=1}^n \Big( x_k - \frac 1n \sum_{i=1}^n x_i \Big)^2 
= \sum_{k=1}^n \gamma_k^2 + 
\sum_{k=1}^n \delta_k^2
- \frac{1}{n} \Big( \sum_{k = 1}^n \delta_k \Big)^2 + 
2 \sum_{k=1}^n \gamma_k \delta_k  .
\end{align}
%\begin{align}
%\gamma = \sum_{k=1}^n \gamma_k^2 + 
%\sum_{k=1}^n \delta_k^2
%- \frac{1}{n} \Big( \sum_{k = 1}^n \delta_k \Big)^2 + 
%2 \sum_{k=1}^n \gamma_k \delta_k  
%- \frac{2}{n} \Big( \sum_{k=1}^n  \gamma_k \Big) \Big( \sum_{k = 1}^n \delta_k \Big) .
%\end{align}
In particular, it follows that
\begin{align}
\E [ \zeta ] = \|\gamma\|_2^2 + 2 (n-1) \sigma^2 .
\end{align}
Moreover, we have that with probability $1 - \dimone^{-12}$,
\begin{align}
\big| \zeta - \E [ \zeta ] \big| &\le \Big| \sum_{k=1}^n \delta_k^2 - 2 n \sigma^2 \Big|
+ \Big| \frac{1}{n} \Big( \sum_{k = 1}^n \delta_k \Big)^2 - 2 \sigma^2 \Big| 
+ 2 \Big| \sum_{k=1}^n \gamma_k \delta_k  \Big| \\
&\lesssim \sigma^2 \sqrt{n \log \dimone} + \sigma \|\gamma\|_2 \sqrt{\log \dimone} ,
\end{align}
where concentration of the first two terms is due to Lemma~\ref{lem:subg-l2}, and the last term is $\subG(C \sigma^2 \|\gamma\|_2^2)$.
This completes the proof.

\subsection{Proof of Theorem \ref{thm:svt}}

This proof technique was developed by~\cite{Cha15, ShaBalGunWai17}, and our presentation follows that of Theorem~2 of~\cite{ShaBalGunWai17}. 
We assume without loss of generality that the underlying true permutations $\pi_1^*$ and $\pi_2^*$ are the identities throughout the proof. Let $\| \cdot \|$ denotes the operator norm of a matrix. It is well known (see Theorem 4.4.5 of~\cite{Ver18}) that $\|\eps\| \le C_1 \sigma \sqrt{\dimone}$ with probability at least $1 - \exp(-\dimone)$. We condition on this event in the sequel. 

%\subsubsection{Proof the upper bound.}
Recall that the singular value decomposition of $y$ is
$$
y = \sum_{i=1}^{\dimtwo} \lambda_i u_i v_i^\top ,
$$
where $\lambda_1, \dots, \lambda_{\dimtwo}$ are ordered non-increasingly, and the SVT estimator is defined as
\begin{align}
\thetasvt \defn \sum_{i=1}^{\dimtwo} \1 \{ \lambda_i > \rho \}  \lambda_i u_i v_i^\top ,
\end{align}
where we choose $\rho = 2 C_1 \sigma \sqrt{\dimone}$.
Moreover, write the singular value decomposition of $\theta^*$ as
$$
\theta^* = \sum_{i=1}^{\dimtwo} \lambda^*_i u^*_i (v^*_i)^\top ,
$$
where $\lambda^*_1, \dots, \lambda^*_{\dimtwo}$ are ordered non-increasingly. 
Let $s$ be the number of singular values of $\theta^*$ that are larger than $C_1 \sigma \sqrt{ \dimone }$, and define 
$$
\theta_s = \sum_{i=1}^{s} \lambda^*_i u^*_i (v^*_i)^\top .
$$

Note that the for each $i > s$, by Weyl’s inequality, we have 
$$
\lambda_i \le \lambda^*_i + \|\eps\| \le C_1 \sigma \sqrt{ \dimone } + C_1 \sigma \sqrt{ \dimone } = \rho .
$$
Therefore, $\thetasvt$ has rank at most $s$, and so has $\theta_s$. It follows that
\begin{align*}
\| \thetasvt - \theta^* \|_F \le \| \thetasvt - \theta_s \|_F + \| \theta_s - \theta^* \|_F 
\le \sqrt{ 2s } \, \| \thetasvt - \theta_s \| + \Big[ \sum_{i=s+1}^{\dimtwo} (\lambda^*_i)^2 \Big]^{1/2} .
\end{align*}
Moreover, it holds that
$$
\| \thetasvt - \theta_s \| \le \| \thetasvt - y \| + \| \eps \| + \| \theta^* - \theta_s \| \le 4 C_1 \sigma \sqrt{ \dimone } . 
$$
Plugging this bound into the previous one, we obtain
\begin{align}
\| \thetasvt - \theta^* \|_F^2 \lesssim s \sigma^2 \dimone + \sum_{i=s+1}^{\dimtwo} (\lambda^*_i)^2  \lesssim \sum_{i=1}^{\dimtwo} \big[ \sigma^2 \dimone \land (\lambda^*_i)^2 \big] .  \label{eq:sum-min}
\end{align}

For any integer $r \ge 6$, Proposition~\ref{prop:low-rank-approx} yields a rank-$r$ matrix $\tilde{\theta} \in \R^{\dimone \times \dimtwo}$ such that 
$$
\|\tilde{\theta} - \theta^* \|_F^2 \lesssim \frac{\dimone \dimtwo}{r^3} V(\theta^*)^2 .
$$
Since $\theta_r = \sum_{i=1}^{r} \lambda^*_i u^*_i (v^*_i)^\top$ is by definition the best rank-$r$ approximation of $\theta^*$ in the Frobenius norm, we see that
$$
\sum_{i=r+1}^{\dimtwo} (\lambda^*_i)^2 = \| \theta_r - \theta^* \|_F^2 \lesssim \frac{\dimone \dimtwo}{r^3} V(\theta^*)^2 .
$$
Hence it follows from~\eqref{eq:sum-min} that
$$
\| \thetasvt - \theta^* \|_F^2 \lesssim r \sigma^2 \dimone +  \frac{\dimone \dimtwo}{r^3} V(\theta^*)^2 .
$$
Choosing the optimal $r^\ast$ and considering the boundary cases \( r^\ast < 1 \) and \( r^\ast > \dimtwo \) then yields 
$$
\frac{1}{\dimone \dimtwo} \| \thetasvt - \theta^* \|_F^2 \lesssim \left[ \frac{\sigma^2}{\dimtwo} + \frac{ \sigma^{3/2} V(\theta^*)^{1/2} }{ \dimtwo^{3/4} } \right] \wedge \sigma^2.
$$
By repeating the same proof keeping track of the failure probability of the statement, we can obtain bounds in expectation as well.

\subsection{Proof of Proposition~\ref{prop:low-rank-approx}}

By rescaling, we may assume that $V( \theta ) = 1$ without loss of generality. Lemma~\ref{lem:c} yields that $\theta = R + S + B$, where $R$ and $S$ are rank-one matrices, and $B$ is anti-Monge, bivariate isotonic (i.e., $B$ has nondecreasing rows and columns).
Additionally, we have $B_{i, 1} = B_{1, j} = 0$ for $i \in [\dimone], j \in [\dimtwo]$, and $B_{\dimone, \dimtwo} = 1$. It suffices to find a low-rank approximation of the matrix $B$.

\subsubsection{Subdivision.}

We claim that there exist two increasing sequences of indices $\{ i_k \}_{k=1}^{r + 1}$ and $\{ j_\ell \}_{\ell=1}^{2 r}$ such that 
\begin{itemize}
%\item
%$0 = i_0 < i_1 < \cdots < i_{r + 1} = \dimone$, and $0 = j_0 < j_1 < \cdots < j_{2 r} = \dimtwo$;
%
\item
$0 \le i_k - i_{k-1} \le \dimone / r$ for $k \in [r + 1]$;

\item
$0 \le j_\ell - j_{\ell-1} \le \dimtwo / r$ and $B_{\dimone, j_\ell} - B_{\dimone, j_{\ell-1} + 1} \le 1/r$ for $\ell \in [2r]$.
\end{itemize}
For $\{i_k\}_{k=1}^{r+1}$, it suffices to choose $i_0 = 0$, $i_k = i_{k-1} + \lfloor \dimone / r \rfloor$ for $k \in [r]$ and $i_{r+1} = \dimone$. For $\{ j_\ell \}_{\ell=1}^{2 r}$, since $B_{\dimone, 1} = 0$, $B_{\dimone, \dimtwo} = 1$ and $B$ has nondecreasing rows, there is an increasing sequence of indices $\{j_\ell'\}_{\ell = 1}^r$ such that $B_{\dimone, j_\ell'} - B_{\dimone, j_{\ell-1}' + 1} \le 1/r$ for all $j \in [r]$. Moreover, by inserting (at most) another $r$ indices between the indices $j_\ell'$ to obtain a new sequence $\{j_\ell\}_{\ell = 1}^{2r}$, we can guarantee that not only $B_{\dimone, j_\ell} - B_{\dimone, j_{\ell-1} + 1} \le 1/r$, but also $j_{\ell} - j_{\ell - 1} \le \dimtwo / r$.

\subsubsection{Low-rank approximation.}

Let $\{ i_k \}_{k=1}^{r + 1}$ and $\{ j_\ell \}_{\ell=1}^{2 r}$ be chosen so that the above conditions are satisfied. We define a matrix $X \in \R^{\dimone \times \dimtwo}$ by setting $X_{i, j} = B_{i, j_{\ell-1}+1}$ for all $i \in [\dimone]$ and $j_{\ell-1} < j \le j_\ell$ where $\ell \in [2r]$. By definition, all columns of $X$ with indices in $(j_{\ell-1}, j_\ell]$ are the same, so $X$ has rank at most $2r$.

Furthermore, we define a matrix $Y \in \R^{\dimone \times \dimtwo}$ by setting $Y_{i, j} = (B-X)_{i_{k-1}+1, j}$ for all $j \in [\dimtwo]$ and $i_{k-1} < i \le i_k$ where $k \in [r+1]$. Similarly, all rows of $Y$ with indices in $(i_{k-1}, i_k]$ are the same, so $Y$ has rank at most $r+1$.

It remains to bound $\|\Delta \|_F^2$ where $\Delta = X+Y - B$. First, let us focus on a block with double indices in $(i_{k-1}, i_k] \times (j_{\ell - 1}, j_\ell]$. \emph{On each of these blocks}, by definition it holds that:
\begin{itemize}
\item
$X$ has constant rows, and $Y$ has constant columns;

\item
the first column of $Y$ is zero, and the first column and first row of $\Delta$ is zero.
\end{itemize}
Then by Lemma~\ref{lem:c} (applied with the corresponding blocks of $(B, X, Y, \Delta)$ in place of $(\theta, R, S, B)$), we see that $\Delta$ is bivariate isotonic on each of these block, and 
$$
\Delta_{i_k, j_\ell} 
= V \big( B_{i_{k-1}+1 : i_k,  j_{\ell - 1} + 1: j_\ell } \big) 
= B_{i_{k-1}+1 ,  j_{\ell - 1} + 1 } + B_{ i_k,  j_\ell } - B_{i_{k-1}+1 ,  j_\ell } - B_{ i_k,  j_{\ell - 1} + 1 } .
$$
Therefore, it follows that
\begin{align*}
\sum_{i=i_{k-1} + 1}^{i_k} \sum_{j=j_{\ell-1} + 1}^{j_\ell} \Delta_{i, j}^2 
&\le (i_k - i_{k-1}) ( j_{\ell} - j_{\ell - 1}) \Delta_{i_k, j_\ell}^2 \\
&\le \frac{\dimone \dimtwo}{r^2} \Delta_{i_k, j_\ell}^2 
= \frac{\dimone \dimtwo}{r^2} V \big( B_{i_{k-1}+1 : i_k,  j_{\ell - 1} + 1: j_\ell } \big)^2 ,
\end{align*}
where the second inequality holds thanks to the above choice of indices. 

Summing over all the blocks, we obtain
$$
\| \Delta \|_F^2
%\sum_{i=1}^{\dimone} \sum_{j=1}^{\dimtwo} \Delta_{i, j}^2 
= \sum_{k=1}^{r+1} \sum_{\ell=1}^{2r} \sum_{i=i_{k-1} + 1}^{i_k} \sum_{j=j_{\ell-1} + 1}^{j_\ell} \Delta_{i, j}^2 
\le \frac{\dimone \dimtwo}{r^2} \sum_{k=1}^{r+1} \sum_{\ell=1}^{2r} V \big( B_{i_{k-1}+1 : i_k,  j_{\ell - 1} + 1: j_\ell } \big)^2 .
$$

\subsubsection{Bounding the sum of variations.} 
It remains to bound the above sum. Recall that a telescoping sum gives
$$
\sum_{k=1}^{r+1} V \big( B_{i_{k-1}+1 : i_k,  j_{\ell - 1} + 1: j_\ell } \big)
= V \big( B_{1 : \dimone,  j_{\ell - 1} + 1: j_\ell } \big) 
\stackrel{(i)}{=} B_{ \dimone,  j_\ell } - B_{ \dimone,  j_{\ell - 1} + 1 } 
\stackrel{(ii)}{\le} 1/r ,
$$
where $(i)$ holds because the first row and first column of $B$ are zero, and $(ii)$ holds because of our choice of $\{j_\ell\}_{\ell = 1}^{2r}$. As a result, it holds by H\"older's inequality that
$$
\sum_{\ell=1}^{2r} \sum_{k=1}^{r+1} V \big( B_{i_{k-1}+1 : i_k,  j_{\ell - 1} + 1: j_\ell } \big)^2 
\le \sum_{\ell=1}^{2r} V \big( B_{1 : \dimone,  j_{\ell - 1} + 1: j_\ell } \big)^2 
\le 2r (1/r)^2 = 2/r .
$$
We therefore obtain
$
\|\Delta\|_F^2 \le \frac{2 \dimone \dimtwo}{r^3}.
$
The proof is complete since $\Delta = X + Y - B = R + S + X + Y - \theta$, where the matrix $R + S + X + Y$ has rank at most $3r + 3$.

%Let us define a partition $[\dimtwo] = J_1 \sqcup \cdots \sqcup J_r$ by setting $j \in J_k$ if $(k-1)/r \le B_{\dimone, j} < k/r$ and setting $j \in J_r$ if $B_{\dimone, j} = 1$. Furthermore, we define a matrix $\tilde{B} \in \R^{\dimone \times \dimtwo}$ by setting $\tilde{B}_{i, j} = B_{i, j_k}$ for $i \in [\dimone], j \in J_k$, where $j_k = \min J_k$. It is clear that $\tilde{B}$ is also supermodular and bivariate isotonic. More importantly, $\tilde{B}$ is rank-$r$ by definition. If we define $\tilde{\theta} = R + S + \tilde{B}$, then $\tilde{\theta}$ is supermodular and rank-$(r+2)$.  
%
%It remains to bound $\|\tilde{\theta} - \theta\|_F^2 = \| \tilde{B} - B \|_F^2$. Note that by the supermodularity and bivariate isotonicity of $B$, for any $i \in [\dimone], j \in J_k$, it holds that
%$$
%B_{\dimone, j} - B_{\dimone, j_k} \ge B_{i, j} - B_{i, j_k} \ge 0 .
%$$
%Using this inequality and that $B_{\dimone, j} - B_{\dimone, j_k}  \le 1/r$ for $j \in J_k$, we obtain
%\begin{align}
%\| \tilde{B} - B \|_F^2 &=  \sum_{i=1}^{\dimone} \sum_{k=1}^r \sum_{j \in J_k} \big( \tilde{B}_{i, j} - B_{i, j} \big)^2 
%= \sum_{i=1}^{\dimone} \sum_{k=1}^r \sum_{j \in J_k} \big( B_{i, j_k} - B_{i, j} \big)^2 \\
%&\le \dimone \sum_{k=1}^r \sum_{j \in J_k} \big( B_{\dimone, j_k} - B_{\dimone, j} \big)^2 
%\le \dimone \sum_{k=1}^r \sum_{j \in J_k} 1/r^2 = \dimone \dimtwo / r^2.
%\end{align}

\section*{Acknowledgment}

We thank Adityanand Guntuboyina for discussing their concurrent work with us. PR was supported by NSF awards IIS-1838071, DMS-1712596 and DMS-TRIPODS-1740751; ONR grant N00014-17- 1-2147 and grant 2018-182642 from the Chan Zuckerberg Initiative DAF. ER was supported by an NSF MSPRF DMS-1703821.

\appendix

\section{Existing results}

In this appendix, we state some existing results that are used in our proofs.

%The following standard result on the least-squares estimator was discovered by Birg\'e~\cite{Bir83,Bir86}, and we use the formulation by Chatterjee~\cite{Cha14}.
%
%\begin{theorem}[Birg\'e; Chatterjee]
%	\label{thm:chatterjee}
%	For a closed, convex set $\cM \subseteq \R^d$, suppose we observe $y = \theta^* + \eps$ where $\theta \in \cM$ and $\eps \sim \subG(\sigma^2)$. Consider the function
%	\begin{align}
%		\label{eq:o}
%		f_{\theta^*}(t)
%		\defn \E \Big[ \sup_{\substack{\theta \in \cM\\ \|\theta - \theta^* \|_2 \leq t}} \langle \varepsilon , \theta - \theta^* \rangle \Big] - \frac{t^2}{2} ,
%	\end{align}
%	and its mode
%	\begin{align}
%		\label{eq:p}
%		t_{\theta^*} \defn \argmax_{t \in [0,\infty)} f_{\theta^*}(t).
%	\end{align}
%	Then it holds that
%	\begin{align}
%		\label{eq:q}
%		\E \Big[ \big\| \Pi_{\cM} (y) - \theta^* \big\|_2^2 \Big] \lesssim t_{\theta^*}^2 \lor \sigma^2.
%	\end{align}
%	Moreover, if there is \( t_{\theta^*}' > 0 \) such that \( f_{\theta^*}(t_{\theta^*}') < 0 \), then
%	\begin{align}
%		\label{eq:r}
%		\E \Big[ \big\| \Pi_{\cM} (y) - \theta^* \big\|_2^2 \Big] \lesssim (t'_{\theta^*})^2 \lor \sigma^2.
%	\end{align}
%\end{theorem}

The following result is Talagrand's majorizing measure theorem~\cite{Tal14}. See also Theorem~8.6.1 of~\cite{Ver18}.

\begin{theorem}[Talagrand's majorizing measure theorem] \label{thm:maj}
For any set $\cM \subset \R^d$ and a Gaussian random vector $\eps \sim \cN(0, I_d)$, we have
$$
\E \sup_{\theta \in \cM} \langle \eps, \theta \rangle \asymp \gamma_2 ( \cM ) ,
$$
where $\gamma_2(\cdot)$ denotes Talagrand's $\gamma_2$ functional. 
\end{theorem}

The following theorem gives a tail bound on the supremum of a sub-Gaussian process~\cite{Tal14, Ver18}.

\begin{theorem}[Generic chaining tail bound] \label{thm:tail}
Consider a sub-Gaussian vector $\eps \sim \subG_d ( \sigma^2 )$. For any set $\cM \subset \R^d$ and $s>0$, it holds with probability at least $1 - 2 \exp( - s^2 )$ that
$$
\sup_{\theta \in \cM} \langle \eps, \theta \rangle \lesssim \sigma \big[ \gamma_2 ( \cM) + s \cdot \sup_{\theta \in \cM} \|\theta\|_2 \big] .
$$
\end{theorem}

\begin{proof}
By Theorem~8.5.5 of~\cite{Ver18}, we have that for any $\theta^* \in \cM$, 
$$
\sup_{\theta \in \cM} | \langle \eps, \theta - \theta^* \rangle | \lesssim \sigma \big[ \gamma_2 ( \cM) + s \cdot \diam ( \cM ) \big] 
$$
with probability at least $1 - 2 \exp(-s^2)$, where $\diam(\cdot)$ denotes the diameter of a set.
Moreover, we have with probability $1 - 2 \exp(-s^2)$ that
$$
| \langle \eps, \theta^* \rangle | \lesssim \sigma s \|\theta^*\|_2 .
$$
It follows that with probability at least $1 - 4 \exp(-s^2)$, 
$$
\sup_{\theta \in \cM} \langle \eps, \theta \rangle \le \sup_{\theta \in \cM} | \langle \eps, \theta - \theta^* \rangle | + | \langle \eps, \theta^* \rangle | 
\lesssim \sigma \big[ \gamma_2 ( \cM) + s \cdot \diam ( \cM ) + s \|\theta^*\|_2 \big] .
$$
Since $\diam ( \cM ) + \|\theta^*\|_2 \lesssim \sup_{\theta \in \cM} \|\theta \|_2$, the proof is complete.
\end{proof}

Assouad's lemma is used to prove the lower bounds (see~\cite[Lemma~24.3]{Vaa98}).

\begin{theorem}[Assouad's Lemma]\label{thm:assouad}
Consider a parameter space $\cM$. Let $\p_\theta$ denote the distribution of the observation given that the true parameter is $\theta \in \cM$. Let $\E_\theta$ denote the corresponding expectation.
%Suppose that we observe $y = \theta^* + \epsilon$, where $\theta^* \in \cM$ and $\epsilon_i \widesimiid \cN(0, \sigma^2)$.
%%$(\epsilon_i)$ are independent normally distributed random variables with mean zero and variance $\sigma^2$. 
%Let $\mathbb P_\theta$ for $g\in\mathcal M$ denote the distribution of the observation $y$ for $\theta^* = \theta$.
Suppose that for each $\tau\in\{-1, 1\}^d$, there is an associated $\theta^\tau \in \cM$. 
Then it holds that
$$\underset{\tilde\theta}{\inf}\underset{\theta^* \in\mathcal M}{\sup} \E_{\theta^\ast} \ell^2(\tilde\theta, \theta^*) \geq \frac d8 \, \underset{\tau\neq\tau'}{\min}\frac{\ell^2(\theta^{\tau}, \theta^{\tau'})}{d_H(\tau, \tau')} \, \underset{d_H(\tau, \tau') = 1}{\min}(1 - \|\mathbb P_{\theta^\tau} - \mathbb P_{\theta^{\tau'}}\|_{TV}),$$
where $\ell$ denotes any distance function on $\cM$, $d_H$ denotes the Hamming distance, $\|\cdot\|_{TV}$ denotes the total variation distance, and the infimum is taken over all estimators $\tilde \theta$ measurable with respect to the observation $y$. 
\end{theorem}

The following lemma can be proven with basic concentration theory, or follows as a special instance from the Hanson-Wright inequality~\cite{HanWri71, RudVer13}.

\begin{lemma} \label{lem:subg-l2}
Suppose that $\eps \in \R^{n}$ is a random vector with independent centered $\subG(\sigma^2)$ entries. Then it holds that for all $t \ge 0$,
$$
\p \Big\{ \Big| \sum_{i=1}^n \big( \eps_i^2 - \E[ \eps_i^2 ] \big) \Big| \ge t \Big\} \le 2 \exp \Big[ - c \min \Big( \frac{t^2}{ \sigma^4 n} , \frac{t}{ \sigma^2 } \Big) \Big] .
$$
\end{lemma}

\section{Lower bounds for the SVT estimator} \label{sec:svt-lower}

Let us focus on the special case $\dimone = \dimtwo = n$, and study a worst-case matrix in $\cM = \cM^{n, n} $, for which the approximation rate given by Proposition~\ref{prop:low-rank-approx} is tight. 

\begin{lemma} \label{lem:approx-lower}
There exists a matrix $\theta \in \cM^{n, n}$ such that 
$$
\min_{\theta_r \text{ has rank } r} \| \theta_r - \theta \|_F^2 \gtrsim \frac{n^2}{r^3} V(\theta)^2 .
$$
\end{lemma}

\begin{proof}
Consider the matrix
$$
\theta \defn \frac{V_0}{n} D^\dag (D^\dag)^\top 
$$
for $V_0 > 0$. It is anti-Monge because
$$
D \theta D^\top = \frac{V_0}{n} I \ge 0 .
$$
It also follows from~\eqref{eq:def-v} that
$$
V( \theta ) = \| D \theta D^\top \|_1 = \frac{V_0}{n} \| I \|_1 = V_0 .
$$

Moreover, the singular values of $D$ are given in~\eqref{eq:cv}, so the eigenvalues of $\theta$ are 
$$
\mu_i = \frac {V_0}{4 n} \Big( \sin \frac{\pi i}{2 n} \Big)^{-2} .
$$
Using the fact that $x/2 \le \sin x \le x$ for $ x \in [0, \pi/2]$, we obtain
\begin{align}
\frac{ V_0 n }{ \pi^2 i^2 }  \le \mu_i \le \frac{ 4 V_0 n }{ \pi^2 i^2 } . 
\label{eq:mu-bd}
\end{align}
If the spectral decomposition of $\theta$ is $\theta = \sum_{i=1}^n \mu_i w_i w_i^\top$,
then the best rank-$r$ approximation of $\theta$ in the Frobenius norm is $\theta_r = \sum_{i=1}^r \mu_i w_i w_i^\top$, and
$$
\| \theta_r - \theta \|_F^2 = \sum_{i=r+1}^n \mu_i^2 \ge \sum_{i=r+1}^n \frac{ V_0^2 n^2 }{ \pi^4 i^4 } \gtrsim \frac{ V_0^2 n^2}{ r^3 }  ,
$$
which completes the proof. 
\end{proof}

Since the anti-Monge matrix $\theta$ in the above proof cannot be approximated by a low-rank matrix at a better rate, we conjecture that for this choice of $\theta$, the rate of convergence given by Theorem~\ref{thm:svt} for the SVT estimator~\eqref{eq:svt} is tight. 
Intuitively, if we set threshold $\rho$ in definition~\eqref{eq:svt} to be larger, then the resulting estimator has a lower rank, thus incurring a larger bias according to the above lemma. On the other hand, setting threshold $\rho$ to be smaller incurs a larger variance due to the noise $\eps$. 

More precisely, under the model $y = \theta + \eps$ where $\eps$ has i.i.d. $\cN(0, \sigma^2)$ entries, we conjecture that for any choice of threshold $\rho$ in the estimator~\eqref{eq:svt}, it holds with constant probability that 
$$
\frac{1}{n^2} \| \thetasvt - \theta \|_F^2 \gtrsim \frac{\sigma^2}{n} + \frac{ \sigma^{3/2} V(\theta)^{1/2} }{ n^{3/4} } .
$$
This is because we believe that the bias-variance trade-off in the proof of Theorem~\ref{thm:svt} is optimal. However, we are unable to prove a lower bound based on a similar argument in \cite{ShaBalGunWai17}, since there, the authors are able to exploit a varying signal-to-noise ratio within the classes of matrices they consider.
This allows them to employ a triangle inequality argument instead of a explicit bias-variance decomposition that holds with equality. Potential other approaches include analyzing an explicit unbiased estimate of the risk for SVT~\cite{CanSinTrz13}, and studying the exact asymptotic optimal choice of threshold $\rho$ as in~\cite{GavDon14}. However, since any of these approaches require asymptotic random matrix theory, we consider them beyond the scope of the current work.

\bibliographystyle{abbrv}
\bibliography{MTP2}

\end{document}